\documentclass[12,ptoneside]{amsart}
\usepackage{amsmath,amstext,amsthm,amssymb,bbm,amsbsy}
\usepackage{mathtools}
\usepackage[numbers]{natbib}
\usepackage{enumerate}
\usepackage{bm}
\usepackage{comment}
\usepackage[ruled,linesnumbered]{algorithm2e}
\usepackage{multirow}
\usepackage{subfig}

\usepackage{tabularx}
\usepackage{makecell}
\usepackage{booktabs}
\usepackage{xcolor}

\usepackage{appendix}

\usepackage[letterpaper, margin=1in]{geometry} % for wide text

\usepackage[backref=page]{hyperref}
\renewcommand*{\backref}[1]{}
\renewcommand*{\backrefalt}[4]{%
    \ifcase #1 (Not cited.)%
    \or        (Cited on page~#2.)%
    \else      (Cited on pages~#2.)%
    \fi}
\newcommand{\ben}{\begin{enumerate}}
\newcommand{\een}{\end{enumerate}}
\newcommand{\bea}{\begin{eqnarray}}
\newcommand{\eea}{\end{eqnarray}}
\newcommand{\bit}{\begin{itemize}}
\newcommand{\eit}{\end{itemize}}
\newcommand{\be}{\begin{equation}}
\newcommand{\ee}{\end{equation}}
\newcommand{\benn}{\begin{equation*}}
\newcommand{\eenn}{\end{equation*}}

\newtheorem{prop}{Proposition}[section]
\newtheorem{lem}{Lemma}[section]

\newtheorem{thm}{Theorem}[section]
\newtheorem{defn}{Definition}
\newtheorem{remark}{Remark}[section]
\newtheorem{example}{Example}[section]
\DeclareMathOperator*{\argmin}{arg\,min}
\DeclareMathOperator*{\argmax}{arg\,max}

\def\bp {\boldsymbol{p}}
\def\bq {\boldsymbol{q}}
\def\bx {\boldsymbol{x}}
\def\by {\boldsymbol{y}}
\def\bv {\boldsymbol{v}}
\def\bu {\boldsymbol{u}}
\def\bW {\boldsymbol{W}}
\def\bolde {\boldsymbol{e}}
\def\R {\mathbb{R}}
\def\N {\mathbb{N}}
\def\dom {\mathrm{dom~}}
\def\inter {\mathrm{int~}}
\def\epi {\mathrm{epi~}}
\def\ri {\mathrm{ri~}}
\def\conv {\mathrm{conv}~}
\def\cobar {\overline{\mathrm{co}}~}
\def\co {\mathrm{co~}}

\def\unitsim {\Lambda}
\def\Rn {\R^n}
\def\gmRn {\Gamma_0(\R^n)}
\def\Laplacian{\Delta_{\boldsymbol x}}

\date{Research supported by NSF DMS 1820821. Authors' names are given in last/family name alphabetical order.}

\newcommand{\revision}[1]{\textcolor{black}{#1}} 

\begin{document}

\title[]{Overcoming the curse of dimensionality for some Hamilton--Jacobi partial differential equations via neural network architectures}

\author[J. Darbon]{J\'er\^ome Darbon}
\address{Division of Applied Mathematics, Brown University.}
\email{jerome\_darbon@brown.edu}

\author[G. P. Langlois]{Gabriel P. Langlois}
\address{Division of Applied Mathematics, Brown University.}
\email{gabriel\_provencher\_langlois@brown.edu}

\author[T. Meng]{Tingwei Meng}
\address{Division of Applied Mathematics, Brown University.}
\email{tingwei\_meng@brown.edu}

%\author{  J\'er\^ome Darbon       \and Gabriel P. Langlois
 %                \and Tingwei Meng
%}


\begin{comment}
 \institute{ F. Author \at
              first address \\
              Tel.: +123-45-678910\\
              Fax: +123-45-678910\\
              \email{}           %  \\
%             \emph{Present address:}
           \and
           S. Author \at
              second address
}

\date{Received: date / Accepted: date}
\end{comment}

\maketitle

% , \revision{thereby paving the way to leverage efficient dedicated hardware implementation of neural networks to evaluate viscosity solutions of certain HJ PDEs.}

\begin{abstract}
We propose new and original mathematical connections between Hamilton-Jacobi (HJ) partial differential equations (PDEs) with initial data and neural network architectures. Specifically, we prove that some classes of neural networks correspond to representation formulas of HJ PDE solutions whose Hamiltonians and initial data are obtained from the parameters of the neural networks. These results do not rely on universal approximation properties of neural networks; rather, our results show that some classes of neural network architectures naturally encode the physics contained in some HJ PDEs. Our results naturally yield efficient neural network-based methods for evaluating solutions of some HJ PDEs in high dimension without using grids or numerical approximations. We also present some numerical results for solving some inverse problems involving HJ PDEs using our proposed architectures.
\end{abstract}

\section{Introduction}
The Hamilton--Jacobi (HJ) equations are an important class of partial differential equation (PDE) models that arise in many scientific disciplines, e.g., physics \cite{Arnold1989Math, Caratheodory1965CalculusI, Caratheodory1967CalculusII, Courant1989Methods,landau1978course}, imaging science \cite{darbon2015convex,darbon2019decomposition,Darbon2016Algorithms}, game theory \cite{BARRON1984213, Buckdahn2011Recent, Evans1984Differential, Ishii1988Representation}, and optimal control \cite{Bardi1997Optimal, Elliott1987Viscosity,fleming1976deterministic,fleming2006controlled,mceneaney2006max}. Exact or approximate solutions to these equations then give practical insight about the models in consideration. We consider here HJ PDEs specified by a Hamiltonian function $H\colon\R^n\to \R$ and \revision{convex} initial data $J\colon\R^n \to \R$
\revision{
\begin{equation} \label{eqn:intro_example}
\begin{dcases} 
\frac{\partial S}{\partial t}(\bx,t)+H(\nabla_{\bx}S(\bx,t)) = 0 & \mbox{{\rm in} }\mathbb{R}^{n}\times(0,+\infty),\\
S(\bx,0)=J(\bx) & \mbox{{\rm in} }\mathbb{R}^{n},
\end{dcases}
\end{equation}
}
where $\frac{\partial S}{\partial t}(\bx,t)$ and $\nabla_{\bx}S(\bx,t) =\left(\frac{\partial S}{\partial x_1} (\bx, t), \dots, \frac{\partial S}{\partial x_n} (\bx, t)\right)$ denote the partial derivative with respect to $t$ and the gradient vector with respect to $\bx$ of the function $(\bx,t) \mapsto S(\bx,t)$, \revision{and the Hamiltonian $H$ only depends on the gradient $\nabla_{\bx}S(\bx,t)$.}

\revision{
Our main motivation is to compute the viscosity solution of certain HJ PDEs of the form of \eqref{eqn:intro_example} in high dimension for a given $\bx \in \R^n$ and $t > 0$ \cite{Bardi1997Optimal, bardi1984hopf, barles1994solutions, crandall1992user} by leveraging new efficient hardware technologies and silicon-based electric circuits dedicated to neural networks. As noted by LeCun in~\cite{lecun2019isscc}, the use of neural networks has been greatly influenced by available hardware. In addition, there has been many initiatives to create new hardware for neural networks that yields extremely efficient (in terms of speed, latency, throughput or energy) implementations: for instance, \cite{farabet-suml-11,farabet-fpl-09,farabet.09.iccvw} propose efficient neural network implementations using field-programmable gate array, \cite{banerjeeEtal2019sfi} optimizes neural network implementations for Intel's architecture and \cite{kundu2019ktanh} provides efficient hardware implementation of certain building blocks widely used in neural networks. It is also worth mentioning that Google created specific hardware, called ``Tensor Processor Unit" \cite{googleTPU17} to implement their neural networks in data centers. Note that Xilinx announced a new set of hardware (Versal AI core) for implementing neural networks while Intel enhances their processors with specific hardware instructions for neural networks. LeCun also suggests in~\cite[Section 3]{lecun2019isscc} possible new trends for hardware dedicated to neural networks. Finally, we refer the reader to \cite{chen2020classification} (see also \cite{Hirjibehedin.20.nature}) that describes the evolution of silicon-based electrical circuits for machine learning.}

\revision{
In this paper, we propose classes of neural network architectures that exactly represent viscosity solutions of certain HJ PDEs of the form of~\eqref{eqn:intro_example}. Our results pave the way to leverage efficient dedicated hardware implementation of neural networks to evaluate viscosity solutions of certain HJ PDEs. 
} % end revision

\bigbreak
\noindent
\revision{\textbf{Related work.}} The viscosity solution to the HJ PDE~\eqref{eqn:intro_example} rarely admits a closed-form expression, and in general it must be computed with numerical algorithms or other methods tailored for the Hamiltonian $H$, initial data $J$, and dimension $n$.

The dimensionality, in particular, matters significantly because in many applications involving HJ PDE models, the dimension $n$ is extremely large. In imaging problems, for example, the vector $\bx$ typically corresponds to a noisy image whose entries are its pixel values, and the associated Hamilton--Jacobi equations describe the solution to an image denoising convex optimization problem \cite{darbon2015convex, darbon2019decomposition}. Denoising a 1080 x 1920 standard full HD image on a smartphone, for example, corresponds to solving a HJ PDE in dimension $n = 1080 \times 1920 = 2,073,600$. 

Unfortunately, standard grid-based numerical algorithms for PDEs are impractical when $n > 4$. Such algorithms employ grids to discretize the spatial and time domain, and the number of grid points required to evaluate accurately solutions of PDEs grows exponentially with the dimension $n$. It is therefore essentially impossible in practice to numerically solve PDEs in high dimension using grid-based algorithms, even with sophisticated high-order accuracy methods for HJ PDEs such as ENO \cite{Osher1991High}, WENO \cite{Jiang2000Weighted}, and DG \cite{Hu1999Discontinuous}. This problem is known as the \textit{curse of dimensionality} \cite{bellman1961adaptive}.

Overcoming the curse of dimensionality in general remains an open problem, but for HJ PDEs several methods have been proposed to solve it. These include, but are not limited to, max-plus algebra methods \cite{akian2006max,akian2008max, dower2015max,Fleming2000Max,gaubert2011curse,mceneaney2006max,McEneaney2007COD,mceneaney2008curse,mceneaney2009convergence}, dynamic programming and reinforcement learning \cite{alla2019efficient,bertsekas2019reinforcement}, tensor decomposition techniques \cite{dolgov2019tensor,horowitz2014linear,todorov2009efficient}, sparse grids \cite{bokanowski2013adaptive,garcke2017suboptimal,kang2017mitigating}, model order reduction \cite{alla2017error,kunisch2004hjb}, polynomial approximation \cite{kalise2019robust,kalise2018polynomial}, optimization methods \cite{darbon2015convex,darbon2019decomposition,Darbon2016Algorithms,yegorov2017perspectives} and neural networks \revision{\cite{bachouch2018deep, Djeridane2006Neural,jiang2016using, Han2018Solving, hure2018deep, hure2019some, lambrianides2019new, Niarchos2006Neural, reisinger2019rectified,royo2016recursive, ruthotto2019machine, Sirignano2018DGM}}. 
Among these methods, neural networks have become increasingly popular tools to solve PDEs \revision{\cite{bachouch2018deep,beck2018solving, beck2019deep, beck2019machine, Berg2018Unified,chan2019machine, Cheng2006Fixed, Djeridane2006Neural, Dissanayake1994Neural,  dockhorn2019discussion, E2017Deep, Farimani2017Deep, Fujii2019Asymptotic, grohs2019deep, grune2020overcoming, Han2018Solving, han2019solving, hsieh2018learning, hure2018deep, hure2019some, hutzenthaler2019proof, jianyu2003numerical, khoo2017solving, khoo2019solving, Lagaris1998ANN, Lagaris2000NN, lambrianides2019new, lee1990neural, lye2019deep, McFall2009ANN, Meade1994Numerical, Milligen1995NN, Niarchos2006Neural, pham2019neural, reisinger2019rectified, royo2016recursive, Rudd2014Constrained, ruthotto2019machine, Sirignano2018DGM, Tang2017Study, Tassa2007Least, weinan2018deep, Yadav2015Intro, yang2018physics, yang2019adversarial}} and inverse problems involving PDEs \cite{long2017pde,long2019pde, meng2019composite, meng2019ppinn, pang2019fpinns, raissi2018deep,raissi2018forward,raissi2017physicsi,raissi2017physicsii,Raissi2019PINN, uchiyama1993solving, yang2018physics, zhang2019learning, zhang2019quantifying}. Their popularity is due to universal approximation theorems that state that neural networks can approximate broad classes of (high-dimensional, nonlinear) functions on compact sets \cite{Cybenko1989Approximation, Hornik1991Approximation, Hornik1989Multilayer,Pinkus1999}. These properties, in particular, have been recently leveraged to approximate solutions to high-dimensional nonlinear HJ PDEs \cite{Han2018Solving,Sirignano2018DGM} and for the development of physics-informed neural networks that aim to solve supervised learning problems while respecting any given laws of physics described by a set of nonlinear PDEs \cite{Raissi2019PINN}. 

In this paper, we propose some neural network architectures that exactly represent viscosity solutions to HJ PDEs of the form of \eqref{eqn:intro_example}, where the Hamiltonians and initial data are obtained from the parameters of the neural network architectures. \revision{Recall our results require the initial data $J$ to be convex and the Hamiltonian $H$ to only depend on the gradient $\nabla_{\bx}S(\bx,t)$ (see Eq.~\eqref{eqn:intro_example}).} In other words, we show that some neural networks correspond to \revision{exact} representation formulas of HJ PDE solutions. \revision{To our knowledge, this is the first result that shows that certain neural networks can exactly represent solutions of certain HJ PDEs.}

\revision{Note that an alternative method to numerically evaluate solutions of HJ PDEs of the form of \eqref{eqn:intro_example} with convex initial data has been proposed in~\cite{Darbon2016Algorithms}.  This method relies on the Hopf formula and is only based on optimization. Therefore, this method is grid and approximation free, and works well in high dimension. Contrary to~\cite{Darbon2016Algorithms}, our proposed approach does not rely on any (possibly non-convex) optimization techniques.}

\bigbreak
\noindent
\textbf{Contributions of this paper.} 
In this paper, we prove that some classes of shallow neural networks are, under certain conditions, viscosity solutions to Hamilton--Jacobi equations. \revision{The main result of this paper is Thm.~\ref{thm:constructHJ}. We show in this theorem that the neural network architecture illustrated in Fig.~\ref{fig:nn_max} represents, under certain conditions, the viscosity solution to a set of first-order HJ PDEs of the form of \eqref{eqn:intro_example}, where the Hamiltonians and the convex initial data are obtained from the parameters of the neural network.} As a corollary of this result for the one-dimensional case, we propose a second neural network architecture (illustrated in Fig. \ref{fig:nn_argmax}) that represents the spatial gradient of the viscosity solution of the HJ PDE above in 1D and show in Proposition \ref{thm:conservation} that under appropriate conditions, this neural network corresponds to entropy solutions of some conservation laws in 1D. 

\revision{Let us emphasize that the proposed architecture in Fig.~\ref{fig:nn_max} for representing solutions to HJ PDEs allows us to numerically evaluate their solutions in high dimension without using grids.}

We also stress that our results do not rely on universal approximation properties of neural networks. \revision{Instead, our results show that the physics contained in HJ PDEs satisfying the conditions of Thm.~\ref{thm:constructHJ} can naturally be encoded by the neural network architecture depicted in Fig.~\ref{fig:nn_max}. Our results further suggest interpretations of this neural network architecture in terms of solutions to PDEs.}

We also test the proposed neural network architecture \revision{(depicted in Fig.~\ref{fig:nn_max})} on some inverse problems. To do so, we consider the following problem. Given training data sampled from the solution $S$ of a first-order HJ PDE \eqref{eqn:intro_example} with unknown \revision{convex} initial function $J$ and Hamiltonian $H$, we aim to recover the unknown initial function. After the training process using the Adam optimizer, the trained neural network with input time variable $t=0$ gives an approximation to the \revision{convex} initial function $J$. Moreover, the parameters in the trained neural network also provide partial information on the Hamiltonian $H$. The parameters only approximate the Hamiltonian at certain points, however, and therefore do not give complete information about the function. We show the experimental results on several examples. Our numerical results show that this problem cannot generally be solved using Adam optimizer with high accuracy. In other words, while \revision{our} theoretical results \revision{(see Thm.~\ref{thm:constructHJ})} show that the neural network representation \revision{(depicted in Fig.~\ref{fig:nn_max})} to some HJ PDEs is exact, the Adam optimizer for training the proposed networks in this paper sometimes gives large errors in some of our inverse problems, and as such there is no guarantee that the \revision{Adam optimizer} works well for the proposed network.

\bigbreak
\noindent

\textbf{Organization of this paper.}
In Sect. \ref{sec:background}, we briefly review shallow neural networks and concepts of convex analysis that will be used throughout this paper. \revision{In Sect. \ref{sec:archNN_theory}, we establish connections between the neural network architecture illustrated in Fig.~\ref{fig:nn_max} and viscosity solutions to HJ PDEs of the form of~\eqref{eqn:intro_example}, and the neural network architecture illustrated in Fig.~\ref{fig:nn_argmax} and one-dimensional conservation laws.} The mathematical set-up for establishing these connections is described in Sect. \ref{subsec:set-up}, our main results, which concern first-order HJ PDEs, are described in Sect. \ref{subsec:main}, \revision{and an extension of these results to one-dimensional conservation laws is presented in Sect. \ref{sec:conservation}}. \revision{In Sect. \ref{sec:numerical}, we perform numerical experiments to test the effectiveness of the Adam optimizer using our proposed architecture (depicted in Fig.~\ref{fig:nn_max}) for solving some inverse problems.} Finally, we draw some conclusions and directions for future work in Sect. \ref{sec:conclusion}. \revision{Several appendices contain proofs of our results.}
\section{Background}
\label{sec:background}
In this section, we introduce mathematical concepts that will be used in this paper. We review the standard structure of shallow neural networks from a mathematical point of view in Sect. \ref{sec:bkgd_nn} and present some fundamental definitions and results in convex analysis in Sect. \ref{sec:bkgd_convex}. For the notation, we use $\R^n$ to denote the $n$-dimensional Euclidean space. The Euclidean scalar product and Euclidean norm on $\mathbb{R}^{n}$ are denoted by $\left\langle \cdot,\cdot\right\rangle$ and $\left\Vert \cdot\right\Vert _{2}$. The set
of matrices with $m$ rows and $n$ columns with real entries is denoted by $\mathcal{M}_{m,n}(\R)$.

\subsection{Shallow neural networks}
\label{sec:bkgd_nn}
\def\sigmoid {\mathrm{sigmoid}}
\def\relu {\mathrm{ReLU}}
\def\softmax {\mathrm{softmax}}
\def\varzt {\tilde{z}}
\def\FunSet{\mathcal{F}}
\def\bw {\boldsymbol{w}}
\def\bb {\boldsymbol{b}}
\def\barf{\bar{f}}
\def\barbw{\bar{\bw}}
\def\barc{\bar{c}}
\def\barb{\bar{b}}
\def\bara{\bar{a}}
\def\nin{n}

Neural networks provide architectures for constructing complicated nonlinear functions from simple building blocks. Common neural network architectures in applications include, for example, feedforward neural networks in statistical learning, recurrent neural networks in natural language processing, and convolutional neural networks in imaging science. In this paper, we focus on shallow neural networks, a subclass of feedforward neural networks that typically consist of one hidden layer and one output layer. We give here a brief mathematical introduction to shallow neural networks. For more details, we refer the reader to \cite{Goodfellow2016Deep, Lecun2015Deep, Schmidhuber2015Deep} and the references listed therein.

A shallow neural network with one hidden layer and one output layer is a composition of affine functions with a nonlinear function. A hidden layer with $m \in \mathbb{N}$ neurons comprises $m$ affine functions of an \textit{input} $\bx \in \R^n$ with \textit{weights} $\bw_i \in \R^n$ and \textit{biases} $b_i \in \R$:
\begin{equation*}
    \R^n \times \R^n \times \R \ni (\bx,\bw_i,b_i) \mapsto \langle \bw_i, \bx\rangle + b_i.
\end{equation*}
These $m$ affine functions can be succinctly written in vector form as $\bW\bx + \boldsymbol{b}$, where the matrix $\bW \in \mathcal{M}_{m,n}(\R)$ has for rows the weights $\bw_i$ and the vector $b \in \R^m$ has for entries the biases $b_i$. The output layer comprises a nonlinear function $\sigma \colon \R^m \to \R$ that takes for input the vector $\bW\bx + \boldsymbol{b}$ of affine functions and gives the number
\begin{equation*}
    \R^n \times \R^n \times \R \ni (\bx,\bw_i,b_i) \mapsto \sigma \left(\bW\bx + \boldsymbol{b}\right).
\end{equation*}
The nonlinear function $\sigma$ is called the \textit{activation function} of the output layer. 

In Sect. \ref{sec:numerical}, we will consider the following problem: Given data points $\{(\bx_i, y_i)\}_{i=1}^N \subset \R^{n}\times \R$, infer the relationship between the input $\bx_i$'s and the output $y_i$'s. To infer this relation, we assume that the output takes the form (or can be approximated by) $y_i = \sigma\left(\bW\bx_i + \boldsymbol{b}\right)$ for some known activation function $\sigma$, unknown matrix of weights $\bW \in \mathcal{M}_{m,n}(\R)$, and unknown vector of bias $\boldsymbol{b}$. A standard approach to solve such a problem is to estimate the weights $\bw_i$ and biases $b_i$ so as to minimize the mean square error
\begin{equation}
    \{(\barbw_i, \barb_i)\}_{i=1}^{m} \in \argmin_{\{(\bw_i, b_i)\}_{i=1}^{m} \subset \R^n \times \R}\left\{\frac{1}{N}\sum_{i=1}^{N}\left(\sigma \left(\bW\bx_i + \boldsymbol{b}\right) - y_i\right)^2\right\}.
\end{equation}
In the field of machine learning, solving this minimization problem is called the \textit{learning} or \textit{training process}. The data $\{(\bx_i, y_i)\}_{i=1}^N$ used in the training process is called \textit{training data}. Finding a global minimizer is generally difficult due to the complexity of the minimization problem and that the objective function is not convex with respect to the weights and biases. State-of-the-art algorithms for solving these problems are stochastic gradient descent based methods with momentum acceleration, such as the Adam optimizer for neural networks \cite{Kingma2015Adam}. This algorithm will be used in our numerical experiments. 

\subsection{Convex analysis} \label{sec:bkgd_convex}
We introduce here several definitions and results of convex analysis that will be used in this paper. We refer readers to Hiriart\textendash Urruty and Lemar\'echal \citep{hiriart2013convexI,hiriart2013convexII} and Rockafellar \cite{rockafellar1970convex} for comprehensive references on finite-dimensional convex analysis. 

\begin{defn}
\label{def:closure_int}(Convex sets, relative interiors, and convex hulls) A set $C\subset\mathbb{R}^{n}$ is called convex if for any $\lambda\in[0,1]$ and any $\boldsymbol{x},\boldsymbol{y}\in C$, the element $\lambda\boldsymbol{x} + (1-\lambda)\boldsymbol{y}$ is in $C$. The relative interior of a convex set $C\subset\mathbb{R}^{n}$, denoted by $\ri C$, consists of the points in the interior of the unique smallest affine set containing $C$. The convex hull of a set $C$, denoted by $\conv{C}$, consists of all the convex combinations of the elements of $C$. An important example of a convex hull is the unit simplex in $\R^n$, which we denote by
\begin{equation}\label{eqt:def_unitsimplex}
    \unitsim_n \coloneqq \left\{(\alpha_1, \dots, \alpha_n)\in [0,1]^n:\ \sum_{i=1}^n \alpha_i = 1\right\}.
\end{equation}
\end{defn} 

\begin{defn}
\label{def:domains_prop}(Domains and proper functions) The domain of a function $f\colon\mathbb{R}^{n}\to\mathbb{R}\cup\{+\infty\}$ is the set
\[
\dom f=\left\{ \boldsymbol{x}\in\mathbb{R}^{n}:f(\boldsymbol{x})<+\infty\right\} .
\]
A function $f$ is called proper if its domain is non-empty and $f(\boldsymbol{x})>-\infty$ for every $\boldsymbol{x}\in\mathbb{R}^{n}$. 
\end{defn}

\begin{defn}
\label{def:convex}(Convex functions, lower semicontinuity, and convex envelopes)
A proper function $f\colon\mathbb{R}^{n}\to\mathbb{R}\cup\{+\infty\}$ is called convex if the set $\dom f$ is convex and if for any $\boldsymbol{x},\boldsymbol{y}\in\dom f$ and all $\lambda\in[0,1]$, there holds
\begin{equation}
f(\lambda\boldsymbol{x}+(1-\lambda)\boldsymbol{y})\leqslant\lambda f(\boldsymbol{x})+(1-\lambda)f(\boldsymbol{y})\label{eq:convex_def}
\end{equation}

A proper function $f\colon\mathbb{R}^{n}\to\mathbb{R}\cup\{+\infty\}$ is called lower semicontinuous if for every sequence $\left\{ \boldsymbol{x}_{k}\right\} _{k=1}^{+\infty}\in\mathbb{R}^{n}$ with $\lim_{k\to+\infty}\boldsymbol{x}_{k}=\boldsymbol{x}\in\mathbb{R}^{n}$, we have $\liminf_{k\to+\infty}f(\boldsymbol{x}_{k})\geqslant f(\boldsymbol{x})$.

The class of proper, lower semicontinuous convex functions is denoted by $\Gamma_{0}(\mathbb{R}^{n})$.

Given a function $f\colon\R^n \to \R\cup\{+\infty\}$, we define its convex envelope $\co{f}$ as the largest convex function such that $\co{f}(\bx)\leqslant f(\bx)$ for every $\bx \in \R^n$. We define the convex lower semicontinuous envelope $\cobar{f}$ as the largest convex and lower semicontinuous function such that $\cobar{f}(\bx) \leqslant f(\bx)$ for every $\bx \in \R^n$.
\end{defn}

\begin{defn}
\label{def:subgrad}(Subdifferentials and subgradients) The subdifferential
$\partial f(\boldsymbol{x})$ of $f\in\Gamma_{0}(\mathbb{R}^{n})$
at $\boldsymbol{x}\in\dom f$ is the set (possibly empty) of vectors
$\boldsymbol{p}\in\mathbb{R}^{n}$ satisfying
\begin{equation}\label{eq:subgrad_def}
\forall\boldsymbol{y}\in\mathbb{R}^{n},\mbox{ }f(\boldsymbol{y})\geqslant f(\boldsymbol{x})+\left\langle \boldsymbol{p},\boldsymbol{y}-\boldsymbol{x}\right\rangle.
\end{equation}
The subdifferential $\partial f(\boldsymbol{x})$ is a closed convex set whenever it is non\textendash empty, and any vector $\boldsymbol{p}\in\partial f(\boldsymbol{x})$ is called a subgradient of $f$ at $\boldsymbol{x}$. If $f$ is a proper convex function, then $\partial f(\boldsymbol{x})\neq\emptyset$ whenever $\boldsymbol{x}\in\ri(\dom f)$, and $\partial f(\boldsymbol{x})=\emptyset$ whenever $\boldsymbol{x}\notin\dom J$ \cite[Thm. 23.4]{rockafellar1970convex}. If a convex function $f$ is differentiable at $\boldsymbol{x}_{0}\in\mathbb{R}^{n}$, then its gradient $\nabla_{\boldsymbol{x}}f(\boldsymbol{x}_{0})$ is the unique subgradient of $f$ at $\boldsymbol{x}_{0}$, and conversely if $f$ has a unique subgradient at $\boldsymbol{x}_{0}$, then $f$ is differentiable at that point \cite[Thm. 21.5]{rockafellar1970convex}.
\end{defn}

\begin{defn}
\label{def:legendre_t}(Fenchel--Legendre transforms) Let $f\in\Gamma_{0}(\mathbb{R}^{n})$.
The Fenchel--Legendre transform $f^{*}\colon\mathbb{R}^{n}\to\mathbb{R}\cup\{+\infty\}$ of $f$ is defined as
\begin{equation}
f^{*}(\boldsymbol{p})=\sup_{\boldsymbol{x}\in\mathbb{R}^{n}}\left\{ \left\langle \boldsymbol{p},\boldsymbol{x}\right\rangle -f(\boldsymbol{x})\right\} .\label{eq:fenchel_t_def}
\end{equation}
For any $f\in\Gamma_{0}(\mathbb{R}^{n})$, the mapping $f\mapsto f^{*}$ is one-to-one, $f^{*}\in\Gamma_{0}(\mathbb{R}^{n})$, and $(f^{*})^{*}=f$. Moreover, for any $(\boldsymbol{x},\boldsymbol{p})\in\mathbb{R}^{n}\times\mathbb{R}^{n}$, the so-called Fenchel's inequality holds:
\begin{equation}
f(\boldsymbol{x})+f(\boldsymbol{p})\geqslant\left\langle \boldsymbol{x},\boldsymbol{p}\right\rangle ,\label{eq:fenchel_ineq}
\end{equation}
with equality attained if and only if $\boldsymbol{p}\in\partial f(\boldsymbol{x})$, if and only if $\boldsymbol{x}\in\partial f^{*}(\boldsymbol{p})$ \cite[Cor. X.1.4.4]{hiriart2013convexII}.
\end{defn}

We summarize some notations and definitions in Tab. \ref{tab:notations}.
\newcolumntype{s}{>{\hsize=.7\hsize}X}
\begin{table}[!h]
\centering
 \caption{Notation used in this paper. Here, we use $C$ to denote a set in $\R^n$, $f$ to denote a function from $\R^n$ to $\R\cup\{+\infty\}$ and $\bx$ to denote a vector in $\Rn$.}
 \label{tab:notations}
\begin{tabularx}{\textwidth}{ l|s|X }
\hline
\noalign{\smallskip}
 Notation & Meaning & Definition\\
 \noalign{\smallskip}
 \hline
 \noalign{\smallskip}

 $\left\langle \cdot,\cdot\right\rangle $ & 
 Euclidean scalar product in $\mathbb{R}^{n}$ & $\langle \bx, \by\rangle \coloneqq \sum_{i=1}^n x_iy_i$
 \\
 $\left\Vert \cdot\right\Vert _{2}$ &
 Euclidean norm in $\mathbb{R}^{n}$ & $\left\Vert \bx\right\Vert _{2} \coloneqq \sqrt{\langle \bx, \bx\rangle}$
 \\
 $\ri C$ & Relative interior of $C$ & The interior of $C$ with respect to the minimal hyperplane containing $C$ in $\Rn$
 \\
 $\conv{C}$  & Convex hull of $C$ & The set containing all convex combinations of the elements of $C$
 \\
 $\unitsim_n$ & Unit simplex in $\R^n$ & $\left\{(\alpha_1, \dots, \alpha_n)\in [0,1]^n:\ \sum_{i=1}^n \alpha_i = 1\right\}$
 \\

 $\dom f$ & Domain of $f$ & $\{\bx\in \Rn:\ f(\bx) < +\infty\}$
 \\
 $\gmRn$ & A useful and standard class of convex functions & The set containing all proper, convex, lower semicontinuous functions from $\Rn$ to $\R\cup\{+\infty\}$
 \\
 $\co{f}$ & Convex envelope of $f$ & The largest convex function such that $\co{f}(\bx)\leqslant f(\bx)$ for every $\bx \in \R^n$
 \\
 $\cobar f$ & Convex and lower semicontinuous envelope of $f$ &
 The largest convex and lower semicontinuous function such that $\cobar{f}(\bx) \leqslant f(\bx)$ for every $\bx \in \R^n$
 \\
 $\partial f(\bx)$ & Subdifferential of $f$ at $\bx$ & $\{\bp\in\Rn:\ f(\by)\geqslant f(\bx) + \langle \bp, \by-\bx\rangle \ \forall \by\in\Rn\}$
 \\
 $f^*$ & Fenchel--Legendre transform of $f$ & $f^*(\bp) \coloneqq \sup_{\bx\in \Rn} \{\langle \bp, \bx\rangle - f(\bx)\}$
 \\
 \noalign{\smallskip}
\hline
 \end{tabularx}
\end{table}
\section{Connections between neural networks and Hamilton--Jacobi equations} \label{sec:archNN_theory}

This section establishes connections between HJ PDEs and neural network architectures.
Subsection~\ref{subsec:set-up} presents the mathematical set-up, subsection~\ref{subsec:main} describes our main results for first-order HJ PDEs, and finally subsection~\ref{sec:conservation} presents our results for first-order one dimensional conservation laws.

\subsection{Set-up} \label{subsec:set-up}

\begin{figure}[ht]
\includegraphics[width=\textwidth]{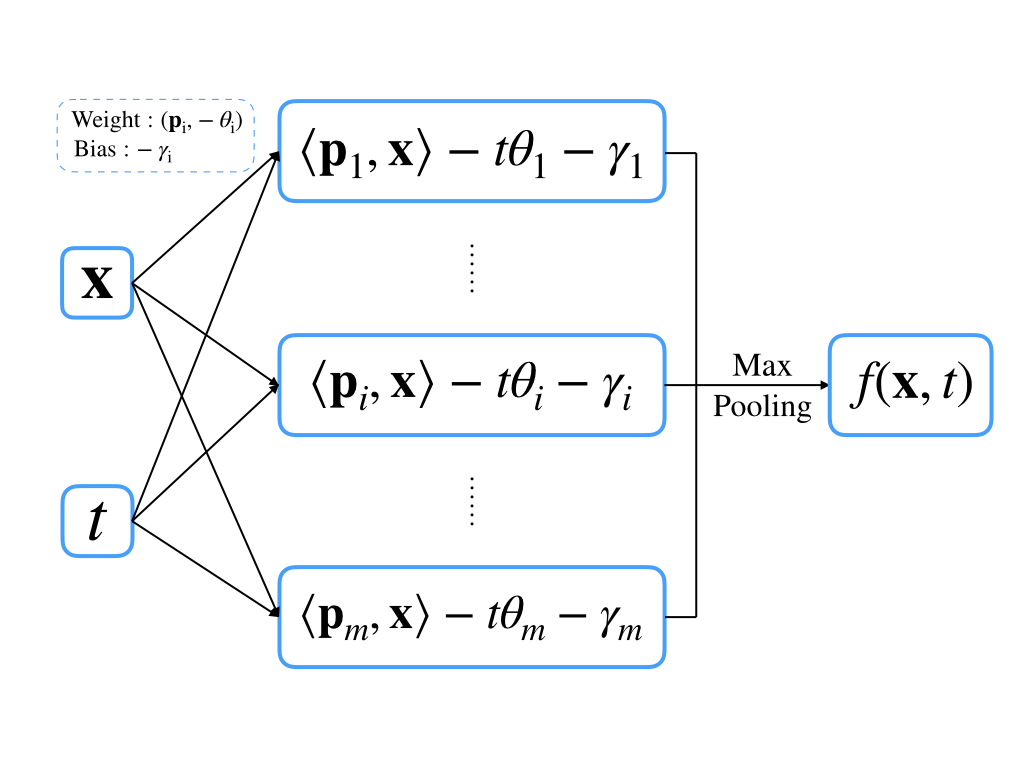}
\caption{Illustration of the structure of the neural network \eqref{eqt:deff} that can represent the viscosity solution to first-order Hamilton\textendash Jacobi equations.}
\label{fig:nn_max}
\end{figure}
In this section, we consider the function $f\colon\ \R^n\times [0,+\infty)\to \R$ given by the neural network in Fig. \ref{fig:nn_max}. Mathematically, the function $f$ can be expressed using the following formula
\begin{equation}\label{eqt:deff}
f(\bx, t; \{(\bp_{i}, \theta_i, \gamma_{i})\}_{i=1}^m) = 
\max_{i\in\{1,\dots,m\}}\{\left\langle \boldsymbol{p}_{i},\boldsymbol{x}\right\rangle -t\theta_{i}-\gamma_{i}\}.
\end{equation}
Our goal is to show that the function $f$ in \eqref{eqt:deff} is the unique uniformly continuous viscosity solution to a suitable Hamilton--Jacobi equation. In what follows we denote $f(\bx,t; \{(\boldsymbol{p}_{i}, \theta_i, \gamma_{i})\}_{i=1}^m)$ by $f(\bx,t)$ when there is no ambiguity in the parameters.

We adopt the following assumptions on the parameters:
\begin{itemize}
    \item[(A1)] The parameters $\{\bp_i\}_{i=1}^{m}$ are pairwise distinct, i.e., $\bp_i\neq \bp_j$ if $i\neq j$.
    \item[(A2)] There exists a convex function $g\colon\ \R^n\to \R$ such that $g(\bp_i) = \gamma_i$.
    \item[(A3)] For any $j\in \{1,\dots,m\}$ and any $(\alpha_1,\dots, \alpha_m) \in \R^m$ that satisfy
    \begin{equation}\label{eqt:assumption3}
        \begin{cases}
        (\alpha_1,\dots, \alpha_m)\in\unitsim_m \text{ with }\alpha_j = 0,\\
        \sum_{i\neq j}\alpha_i \bp_i = \bp_j,\\
        \sum_{i\neq j}\alpha_i \gamma_i = \gamma_j,
        \end{cases}
    \end{equation}
    there holds $\sum_{i\neq j}\alpha_i \theta_i > \theta_j$.
\end{itemize}

Note that (A3) is not a strong assumption. Indeed, if there exist $j\in\{1,\dots,m\}$ and $(\alpha_1,\dots, \alpha_m)\in \R^m$ satisfying Eq. \eqref{eqt:assumption3} and $\sum_{i\neq j}\alpha_i \theta_i \leqslant \theta_j$, then
\begin{equation*}
    \langle \bp_j,\bx\rangle - t\theta_j - \gamma_j
    \leqslant \sum_{i\neq j} \alpha_i(\langle \bp_i,\bx\rangle - t\theta_i - \gamma_i)
    \leqslant \max_{i\neq j}\{\left\langle \boldsymbol{p}_{i},\boldsymbol{x}\right\rangle -t\theta_{i}-\gamma_{i}\}.
\end{equation*}
As a result, the j\textsuperscript{th} neuron in the network can be removed without changing the value of $f(\bx,t)$ for any $\bx\in\R^n$ and $t\geqslant 0$. Removing all such neurons in the network, we can therefore assume (A3) holds.

Our aim is to identify the HJ equations whose viscosity solutions correspond to the neural network $f$ defined by Eq.~\eqref{eqt:deff}. Here, $\bx$ and $t$ play the role of the spatial and time variables, and $f(\cdot,0)$ corresponds to the initial data. To simplify the notation, we define the function $J\colon\ \R^n\to \R$ as
\begin{equation}\label{eqt:defJ}
    f(\bx,0) = J(\bx)\coloneqq
    \max_{i\in\{1,\dots,m\}}\{\left\langle \boldsymbol{p}_{i},\boldsymbol{x}\right\rangle -\gamma_{i}\}
\end{equation}
and the set $I_{\bx}$ as the collection of maximizers in Eq.~\eqref{eqt:defJ} at $\bx$, that is,
\begin{equation}\label{eqt:defIx}
    I_{\bx}\coloneqq 
    \argmax_{i\in\{1,\dots,m\}}\{\left\langle \boldsymbol{p}_{i},\boldsymbol{x}\right\rangle -\gamma_{i}\}.
\end{equation}
\revision{Note that the initial data $J$ given by \eqref{eqt:defJ} is a convex and polyhedral function, and it} satisfies several properties that we describe in the following lemma.

\begin{lem}\label{lem:formulaJstar}
Suppose $\{(\bp_i,\gamma_i)\}_{i=1}^{m}\subset \R^n \times \R$ satisfy assumptions (A1) and (A2). Then the following statements hold.
\begin{itemize}
    \item[(i)] The Fenchel--Legendre transform of $J$ is given by the convex and lower semicontinuous function
    \begin{equation}\label{eqt:defJstar}
        J^*(\bp) = \begin{dcases}
        \min_{\substack{(\alpha_{1},\dots,\alpha_{m})\in\unitsim_m\\
        \sum_{i=1}^{m}\alpha_{i}\bp_{i}=\bp
        }
        }\left\{\sum_{i=1}^{m}\alpha_{i}\gamma_{i}\right\}, &\text{\rm if }\bp \in \conv(\{\bp_i\}_{i=1}^{m}),\\
        +\infty, &\text{\rm otherwise}.
        \end{dcases}
    \end{equation}
Moreover, its restriction to $\dom J^*$ is continuous, and the subdifferential $\partial J^*(\bp)$ is non-empty for every $\bp\in\dom J^*$.

\item[(ii)] Let $\bp\in \dom J^*$ and $\bx\in\partial J^*(\bp)$. Then $(\alpha_1,\dots,\alpha_m)\in \R^m$ is a minimizer in Eq.~\eqref{eqt:defJstar} if and only if it satisfies the constraints
\begin{itemize}
    \item[(a)] $(\alpha_1, \dots, \alpha_m)\in\unitsim_m$,
    \item[(b)] $\sum_{i=1}^{m}\alpha_{i}\bp_{i}=\bp$,
    \item[(c)] $\alpha_i=0$ for any $i\not\in I_{\bx}$.
\end{itemize}
\item[(iii)] For each $i,k\in\{1,\dots,m\}$, let
\begin{equation*}
    \alpha_i = \delta_{ik}\coloneqq\begin{cases}
    1, &\text{\rm if }i=k,\\
    0, &\text{\rm if }i\neq k.
    \end{cases}
\end{equation*}
Then $(\alpha_1,\dots, \alpha_m)$ is a minimizer in Eq.~\eqref{eqt:defJstar} at the point $\bp=\bp_k$. Hence, we have $J^*(\bp_k) = \gamma_k$.
\end{itemize}
\end{lem}
\begin{proof}
See Appendix \ref{sec:prooflem31} for the proof.
\end{proof}

Having defined the initial condition $J$, the next step is to define a Hamiltonian $H$. To do so, first denote by $\mathcal{A}(\bp)$ the set of minimizers in Eq.~\eqref{eqt:defJstar} evaluated at $\bp\in\dom J^*$, i.e.,
\begin{equation}\label{eqt:defsetA}
\mathcal{A}(\bp) \coloneqq 
\argmin_{\substack{(\alpha_{1},\dots\alpha_{m})\in \unitsim_m\\
\sum_{i=1}^{m}\alpha_{i}\bp_{i}=\bp
}
}\left\{\sum_{i=1}^{m}\alpha_{i}\gamma_{i}\right\}.
\end{equation}
Note that the set $\mathcal{A}(\bp)$ is non-empty for every $\bp\in\dom J^*$ by Lem. \ref{lem:formulaJstar}(i). Now, we define the Hamiltonian function $H\colon\R^n \to \R\cup\{+\infty\}$ by
\begin{equation}\label{eqt:defH}
    H(\bp)\coloneqq \begin{dcases}
    \inf_{\bm{\alpha}\in \mathcal{A}(\bp)} \left\{\sum_{i=1}^m \alpha_i \theta_i\right\},& \text{if }\bp \in \dom J^*,\\
    +\infty, &\text{otherwise}.
    \end{dcases}
\end{equation}
\revision{The function $H$ defined in \eqref{eqt:defH} is a polyhedral function whose properties are stated in the following lemma.}

\begin{lem}\label{lem:Hprop}
Suppose $\{(\bp_i,\theta_i,\gamma_i)\}_{i=1}^{m}\subset \R^n \times \R \times \R$ satisfy assumptions (A1)-(A3). Then the following statements hold.
\begin{itemize}
    \item[(i)] For every $\bp \in \dom J^*$, the set $\mathcal{A}(\bp)$ is compact and Eq.~\eqref{eqt:defH} has at least one minimizer.
    \item[(ii)] The restriction of $H$ to $\dom J^*$ is a bounded and continuous function.
    \item[(iii)] There holds $H(\bp_i) = \theta_i$ for each $i \in\{1,\dots,m\}$.
\end{itemize}
\end{lem}

\begin{proof}
See Appendix \ref{sec:prooflem32} for the proof.
\end{proof}

\subsection{Main results: First-order Hamilton--Jacobi equations} \label{subsec:main}
Let $f$ be the function represented by the neural network architecture in Fig. \ref{fig:nn_max}, whose mathematical definition is given in Eq.~\eqref{eqt:deff}. In the following theorem, we identify the set of first-order HJ equations whose viscosity solutions correspond to the neural network $f$. Specifically, $f$ solves a first-order HJ equation with Hamiltonian $H$ and initial function $J$ that were defined previously in Eqs.~\eqref{eqt:defH} and \eqref{eqt:defJ}, respectively. Furthermore,  we provide necessary and sufficient conditions for a first-order HJ equation of the form of \eqref{eqn:intro_example} to have for viscosity solution the neural network $f$.

\begin{thm}\label{thm:constructHJ}
Suppose the parameters $\{(\bp_i,\theta_i,\gamma_i)\}_{i=1}^{m}\subset \R^n \times \R \times \R$ satisfy assumptions (A1)-(A3), and let $f$ be the neural network defined by Eq. \eqref{eqt:deff} with these parameters. Let $J$ and $H$ be the functions defined in Eqs. \eqref{eqt:defJ} and \eqref{eqt:defH}, respectively, and let $\tilde{H}\colon\ \R^n\to\R$ be a continuous function. Then the following two statements hold.
\begin{itemize}
    \item[(i)] The neural network $f$ is the unique uniformly continuous viscosity solution to the first-order Hamilton--Jacobi equation 
    \begin{equation} \label{eqt:HJ-H}
    \begin{dcases}
    \frac{\partial f}{\partial t}(\bx,t) + H(\nabla_{\bx} f(\bx,t)) =0, & \text{\rm in }\mathbb{R}^{n}\times(0,+\infty),\\
    f(\bx,0) = J(\bx), & \text{\rm in }\mathbb{R}^{n}.
    \end{dcases}
    \end{equation}
    Moreover, $f$ is jointly convex in ($\bx$,t).
    \item[(ii)] The neural network $f$ is the unique uniformly continuous viscosity solution to the first-order Hamilton--Jacobi equation
\begin{equation} \label{eqt:HJ}
    \begin{dcases}
    \frac{\partial f}{\partial t}(\bx,t) + \tilde{H}(\nabla_{\bx} f(\bx,t)) =0, & \text{\rm in }\mathbb{R}^{n}\times(0,+\infty),\\
    f(\bx,0) = J(\bx), & \text{\rm in }\mathbb{R}^{n},
    \end{dcases}
\end{equation}
if and only if $\tilde{H}(\bp_i) = H(\bp_i)$ for each $i\in \{1,\dots,m\}$ and $\tilde{H}(\bp)\geqslant H(\bp)$ for every $\bp\in \dom J^*$. 
\end{itemize}
\end{thm}

\begin{proof}
\revision{See Appendix \ref{sec:pf_thmconstructHJ} for the proof.}
\end{proof}

\begin{remark}
This theorem identifies the set of HJ equations with initial data $J$ whose solution is given by the neural network $f$. To each such HJ equation, there corresponds a continuous Hamiltonian $\tilde{H}$ satisfying $\tilde{H}(\bp_i) = H(\bp_i)$ for every $i=\{1,\dots,m\}$ and $\tilde{H}(\bp)\geqslant H(\bp)$ for every $\bp\in \dom J^*$. The smallest possible Hamiltonian satisfying these constraints is the function $H$ defined in \eqref{eqt:defH}, and its corresponding HJ equation is given by \eqref{eqt:HJ-H}.
\end{remark}

\def \Jtrue {J^\text{true}}
\def \Htrue {H^\text{true}}

%\begin{remark} \label{rem:egHJ}
\begin{example}\label{eg:egHJ1}
In this example, we consider the HJ PDE with initial data $\Jtrue(\bx) = \|\bx\|_1$ and the Hamiltonian $\Htrue(\bp) = -\frac{\|\bp\|_2^2}{2}$ for all $\bx, \bp\in \R^n$. The viscosity solution to this HJ PDE is given by
\[
S(\bx,t) = \|\bx\|_1 + \frac{nt}{2} = \max_{i\in \{1,\dots, m\}} \{\langle \bp_i, \bx\rangle - t\theta_i -\gamma_i\}, \text{ for every } \bx\in\R^n \text{ and } t\geqslant 0,
\]
where $m=2^n$, each entry of $\bp_i$ takes value in $\{\pm 1\}$, and $\theta_i = -\frac{n}{2}$, $\gamma_i = 0$ for every $i\in\{1,\dots, m\}$. In other words, the solution $S$ can be represented using the proposed neural network with parameters $\{(\bp_i, -\frac{n}{2}, 0)\}_{i=1}^m$. We can compute the functions $J$ and $H$ using definitions in Eqs. \eqref{eqt:defJ} and \eqref{eqt:defH} and then obtain
\begin{equation*}
    \begin{split}
        &J(\bx) = \|\bx\|_1 = \Jtrue(\bx), \text{ for every }\bx\in\R^n;\\
        &H(\bp) = \begin{cases}
        -\frac{n}{2}, & \bp\in [-1,1]^n;\\
        +\infty, &\text{otherwise}.
        \end{cases}
    \end{split}
\end{equation*}
Thm. \ref{thm:constructHJ} stipulates that $S$ solves the HJ PDE \eqref{eqt:HJ} if and only if $\tilde{H}(\bp_i) = -\frac{n}{2}$ for every $i\in\{1,\dots, m\}$ and $\tilde{H}(\bp)\geq -\frac{n}{2}$ for every $\bp\in [-1,1]^n\setminus \{\bp_i\}_{i=1}^m$. The Hamiltonian $\Htrue$ is one candidate satisfying these constraints.
\end{example}

\begin{example}
In this example, we consider the case when $\Jtrue(\bx) = \|\bx\|_\infty$ and $\Htrue(\bp) = -\frac{\|\bp\|_2^2}{2}$ for every $\bx, \bp\in \R^n$. Denote by $\bolde_i$ the $i^\text{th}$ standard unit vector in $\R^n$. Let $m=2n$, $\{\bp_i\}_{i=1}^m = \{\pm \bolde_i\}_{i=1}^n$, $\theta_i = -\frac{n}{2}$, and $\gamma_i = 0$ for every $i\in\{1,\dots, m\}$. The viscosity solution $S$ is given by
\[
S(\bx,t) = \|\bx\|_\infty + \frac{nt}{2} = \max_{i\in \{1,\dots, m\}} \{\langle \bp_i, \bx\rangle - t\theta_i - \gamma_i \}, \text{ for every } \bx\in\R^n \text{ and } t\geqslant 0.
\]
Hence, $S$ can be represented using the proposed neural network with parameters $\{(\bp_i, -\frac{n}{2}, 0)\}_{i=1}^m$. 
Similarly as in the first example, we compute $J$ and $H$ and obtain the following results
\begin{equation*}
    \begin{split}
        &J(\bx) = \|\bx\|_\infty, \text{ for every } \bx \in \R^n;\\
        &H(\bp) = \begin{cases}
        -\frac{n}{2}, & \bp\in B_n;\\
        +\infty, &\text{otherwise},
        \end{cases}
    \end{split}
\end{equation*}
where $B_n$ denotes the unit ball with respect to the $l^1$ norm in $\R^n$, i.e., $B_n = \conv \{\pm \bolde_i:\ i\in \{1,\dots, n\}\}$. By Thm. \ref{thm:constructHJ}, $S$ is a viscosity solution to the HJ PDE \eqref{eqt:HJ} if and only if $\tilde{H}(\bp_i) = -\frac{n}{2}$ for every $i\in\{1,\dots, m\}$ and $\tilde{H}(\bp)\geq -\frac{n}{2}$ for every $\bp\in B_n\setminus \{\bp_i\}_{i=1}^m$. The Hamiltonian $\Htrue$ is one candidate satisfying these constraints.
\end{example}

\begin{example} \label{eg:egHJ3}
\revision{
In this example, we consider the HJ PDE with Hamiltonian $\Htrue(\bp) = \|\bp\|_1$ and initial data $\Jtrue(\bx) = \max\left\{\|\bx\|_\infty, \frac{1}{\sqrt{2}}(|x_1|+|x_2|)\right\}$, for all $\bp\in\R^n$ and $\bx = (x_1,x_2,\dots,x_n)\in\R^n$. The corresponding neural network has $m = 2n+5$ neurons, where the parameters are given by 
\begin{equation*}
\begin{split}
    &\{(\bp_i, \theta_i, \gamma_i)\}_{i=1}^{2n} = \{(\bolde_i, 1, 0)\}_{i=1}^n \cup \{(-\bolde_i, 1, 0)\}_{i=1}^n, \\
    &(\bp_{2n+1}, \theta_{2n+1}, \gamma_{2n+1}) = (\mathbf{0}, 0, 0), \\
    &\{(\bp_i, \theta_i, \gamma_i)\}_{i=2n+2}^{2n+5} = \left\{\frac{1}{\sqrt{2}} (\alpha \bolde_1 + \beta \bolde_2, 2, 0)\colon \alpha, \beta\in \{\pm 1\}\right\},
\end{split}
\end{equation*}
where $\bolde_i$ is the $i^\text{th}$ standard unit vector in $\R^n$ and $\mathbf{0}$ denotes the zero vector in $\R^n$. The functions $J$ and $H$ defined by \eqref{eqt:defJ} and \eqref{eqt:defH} coincide with the underlying true initial data $\Jtrue$ and Hamiltonian $\Htrue$. Therefore, by Thm. \ref{thm:constructHJ}, the proposed neural network represents the viscosity solution to the HJ PDE. In other words, given the true parameters $\{(\bp_i, \theta_i, \gamma_i)\}_{i=1}^m$, the proposed neural network solves this HJ PDE without the curse of dimensionality. We illustrate the solution with dimension $n=16$ in Fig. \ref{fig:eg3}, which shows several slices of the solution evaluated at $\bx = (x_1, x_2, 0, \dots, 0)\in\R^{16}$ and $t = 0, 1, 2, 3$ in figures 2(A), 2(B), 2(C), 2(D), respectively. In each figure, the $x$ and $y$ axes correspond to the first two components $x_1$ and $x_2$ in $\bx$, while the color represents the function value $S(\bx, t)$.
}
\begin{figure}[htbp]
\centering
\subfloat[]{\label{subfig:eg3_1}\includegraphics[width = 0.5\textwidth]{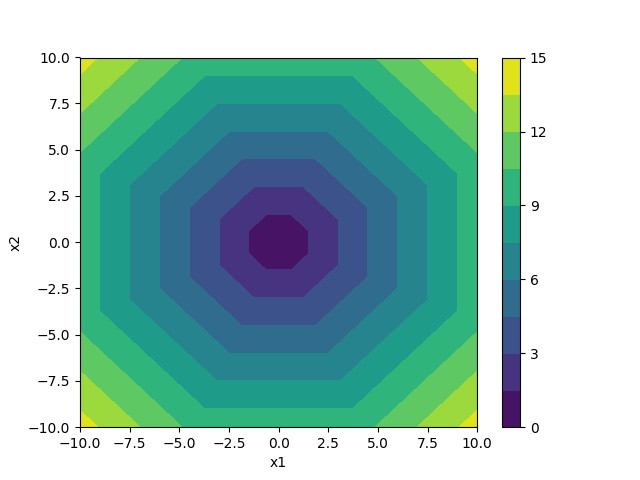}}
\subfloat[]{\label{subfig:eg3_2}\includegraphics[width = 0.5\textwidth]{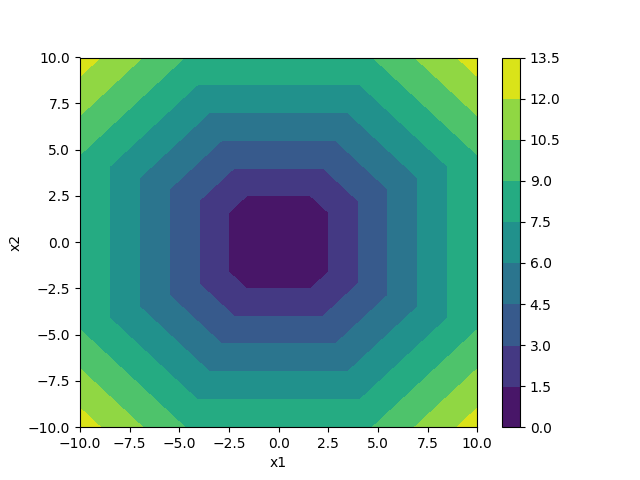}}\\
\subfloat[]{\label{subfig:eg3_3}\includegraphics[width = 0.5\textwidth]{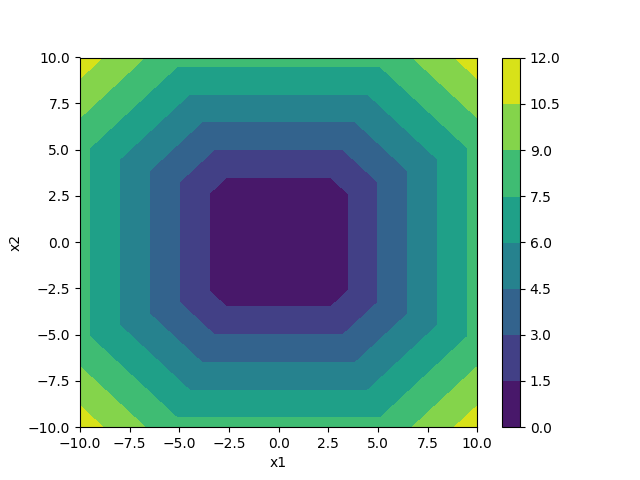}}
\subfloat[]{\label{subfig:eg3_4}\includegraphics[width = 0.5\textwidth]{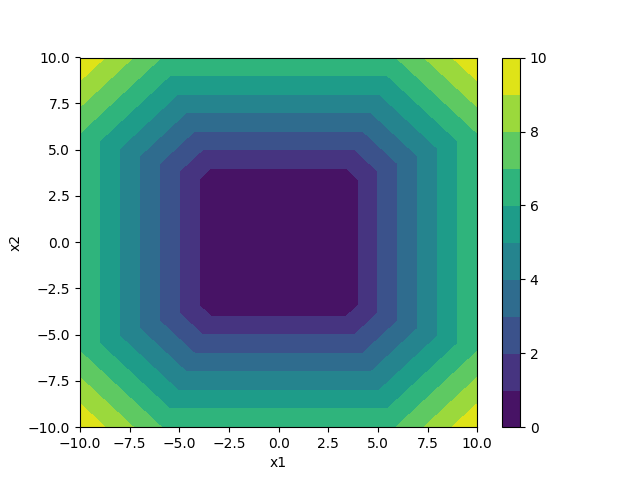}}
\caption{
\revision{
The solution $S\colon \R^{16}\times [0,+\infty)\to \R$ to the HJ PDE in Example \ref{eg:egHJ3} is solved using the proposed neural network. Several slices of the solution $S$ evaluated at $\bx=(x_1, x_2, 0,\dots, 0)$ and $t=0,1,2,3$ are shown in figures 2(A), 2(B), 2(C), 2(D), respectively. In each figure, the $x$ and $y$ axes correspond to the first two components $x_1$ and $x_2$ in the variable $\bx$, while the color represents the function value $S(\bx, t)$.
} % end revision
}
\label{fig:eg3}
\end{figure}
\end{example}

\begin{remark}
\revision{Let $\epsilon > 0$ and consider the neural network $f_{\epsilon} \colon \R^n\times[0,+\infty)\to \R$ defined by
\begin{equation} \label{eq:log-exponential_network}
f_{\epsilon}(\bx,t)\coloneqq\epsilon\log\left(\sum_{i=1}^{m}e^{\left(\left\langle \bp_{i},\bx\right\rangle -t\theta_{i}-\gamma_{i}\right)/\epsilon}\right)
\end{equation}
and illustrated in Fig. \ref{fig: log-exponential_network_figure}. This neural network substitutes the non-smooth maximum activation function in the neural network $f$ defined by Eq.~\eqref{eqt:deff} (and depicted in Fig.~\ref{fig:nn_max}) with a smooth log-exponential activation function. When the parameter $\theta_i = -\frac{1}{2}\left\Vert \bp_{i}\right\Vert _{2}^{2}$, then the neural network $f_\epsilon$ is the unique, jointly convex and smooth solution to the following viscous HJ PDE
\begin{equation}
\begin{dcases}\label{eq:hj_visc_pde_sum}
\frac{\partial f_{\epsilon}(\bx,t)}{\partial t}-\frac{1}{2}\left\Vert \nabla_{\bx}f_{\epsilon}(\bx,t)\right\Vert _{2}^{2}=\frac{\epsilon}{2}\Laplacian f_{\epsilon}(\bx,t) & \text{\rm in }\mathbb{R}^{n}\times(0,+\infty),\\
f_{\epsilon}(\bx,0)=\epsilon\log\left(\sum_{i=1}^{m}e^{\left(\left\langle \bp_{i},\bx\right\rangle -\gamma_{i}\right)/\epsilon}\right) & \text{\rm in }\mathbb{R}^{n}.
\end{dcases}
\end{equation}
This result relies on the Cole--Hopf transformation (\cite{evans1998partial}, Sect. 4.4.1); see Appendix~\ref{sec:pf_thmvisc} for the proof. While this neural network architecture represents, under certain conditions, the solution to the viscous HJ PDE~\eqref{eq:hj_visc_pde_sum}, we note that the particular form of the convex initial data in the HJ PDE~\eqref{eq:hj_visc_pde_sum}, which effectively corresponds to a soft Legendre transform in that $\lim_{\substack{\epsilon\to 0\\ \epsilon>0}} \epsilon\log\left(\sum_{i=1}^{m}e^{\left(\left\langle \bp_{i},\bx\right\rangle -\gamma_{i}\right)/\epsilon}\right) = \max_{i\in\{1,\dots,m\}}\{\left\langle \boldsymbol{p}_{i},\boldsymbol{x}\right\rangle -\gamma_{i}\}$, severely restricts the practicality of this result.}
\end{remark}

\begin{figure}[ht]
\includegraphics[width=\textwidth]{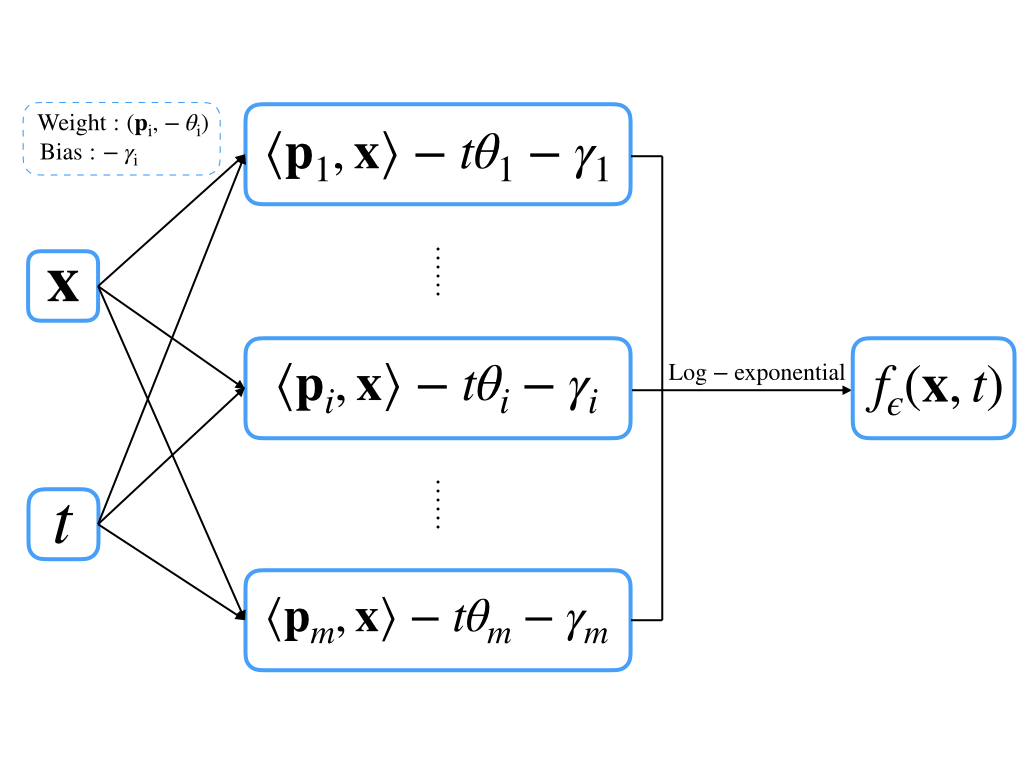}
\caption{Illustration of the structure of the neural network \eqref{eq:log-exponential_network} that represents the solution to a subclass of second-order HJ equations when $\theta_i = -\frac{1}{2}\|\bp_i\|_2^2$ for $i\in \{1,\dots, m\}$. \label{fig: log-exponential_network_figure}}
\end{figure}

\subsection{First-order one-dimensional conservation laws}
\label{sec:conservation}
\begin{figure}[ht]
\includegraphics[width=\textwidth]{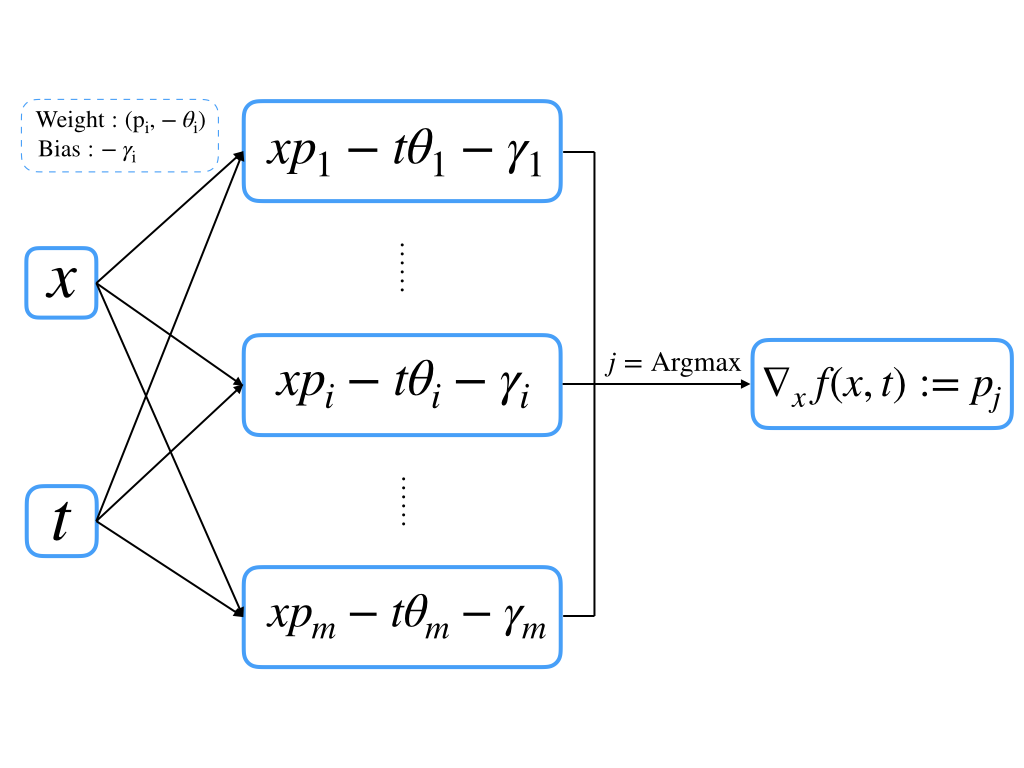}
\caption{Illustration of the structure of the neural network \eqref{eqt:nn_conservation} that can represent the entropy solution to one-dimensional conservation laws.}
\label{fig:nn_argmax}
\end{figure}

It is well-known that one-dimensional conservation laws are related to HJ equations (see, e.g., \cite{Aaibid2006Direct,brenier1986approximate,Brenier1988Discrete, Caselles1992Scalar, Corrias1995Numerical, Jin1998Numerical, Karlsen2002Note, Krukov1967Generalized, Lions1995Convergence}, and also \cite{Dafermos2016Hyperbolic} for a comprehensive introduction to conservation laws and entropy solutions). 
Formally, by taking spatial gradient of the HJ equation \eqref{eqn:intro_example} and identifying the gradient $\nabla_{x}f \equiv u$, we obtain the conservation law
\begin{equation} \label{eqt:conservation}
    \begin{dcases}
    \frac{\partial u}{\partial t}(x,t) + \nabla_{x} H(u(x,t)) =0, & \text{\rm in }\mathbb{R}\times(0,+\infty),\\
    u(x,0) = u_0(x) \coloneqq \nabla J(x), & \text{\rm in }\mathbb{R},
    \end{dcases}
\end{equation}
where the flux function corresponds to the Hamiltonian $H$ in the HJ equation. Here, we assume that the initial data $J$ is convex and globally Lipschitz continuous, and the symbols $\nabla$ and $\nabla_{x}$ in this section correspond to derivatives in the sense of distribution if the classical derivatives do not exist.

In this section, we show that the conservation law derived from the HJ equation \eqref{eqn:intro_example} can be represented by a neural network architecture. Specifically, the corresponding entropy solution $u(x,t) \equiv \nabla_{x}f(x,t)$ to the one-dimensional conservation law \eqref{eqt:conservation} can be represented using a neural network architecture with an argmax based activation function, i.e.,
\begin{equation}\label{eqt:nn_conservation}
    \nabla_{x} f(x,t) = p_j, \text{ where } j \in \argmax_{i\in\{1,\dots,m\}}\{\left\langle p_{i},x\right\rangle -t\theta_{i}-\gamma_{i}\}.
\end{equation}
The structure of this network is shown in Fig. \ref{fig:nn_argmax}. When more than one maximizer exist in the optimization problem above, one can choose any maximizer $j$ and define the value to be $p_j$. We now prove that the function $\nabla_{x}f$ given by the neural network \eqref{eqt:nn_conservation} is indeed the entropy solution to the one-dimensional conservation law \eqref{eqt:conservation} with flux function $H$ and initial data $\nabla J$, where $H$ and $J$ are defined by Eqs. \eqref{eqt:defH} and \eqref{eqt:defJ}, respectively.

\def \um {u^-}
\def \up {u^+}
\def \unon {u_0}
\def \Ncal {\mathcal{N}}
\def \cik {c_i^k}
\def \cikp {c_i^{k+1}}
\def \conek {c_1^k}
\def \conekp {c_1^{k+1}}
\def \cmk {c_m^k}
\def \cmkp {c_m^{k+1}}
\def \ckk {c_k^k}
\def \ckpkp {c_{k+1}^{k+1}}
\def \ak {b_k}
\def \akp {b_{k+1}}
\def \ind {\omega}

\begin{prop} \label{thm:conservation}
Consider the one-dimensional case, i.e., $ n=1$. Suppose the parameters $\{(p_i,\theta_i,\gamma_i)\}_{i=1}^{m}\subset \R \times \R \times \R$ satisfy assumptions (A1)-(A3), and let $u\coloneqq \nabla_{x} f$ be the neural network defined in Eq. \eqref{eqt:nn_conservation} with these parameters. Let $J$ and $H$ be the functions defined in Eqs. \eqref{eqt:defJ} and \eqref{eqt:defH}, respectively, and let $\tilde{H}\colon \R \to \R$ be a locally Lipschitz continuous function. Then the following two statements hold.
\begin{itemize}
    \item[(i)] The neural network $u$ is the entropy solution to the conservation law
    \begin{equation} \label{eqt:conservation_1D}
        \begin{dcases}
        \frac{\partial u}{\partial t}(x,t) + \nabla_{x} H(u(x,t)) =0, & \text{\rm in }\mathbb{R}\times(0,+\infty),\\
        u(x,0) = \nabla J(x), &\text{\rm in }\mathbb{R}.
        \end{dcases}
    \end{equation}
    \item[(ii)] 
    The neural network $u$ is the entropy solution to the conservation law
    \begin{equation}\label{eqt:conservation_Htilde}
    \begin{dcases}
    \frac{\partial u}{\partial t}(x,t) + \nabla_{x} \tilde{H}(u(x,t)) = 0, & \text{\rm in }\mathbb{R}\times(0,+\infty),\\
    u(x,0) = \nabla J(x), &\text{\rm in }\mathbb{R},
    \end{dcases}
\end{equation}
if and only if there exists a constant $C\in\R$ such that $\tilde{H}(p_i)= H(p_i) +C$ for every $i\in\{1,\dots,m\}$ and $\tilde{H}(p) \geqslant H(p)+C$ for any $p\in \conv{\{p_i\}_{i=1}^{m}}$.
\end{itemize}
\end{prop}
\begin{proof}
See Appendix \ref{sec:pf_thmconservation} for the proof.
\end{proof}

\begin{example}
Here, we give one example related to Example \ref{eg:egHJ1}.
Consider $\Jtrue(x) = |x|$ and $\Htrue(p) = -\frac{p^2}{2}$ for every $x,p\in \R$. The entropy solution $u$ to the corresponding one dimensional conservation law is given by 
\[
u(x,t) = \begin{cases}
1 & \text{if } x > 0,\\
-1 & \text{if } x < 0.
\end{cases}
\]
This solution $u$ can be represented using the neural network in Fig. \ref{fig:nn_argmax} with $m=2$, $p_1 = 1$, $p_2 = -1$, $\theta_1 = \theta_2 = -\frac{1}{2}$ and $\gamma_1 = \gamma_2 = 0$. To be specific, we have
\begin{equation*}
u(x)= p_j, \text{ where } j \in \argmax_{i\in \{1,\dots,m\}} \left\{xp_i - t\theta_i - \gamma_i\right\}.
\end{equation*}
The initial data $J$ and  Hamiltonian $H$ defined in Eqs. \eqref{eqt:defJ} and \eqref{eqt:defH} are given by
\begin{equation*}
    \begin{split}
        &J(x) = |x|, \text{ for every } x\in\R;\\
        &H(p) = \begin{cases}
        -\frac{1}{2}, &p\in [-1,1],\\
        +\infty, &\text{otherwise}.
        \end{cases}
    \end{split}
\end{equation*}
By Prop. \ref{thm:conservation}, $u$ solves the one dimensional conservation law \eqref{eqt:conservation_Htilde} if and only if there exists some constant $C\in\R$ such that $\tilde{H}(\pm 1) = -\frac{1}{2} + C$ and $\tilde{H}(p) \geqslant -\frac{1}{2} + C$ for every $p\in (-1,1)$. Note that $\Htrue$ is one candidate satisfying these constraints.
\end{example}

\section{Numerical experiments}
\label{sec:numerical}

\subsection{First-order Hamilton\textendash Jacobi equations}

\revision{In this subsection, we present several numerical experiments to test the effectiveness of the Adam optimizer using our proposed architecture (depicted in Fig.~\ref{fig:nn_max}) for solving some inverse problems.} We focus on the following inverse problem: We are given data samples from a function $S\colon \R^n\times [0,+\infty) \to \R$ that is the viscosity solution to an HJ equation \eqref{eqn:intro_example} with unknown \revision{convex} initial data $J$ and Hamiltonian $H$\revision{, which only depends on $\nabla_{\bx}S(\bx,t)$}.  Our aim is to recover the \revision{convex} initial data $J$. We propose to learn the neural network using machine learning techniques to recover the \revision{convex} initial data $J$. We shall see that this approach also provides partial information on the Hamiltonian $H$.

Specifically, given data samples $\{(\bx_{j}, t_{j}, S(\bx_{j}, t_{j}))\}_{j=1}^{N}$, where $\{(\bx_{j}, t_{j})\}_{j=1}^N \subset \R^n\times [0,+\infty)$, we train the neural network $f$ with structure in Fig. \ref{fig:nn_max} using the mean square loss function defined by

\begin{equation*}
    l(\{(\bp_i, \theta_i, \gamma_i)\}_{i=1}^m) = \frac{1}{N}\sum_{j=1}^N |f(\bx_{j}, t_{j}; \{(\bp_i, \theta_i, \gamma_i)\}_{i=1}^m) - S(\bx_{j}, t_{j})|^2.
\end{equation*}
The training problem is formulated as
\begin{equation}\label{eqt:num_trainingprob}
    \argmin_{\{(\bp_i, \theta_i, \gamma_i)\}_{i=1}^m \subset \R^n\times \R\times \R} l(\{(\bp_i, \theta_i, \gamma_i)\}_{i=1}^m).
\end{equation}
After training, we approximate the initial condition in the HJ equation, denoted by $\tilde{J}$, by evaluating the trained neural network at $t=0$. That is, we approximate the initial condition by 
\begin{equation}\label{eqt:num_defJtilde}
    \tilde{J}\coloneqq f(\cdot, 0).
\end{equation}
In addition, we obtain partial information of the Hamiltonian $H$ using the parameters in the trained neural network via the following procedure. We first detect the effective neurons of the network, which we define to be the affine functions $\{\left\langle \bp_{i},\bx\right\rangle -t\theta_{i}-\gamma_{i}\}$ that contribute to the pointwise maximum in the neural network $f$ (see Eq. \eqref{eqt:deff}). We then denote by $L$ the set of indices that correspond to the parameters of the effective neurons, i.e.,
\begin{equation*}
    L\coloneqq \bigcup_{\bx \in \R^n,\, t\geq 0} \argmax_{i\in \{1,\dots,m\}} \{\left\langle \bp_{i},\bx\right\rangle -t\theta_{i}-\gamma_{i}\},
\end{equation*}
and we finally use each effective parameter $(\bp_l, \theta_l)$ for $l\in L$ to approximate the point $(\bp_l, H(\bp_l))$ on the graph of the Hamiltonian. In practice, we approximate the set $L$ using a large number of points $(\bx,t)$ sampled in the domain $\R^n \times [0,+\infty)$.

\def \pitrue {\bp_i^{true}}
\def \thetaitrue {\theta_i^{true}}
\def \gammaitrue {\gamma_i^{true}}
\subsubsection{Randomly generalized piecewise affine $H$ and $J$} \label{sec:sanity-HJ}

\begin{table}[tbp]
\centering
\caption{Relative mean square errors of the parameters in the neural network $f$ with $2$ neurons in different cases and different dimensions averaged over $100$ repeated experiments.}
\label{tab:sanity-width2}
 \begin{tabular}{c |c ||c c c c } 
 \hline
 \noalign{\smallskip}
 \multicolumn{2}{c||}{\# Case} & Case 1&	Case 2&	Case 3 & Case 4\\
 \noalign{\smallskip}
 \hline
 \noalign{\smallskip}
 \multirow{5}{6.5em}{Averaged Relative Errors of $\{p_i\}$} 
&	2D	&	4.10E-03	&	2.10E-03	&	3.84E-03	&	2.82E-03	\\
&	4D	&	1.41E-09	&	1.20E-09	&	1.38E-09	&	1.29E-09	\\
&	8D	&	1.14E-09	&	1.03E-09	&	1.09E-09	&	1.20E-09	\\
&	16D	&	1.14E-09	&	6.68E-03	&	1.23E-09	&	7.74E-03	\\
&	32D	&	1.49E-09	&	3.73E-01	&	1.46E-03	&	4.00E-01	\\
\noalign{\smallskip}
\hline
\noalign{\smallskip}
\multirow{5}{6.5em}{Averaged Relative Errors of $\{\theta_i\}$} 
&	2D	&	4.82E-02	&	7.31E-02	&	1.17E-01	&	1.79E-01	\\
&	4D	&	3.47E-10	&	2.82E-10	&	1.15E-09	&	1.15E-09	\\
&	8D	&	1.47E-10	&	1.08E-10	&	2.10E-10	&	2.25E-10	\\
&	16D	&	5.44E-11	&	1.69E-03	&	4.75E-11	&	4.12E-03	\\
&	32D	&	3.61E-11	&	3.27E-01	&	6.42E-03	&	2.39E-01	\\
\noalign{\smallskip}
 \hline
\noalign{\smallskip}
\multirow{5}{6.5em}{Averaged Relative Errors of $\{\gamma_i\}$}
&	2D	&	1.35E-02	&	1.01E-01	&	1.33E-02	&	9.24E-02	\\
&	4D	&	3.71E-10	&	1.24E-09	&	3.67E-10	&	1.10E-09	\\
&	8D	&	2.91E-10	&	1.74E-10	&	2.82E-10	&	2.01E-10	\\
&	16D	&	2.80E-10	&	2.08E-04	&	3.10E-10	&	3.20E-04	\\
&	32D	&	3.56E-10	&	1.88E-02	&	1.56E-01	&	3.62E-02	\\
\noalign{\smallskip}
\hline
 \end{tabular}
\end{table}

\begin{table}[tbp]
\centering
 \caption{Relative mean square errors of the parameters in the neural network $f$ with $4$ neurons in different cases and different dimensions averaged over $100$ repeated experiments.}
 \label{tab:sanity-width4}
 \begin{tabular}{c |c ||c c c c} 
 \hline
 \noalign{\smallskip}
 \multicolumn{2}{c||}{\# Case} & Case 1&	Case 2& Case 3 & Case 4\\
 \noalign{\smallskip}
 \hline
 \noalign{\smallskip}
 \multirow{5}{6.5em}{Averaged Relative Errors of $\{p_i\}$} 
&	2D	&	3.12E-01	&	2.21E-01	&	2.85E-01	&	2.14E-01	\\
&	4D	&	7.82E-02	&	6.12E-02	&	7.92E-02	&	4.30E-02	\\
&	8D	&	2.62E-02	&	4.31E-03	&	4.02E-02	&	7.82E-03	\\
&	16D	&	2.88E-02	&	3.64E-02	&	4.35E-02	&	1.73E-02	\\
&	32D	&	1.42E-02	&	3.72E-01	&	1.42E-01	&	5.04E-01	\\
\noalign{\smallskip}
 \hline
\noalign{\smallskip}
\multirow{5}{6.5em}{Averaged Relative Errors of $\{\theta_i\}$} 
&	2D	&	2.59E-01	&	3.68E-01	&	4.82E-01	&	1.34E+00	\\
&	4D	&	6.07E-02	&	8.37E-02	&	9.47E-02	&	1.23E-01	\\
&	8D	&	1.04E-02	&	8.48E-03	&	1.41E-02	&	1.31E-02	\\
&	16D	&	2.66E-03	&	2.53E-02	&	7.80E-03	&	1.90E-02	\\
&	32D	&	8.09E-04	&	4.41E-01	&	1.81E-02	&	3.66E-01	\\
\noalign{\smallskip}
 \hline
\noalign{\smallskip}
 \multirow{5}{6.5em}{Averaged Relative Errors of $\{\gamma_i\}$}
&	2D	&	1.01E-02	&	3.19E-01	&	1.51E-02	&	2.65E-01	\\
&	4D	&	6.72E-03	&	1.79E-02	&	1.03E-02	&	1.30E-02	\\
&	8D	&	3.22E-03	&	2.34E-03	&	3.93E-03	&	2.65E-03	\\
&	16D	&	9.48E-03	&	3.70E-03	&	1.92E-02	&	1.94E-03	\\
&	32D	&	1.33E-02	&	5.35E-02	&	4.73E-01	&	1.17E-01	\\
\noalign{\smallskip}
\hline
 \end{tabular}
\end{table}

In this subsection, we randomly select $m$ parameters $\pitrue$ in $[-1, 1)^n$ for $i\in \{1,\dots, m\}$, and define $\thetaitrue$ and $\gammaitrue$ as follows
\begin{itemize}
    \item[Case 1.] $\thetaitrue = -\|\pitrue\|_2$ and $\gammaitrue = 0$, for $i\in\{1,\dots, m\}$.
    \item[Case 2.] $\thetaitrue = -\|\pitrue\|_2$ and $\gammaitrue = \frac{1}{2}\|\pitrue\|_2^2$, for $i\in\{1,\dots, m\}$.
    \item[Case 3.] $\thetaitrue = -\frac{1}{2}\|\pitrue\|_2^2$ and $\gammaitrue = 0$, for $i\in\{1,\dots, m\}$.
    \item[Case 4.] $\thetaitrue = -\frac{1}{2}\|\pitrue\|_2^2$ and $\gammaitrue = \frac{1}{2}\|\pitrue\|_2^2$, for $i\in\{1,\dots, m\}$.
\end{itemize}
Define the function $S$ as 
\begin{equation*}
    S(\bx,t)\coloneqq \max_{i \in\{1,\dots, m\}} \{\langle \pitrue, \bx\rangle - t\thetaitrue - \gammaitrue\}.
\end{equation*}
By Thm. \ref{thm:constructHJ}, this function $S$ is a viscosity solution to the HJ equations whose Hamiltonian and initial function are the piecewise affine functions defined in Eqs. \eqref{eqt:defH} and \eqref{eqt:defJ}, respectively. 
In other words, $S$ solves the HJ equation with initial data $J$ satisfying
\begin{equation}\label{eqt:num_eqtJ}
\begin{split}
    &J(\bx)\coloneqq
    \max_{i\in\{1,\dots,m\}} \left\langle \pitrue,\boldsymbol{x}\right\rangle, \quad \text{ for Case $1$ and $3$};\\
    &J(\bx)\coloneqq
    \max_{i\in\{1,\dots,m\}}\left\{\left\langle \pitrue,\boldsymbol{x}\right\rangle -\frac{1}{2}\|\pitrue\|^2\right\}, \quad \text{ for Case $2$ and $4$},
\end{split}
\end{equation}
and Hamiltonian $H$ satisfying
\begin{equation*}
\begin{split}
    &H(\bp)\coloneqq \begin{dcases}
    -\max_{\bm{\alpha}\in \mathcal{A}(\bp)} \left\{\sum_{i=1}^m \alpha_i \|\pitrue\|_2\right\},& \text{if }\bp \in \dom J^*,\\
    +\infty, &\text{otherwise},
    \end{dcases}
    \quad\quad \quad \text{ for Case $1$ and $2$;}\\
    &H(\bp)\coloneqq \begin{dcases}
    -\frac{1}{2}\max_{\bm{\alpha}\in \mathcal{A}(\bp)} \left\{\sum_{i=1}^m \alpha_i \|\pitrue\|_2^2\right\},& \text{if }\bp \in \dom J^*,\\
    +\infty, &\text{otherwise},
    \end{dcases}
    \quad \quad \ \text{ for Case $3$ and $4$},
\end{split}
\end{equation*}
where $\mathcal{A}(\bp)$ is the set of maximizers of the corresponding maximization problem in Eq. \eqref{eqt:num_eqtJ}.
Specifically, if we construct a neural network $f$ as shown in Fig. \ref{fig:nn_max} with the underlying parameters $\{(\pitrue, \thetaitrue, \gammaitrue)\}_{i=1}^m$, then the function given by the neural network is exactly the same as the function $S$. In other words, $\{(\pitrue, \thetaitrue, \gammaitrue)\}_{i=1}^m$ is a global minimizer for the training problem \eqref{eqt:num_trainingprob} with the global minimal loss value equal to zero.

Now, we train the neural network $f$ with training data $\{(\bx_{j}, t_{j}, S(\bx_{j}, t_{j}))\}_{j=1}^{N}$, where the points $\{(\bx_{j}, t_{j})\}_{j=1}^N$ are randomly sampled in $\R^n\times [0,+\infty)$ with respect to the standard normal distribution for each $j\in\{1,\dots, N\}$ (we take the absolute value for $t$ to make sure it is non-negative). Here and after, the number of training data points is $N=$ 20,000. We run 60,000 descent steps using the Adam optimizer to train the neural network $f$. The parameters for the Adam optimizer are chosen to be $\beta_1 = 0.5$, $\beta_2 = 0.9$, the learning rate is $10^{-4}$ and the batch size is $500$.

To measure the performance of the training process, we compute the relative mean square errors of the sorted parameters in the trained neural network, denoted by $\{(\bp_i, \theta_i, \gamma_i)\}_{i=1}^m$,  and the sorted underlying true parameters $\{(\pitrue, \thetaitrue, \gammaitrue)\}_{i=1}^m$. To be specific, the errors are computed as follows
\begin{equation*}
\begin{split}
    &\text{relative mean square error of }\{\bp_i\}= \frac{\sum_{i=1}^m \|\bp_i - \pitrue \|_2^2}{\sum_{i=1}^m\|\pitrue\|_2^2},\\
    &\text{relative mean square error of }\{\theta_i\}= \frac{\sum_{i=1}^m |\theta_i - \thetaitrue|^2}{\sum_{i=1}^m |\thetaitrue|^2},\\
    &\text{relative mean square error of }\{\gamma_i\}= \frac{\sum_{i=1}^m |\gamma_i - \gammaitrue|^2}{\sum_{i=1}^m |\gammaitrue|^2}.
    \end{split}
\end{equation*}
For the cases when the denominator $\sum_{i=1}^m |\gammaitrue|^2$ is zero, such as Case $1$ and Case $3$, we measure the absolute mean square error $\frac{1}{m}\sum_{i=1}^m |\gamma_i - \gammaitrue |^2$ instead. 

We test Cases 1--4 on the neural networks with $2$ and $4$ neurons, i.e., we set $m=2,4$ and repeat the experiments $100$ times. We then compute the relative mean square errors in each experiments and take the average. The averaged relative mean square errors are shown in Tabs. \ref{tab:sanity-width2} and \ref{tab:sanity-width4}, respectively.  From the error tables, we observe that the training process performs pretty well and gives errors below $10^{-8}$ in some cases when $m = 2$. However, for the case when $m=4$, we do not obtain the global minimizers and the error is above $10^{-3}$. Therefore, there is no guarantee for the performance of the Adam optimizer in this training problem and it may be related to the complexity of the solution $S$ to the underlying HJ equation.

\subsubsection{Quadratic Hamiltonians} \label{sec:num_1stHJ_quadraticH}
In this subsection, we consider two inverse problems of first-order HJ equations whose Hamiltonians and initial data are defined as follows:
\begin{enumerate}
    \item[1.] $H(\bp) = - \frac{1}{2}\|\bp\|_2^2$ and $J(\bx) = \|\bx\|_1$ for $\bp, \bx\in\R^n$.
    \item[2.] $H(\bp) = \frac{1}{2}\|\bp\|_2^2$ and $J(\bx) = \|\bx\|_1$ for $\bp, \bx\in\R^n$.
\end{enumerate}
The solution to each of the two corresponding HJ equations can be represented using the Hopf formula \cite{Hopf1965} and reads
\begin{enumerate}
    \item[1.] $S(\bx, t) = \|\bx\|_1 + \frac{nt}{2}$ for $\bx\in \R^n$ and $t\geqslant 0$.
    \item[2.] $S(\bx, t) = \sum_{i: |x_i|\geqslant t} \left(|x_i| - \frac{t}{2}\right) + \sum_{i: |x_i|< t}\frac{x_i^2}{2t}$, where $\bx = (x_1,\dots, x_n)\in \R^n$ and $t\geqslant 0$.
\end{enumerate}

We train the neural network $f$ using the same procedure as in the previous subsection and obtain the function $\tilde{J}$ (see Eq. \eqref{eqt:num_defJtilde}) and the parameters $\{(\bp_l, \theta_l)\}_{l\in L}$ associated to the effective neurons.
We compute the relative mean square error of $\tilde{J}$ and $\{(\bp_l, \theta_l)\}_{l\in L}$ as follows
\begin{equation*}
    \begin{split}
        &\text{relative error of }\tilde{J} \coloneqq \frac{\sum_{j=1}^{N^{test}} |\tilde{J}(\bx_i^{test}) - J(\bx_i^{test})|^2 }{\sum_{j=1}^{N^{test}} |J(\bx_i^{test})|^2},\\
        &\text{relative error of }\{(\bp_l,\theta_l)\}_l \coloneqq \frac{\sum_{l\in L} |\theta_l - H(\bp_l)|^2 }{\sum_{l\in L} |H(\bp_l)|^2},\\
    \end{split}
\end{equation*}
where $\{\bx_i^{test}\}$ are randomly sampled with respect to the standard normal distribution in $\R^n$ and there are in total $N^{test} = $ 2,000 testing data points.
We repeat the experiments $100$ times.  The corresponding averaged errors in the two examples are listed in Tabs. \ref{tab:inverse-eg1} and \ref{tab:inverse-eg2}, respectively. 

\begin{table}[tbp]
\centering
 \caption{Relative mean square errors of $\tilde{J}$ and $\{(\bp_l,\theta_l)\}$ for the inverse problems of the first-order HJ equations in different dimensions with $J=\|\cdot\|_1$ and $H= -\frac{1}{2}\|\cdot\|_2^2$, averaged over $100$ repeated experiments.}
 \label{tab:inverse-eg1}
 \begin{tabular}{c |c ||c c c c c } 
 \hline
 \noalign{\smallskip}
 \multicolumn{2}{c||}{\# Neurons} & 64	&	128	&	256	&	512	&
 1024\\
 \noalign{\smallskip}
 \hline
 \noalign{\smallskip}
 \multirow{5}{6.5em}{Averaged Relative Errors of $\tilde{J}$} 
&	1D	&	2.29E-07	&	2.20E-07	&	2.12E-07	&	2.14E-07	&	1.82E-07	\\
&	2D	&	1.49E-06	&	1.27E-06	&	1.16E-06	&	1.01E-06	&	9.25E-07	\\
&	4D	&	6.27E-04	&	1.81E-04	&	5.93E-05	&	1.69E-06	&	3.44E-07	\\
&	8D	&	1.27E-02	&	1.10E-02	&	1.03E-02	&	9.92E-03	&	9.73E-03	\\
&	16D	&	5.69E-02	&	5.83E-02	&	5.96E-02	&	5.99E-02	&	6.01E-02	\\
\noalign{\smallskip}
\hline
\noalign{\smallskip}
\multirow{5}{6.5em}{Averaged Relative Errors of $\{(\bp_l,\theta_l)\}$} 
&	1D	&	2.58E-01	&	1.29E-01	&	7.05E-02	&	3.56E-02	&	1.72E-02	\\
&	2D	&	4.77E-02	&	3.28E-02	&	2.03E-02	&	1.03E-02	&	6.53E-03	\\
&	4D	&	9.36E-03	&	4.09E-03	&	1.58E-03	&	5.31E-04	&	1.73E-04	\\
&	8D	&	3.75E-02	&	3.39E-02	&	3.25E-02	&	2.78E-02	&	2.60E-02	\\
&	16D	&	5.30E-01	&	5.40E-01	&	5.43E-01	&	5.43E-01	&	5.42E-01	\\
\noalign{\smallskip}
 \hline
\noalign{\smallskip}
 \multirow{5}{6.5em}{Averaged Number of Effective Neurons}
&	1D	&	4.45	&	4.37	&	4.18	&	3.92	&	3.55	\\
&	2D	&	8.84	&	8.59	&	7.87	&	7.1	&	6.3	\\
&	4D	&	20.04	&	20.62	&	19.52	&	18.3	&	17.06	\\
&	8D	&	36.97	&	43.91	&	47.84	&	49.19	&	50.03	\\
&	16D	&	48.2	&	59.53	&	64.85	&	65.79	&	64.84	\\
\noalign{\smallskip}
\hline
 \end{tabular}
\end{table}

\begin{table}[tbp]
\centering
 \caption{Relative mean square errors of $\tilde{J}$ and $\{(\bp_l,\theta_l)\}$ for the inverse problems of the first-order HJ equations in different dimensions with $J=\|\cdot\|_1$ and $H= \|\cdot\|_2^2/2$, averaged over $100$ repeated experiments.}
 \label{tab:inverse-eg2}
 \begin{tabular}{c |c ||c c c c c } 
 \hline
 \noalign{\smallskip}
 \multicolumn{2}{c||}{\# Neurons}&64	&	128	&	256	&	512	&
 1024\\
 \noalign{\smallskip}
 \hline
 \noalign{\smallskip}
 \multirow{5}{6.5em}{Averaged Relative Errors of $\tilde{J}$} 
&	1D	&	5.23E-08	&	2.45E-08	&	1.96E-08	&	1.77E-08	&	1.77E-08	\\
&	2D	&	1.75E-05	&	1.67E-05	&	1.77E-05	&	1.85E-05	&	1.91E-05	\\
&	4D	&	5.82E-04	&	4.94E-04	&	5.28E-04	&	5.76E-04	&	6.16E-04	\\
&	8D	&	1.54E-02	&	1.40E-02	&	1.35E-02	&	1.33E-02	&	1.32E-02	\\
&	16D	&	4.19E-02	&	4.33E-02	&	4.43E-02	&	4.46E-02	&	4.49E-02	\\
\noalign{\smallskip}
\hline
\noalign{\smallskip}
\multirow{5}{6.5em}{Averaged Relative Errors of $\{(\bp_l,\theta_l)\}$} 
&	1D	&	3.25E-02	&	1.93E-02	&	1.24E-02	&	5.62E-03	&	2.92E-03	\\
&	2D	&	8.30E-03	&	7.08E-03	&	5.78E-03	&	4.25E-03	&	3.47E-03	\\
&	4D	&	2.41E-02	&	2.41E-02	&	2.51E-02	&	2.65E-02	&	2.82E-02	\\
&	8D	&	7.33E-02	&	7.32E-02	&	7.25E-02	&	7.15E-02	&	7.08E-02	\\
&	16D	&	3.85E-01	&	3.90E-01	&	3.92E-01	&	3.92E-01	&	3.91E-01	\\
\noalign{\smallskip}
 \hline
\noalign{\smallskip}
 \multirow{5}{6.5em}{Averaged Number of Effective Neurons}
&	1D	&	20.26	&	26.94	&	32.26	&	36.02	&	38.61	\\
&	2D	&	32.74	&	48.05	&	65.7	&	84.87	&	99.83	\\
&	4D	&	46.69	&	72.3	&	103.71	&	147.41	&	198.27	\\
&	8D	&	55.55	&	82.04	&	95.46	&	90.82	&	82.5	\\
&	16D	&	61.51	&	99.63	&	119.95	&	118.89	&	109.1	\\
\noalign{\smallskip}
\hline
 \end{tabular}
\end{table}

In the first example, we have $H(\bp) = - \frac{1}{2}\|\bp\|_2^2$ and $J(\bx) = \|\bx\|_1$. According to Thm. \ref{thm:constructHJ}, the solution $S$ can be represented without error by the neural network in Fig. \ref{fig:nn_max} with parameters 
\begin{equation}\label{eqt:num_1stHJ_eg1_param}
    \left\{(\bp,\theta, \gamma)\in \R^n\times \R\times \R: \ \bp(i) \in \{\pm 1\}, \text{ for }i \in \{1,\dots, n\},\ \theta = \frac{n}{2},\ \gamma = 0\right\},
\end{equation}
where $\bp(i)$ denotes the  $i^{\text{th}}$ entry of the vector $\bp$.
In other words, the global minimal loss value in the training problem is theoretically guaranteed to be zero. 
From the numerical errors in Tab. \ref{tab:inverse-eg1}, 
we observe that in low dimension such as 1D and 2D, the errors of the initial function are small. 
However, in most cases, the errors of the parameters are pretty large. 
In the case of $n$ dimension, the viscosity solution can be represented using the $2^n$ parameters in Eq. \eqref{eqt:num_1stHJ_eg1_param}. 
However, the number of effective neurons are larger than $2^n$ in all cases, which also implies that the Adam optimizer does not find the global minimizers in this example.

In the second example, the solution $S$ cannot be represented using our proposed neural network without error. 
Hence the results describes the approximation of the solution $S$ by the neural network.
From Tab. \ref{tab:inverse-eg2}, we  observe 
that the errors become larger when the dimension increases. For this example, the number of effective neurons should be $m$ where $m$ is the number of neurons used in the architecture. Tab. \ref{tab:inverse-eg2} shows that the average number of effective neurons is below this optimal number. Therefore, this implies that the Adam optimizer does not find the global minimizers in this example either.

In conclusion, these numerical experiments suggest that recovering initial data from data samples using our proposed neural network architecture with the Adam optimizer is unsatisfactory for solving these inverse problems. In particular, Adam optimizer is not always able to find a global minimizer when the solution can be represented without error using our network architecture.
\subsection{One-dimensional conservation laws}

\begin{figure}[htbp]
\centering
\subfloat[]{\label{subfig:conservation_eg1}\includegraphics[width = 0.5\textwidth]{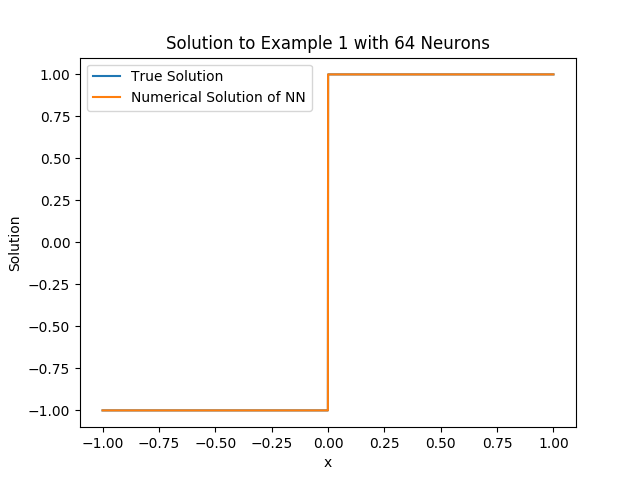}}
\caption{Plot of the function represented by the neural network $\nabla_x f$ at time $t=1$ with 64 neurons whose parameters are defined using $H$ and $J^*$ in example 1. 
The function given by the neural network is plotted in orange and the true solution is plotted in blue.}
\label{fig:conservation-representability1}
\end{figure}

\begin{figure}[htbp]
\centering
\subfloat[]{\label{subfig:conservation_eg2_32}\includegraphics[width = 0.5\textwidth]{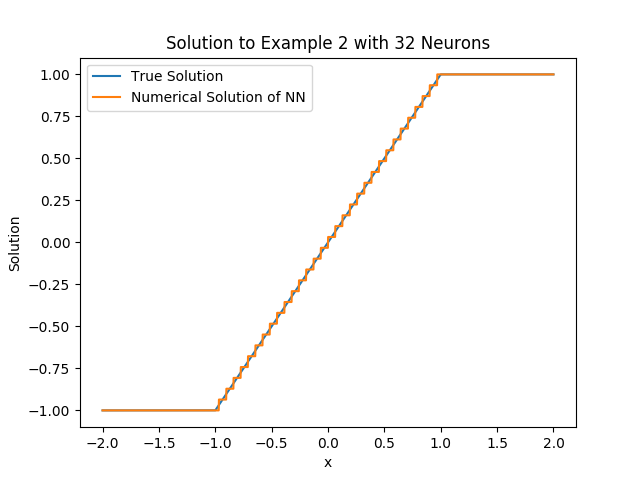}}
\subfloat[]{\label{subfig:conservation_eg2_128}\includegraphics[width = 0.5\textwidth]{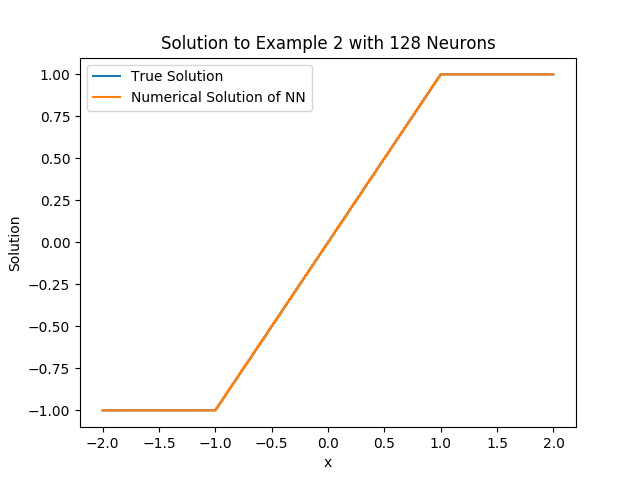}}
\caption{Plot of the function represented by the neural network $\nabla_x f$ at time $t=1$ with 32 and 128 neurons whose parameters are defined using $H$ and $J^*$ in example 2. 
The function given by the neural network is plotted in orange and the true solution is plotted in blue. The neural network with 32 neurons is shown on the left, while the neural network with 128 neurons is shown on the right.}
\label{fig:conservation-representability2}
\end{figure}

In this part, we show the representability of the neural network $\nabla_{x} f$ given in Fig. \ref{fig:nn_argmax} and Eq. \eqref{eqt:nn_conservation}. Since the number of neurons is finite, the function $\nabla_{x} f$ only takes values in the finite set $\{p_i\}_{i=1}^m$. In other words, it can represent the entropy solution $u$ to the PDE \eqref{eqt:conservation} without error only if $u$ takes values in a finite set.

Here, we consider the following two examples 
\begin{itemize}
    \item[1.] $H(p) = -\frac{1}{2}p^2$ and $J(x) = |x|$ for $p,x\in\R$. The initial condition $u_0$ is then given by
    \begin{equation*}
        u_0(x) = \begin{cases}
        1, & x > 0,\\
        -1, & x < 0.
        \end{cases}
    \end{equation*}
    \item[2.] $H(p) = \frac{1}{2}p^2$ and $J(x) = |x|$ for $p,x\in\R$. Hence, the initial function $u_0$ is the same as in example 1.
\end{itemize}
In the first example, the entropy solution $u$ only takes values in the finite set $\{\pm 1\}$, and it can be represented by the neural network $\nabla_x f$ without error by Prop. \ref{thm:conservation}. However, in the second example, the solution $u$ takes values in the infinite set $[-1,1]$, hence the neural network $\nabla_x f$ is only an approximation of the corresponding solution $u$.

To show the representability of the neural network, in each example, we choose the parameters $\{p_i\}_{i=1}^m$ to be the uniform grid points in $[-1,1]$, i.e.,
\begin{equation*}
    p_i = -1 + \frac{2(i-1)}{m-1}, \text{ for } i\in \{1,\dots, m\}.
\end{equation*}
We set $\theta_i = H(p_i)$ and $\gamma_i = J^*(p_i)$ for each $i\in\{1,\dots,m\}$, where $J^*$ is the Fenchel--Legendre transform of the anti-derivative of the initial function $u_0$. Hence, in these two examples, $\gamma_i$ equals for each $i$. Figs. \ref{fig:conservation-representability1} and \ref{fig:conservation-representability2} show the neural network $\nabla_x f$ and the true entropy solution $u$ in these two examples at time $t=1$. 
As expected, the error in Fig. \ref{fig:conservation-representability1} for example 1 is negligible. For example 2, we consider neural networks with 32 and 128 neurons whose graphs are plotted in Figs. \ref{subfig:conservation_eg2_32} and \ref{subfig:conservation_eg2_128}, respectively.
 We observe in these figures
 that the error of the neural networks with the specific parameters decreases as the number of neurons increases. 
In conclusion, the neural network $\nabla_x f$ with the architecture in Fig. \ref{fig:nn_argmax} can represent the solution to the one-dimensional conservation laws given in Eq. \eqref{eqt:conservation} pretty well.
In fact, because of the discontinuity of the activation function, the proposed neural network $\nabla_x f$ has advantages in representing the discontinuity in solution such as shocks, but it requires more neurons when approximating non-constant smooth parts of the solution.
\section{Conclusion}
\label{sec:conclusion}

\noindent
\textbf{Summary of the proposed work.} In this paper, we have established novel mathematical connections between some classes of HJ PDEs with \revision{convex} initial data and neural network architectures. Our \revision{main} results give conditions under which \revision{the neural network architecture illustrated in Fig.~\ref{fig:nn_max}} represents viscosity solutions to HJ PDEs of the form of \eqref{eqn:intro_example}. These results do not rely on universal approximation properties of neural networks; rather, our results show that some neural networks correspond to representation formulas \revision{of solutions to HJ PDEs} whose Hamiltonians and \revision{convex} initial data are obtained from the parameters of the \revision{neural network}. This means that some neural network architectures naturally encode the physics contained in some HJ PDEs \revision{satisfying the conditions in Thm.~\ref{thm:constructHJ}}.

The first neural network architecture that we have proposed is depicted in Fig. \ref{fig:nn_max}. We have shown in Thm. \ref{thm:constructHJ} that under certain conditions on the parameters, this neural network architecture represents the viscosity solution of the HJ PDEs \eqref{eqt:HJ}. The corresponding Hamiltonian and \revision{convex} initial data can be recovered from the parameters of this neural network. As a corollary of this result for the one-dimensional case, we have proposed a second neural network architecture (depicted in Fig. \ref{fig:nn_argmax}) that represents the spatial gradient of the viscosity solution of the HJ PDEs \eqref{eqn:intro_example} (in one dimension), and we have shown in Prop.~\ref{thm:conservation} that under appropriate conditions on the parameters, this neural network corresponds to entropy solutions of the conservation laws \eqref{eqt:conservation_Htilde}.

\revision{Let us emphasize that the neural network architecture depicted in Fig. \ref{fig:nn_max} that represents solutions to the HJ PDEs \eqref{eqt:HJ} allows us to numerically evaluate these solutions in high dimension without using grids or numerical approximations. Our work also paves the way to leverage  efficient technologies and hardware developed for neural networks to compute efficiently solutions to certain HJ PDEs.
}

We have also tested the \revision{performance of the state-of-the-art Adam optimizer using our proposed neural network architecture (depicted in Fig.~\ref{fig:nn_max})} on some inverse problems. Our numerical experiments in Sect. \ref{sec:numerical} show that these problems cannot generally be solved with the Adam optimizer with high accuracy. These numerical results suggest further developments of efficient neural network training algorithms for solving inverse problems with our proposed neural network architectures.

\bigbreak
\noindent
\textbf{Perspectives on other neural network architectures and HJ PDEs.} We now present extensions of the proposed architectures that are viable candidates for representing solutions of HJ PDEs.

First consider the following multi-time HJ PDE \cite{barles2001commutation,cardin2008commuting,darbon2019decomposition,lions1986hopf,motta2006nonsmooth,plaskacz2002oleinik, rochet1985taxation,tho2005hopf}  which reads
\begin{equation} \label{eqt:conclusion-H-J}
\begin{dcases}
\frac{\partial S}{\partial t_j}(\bx, t_1,\dots,t_N) + H_j(\nabla_{\bx} S(\bx, t_1,\dots,t_N)) = 0 \text{ for each }j\in\{1,\dots, N\}, & \text{ in } \R^n \times (0,+\infty)^N, \\
S(\bx,0,\dots,0) = J(\bx), & \text{ in } \R^n.
\end{dcases}
\end{equation}
A generalized Hopf formula \cite{darbon2019decomposition,lions1986hopf,rochet1985taxation} for this multi-time HJ equation is given by
\begin{equation}\label{eqt:conclusion-multitime-Hopf}
S(\bx,t_1, \dots, t_N) = \left(\sum_{i=1}^N t_i H_i + J^* \right)^*(\bx) = \sup_{\bp\in\Rn} \left\{\langle \bp,\bx\rangle - \sum_{j=1}^N t_jH_j(\bp) - J^*(\bp)\right\}, 
\end{equation}
for any $\bx\in\R^n$ and $t_1,\dots, t_N\geqslant 0$.
Based on this formula, we propose a neural network architecture, depicted in Fig. \ref{fig:nn_multitime}, whose mathematical definition is given by 
\begin{equation}\label{eqt:conclusion-nn-multitime}
    f(\bx, t_1,\dots, t_N;\{(\bp_i, \theta_{i1}, \dots, \theta_{iN}, \gamma_i)\}_{i=1}^m) = \max_{i\in \{1,\dots, m\}} \left\{\langle \bp_i, \bx\rangle - \sum_{j=1}^N t_j \theta_{ij} - \gamma_i \right\},
\end{equation}
where $\{(\bp_i, \theta_{i1}, \dots, \theta_{iN}, \gamma_i)\}_{i=1}^m \subset \R^n\times \R^N\times \R$ is the set of parameters. The generalized Hopf formula \eqref{eqt:conclusion-multitime-Hopf} suggests that the neural network architecture depicted in Fig. \ref{fig:nn_multitime} is a good candidate for representing the solution to \eqref{eqt:conclusion-H-J} under appropriate conditions on the parameters of the network. 

\begin{figure}[htbp]
\centering
\includegraphics[width = \textwidth]{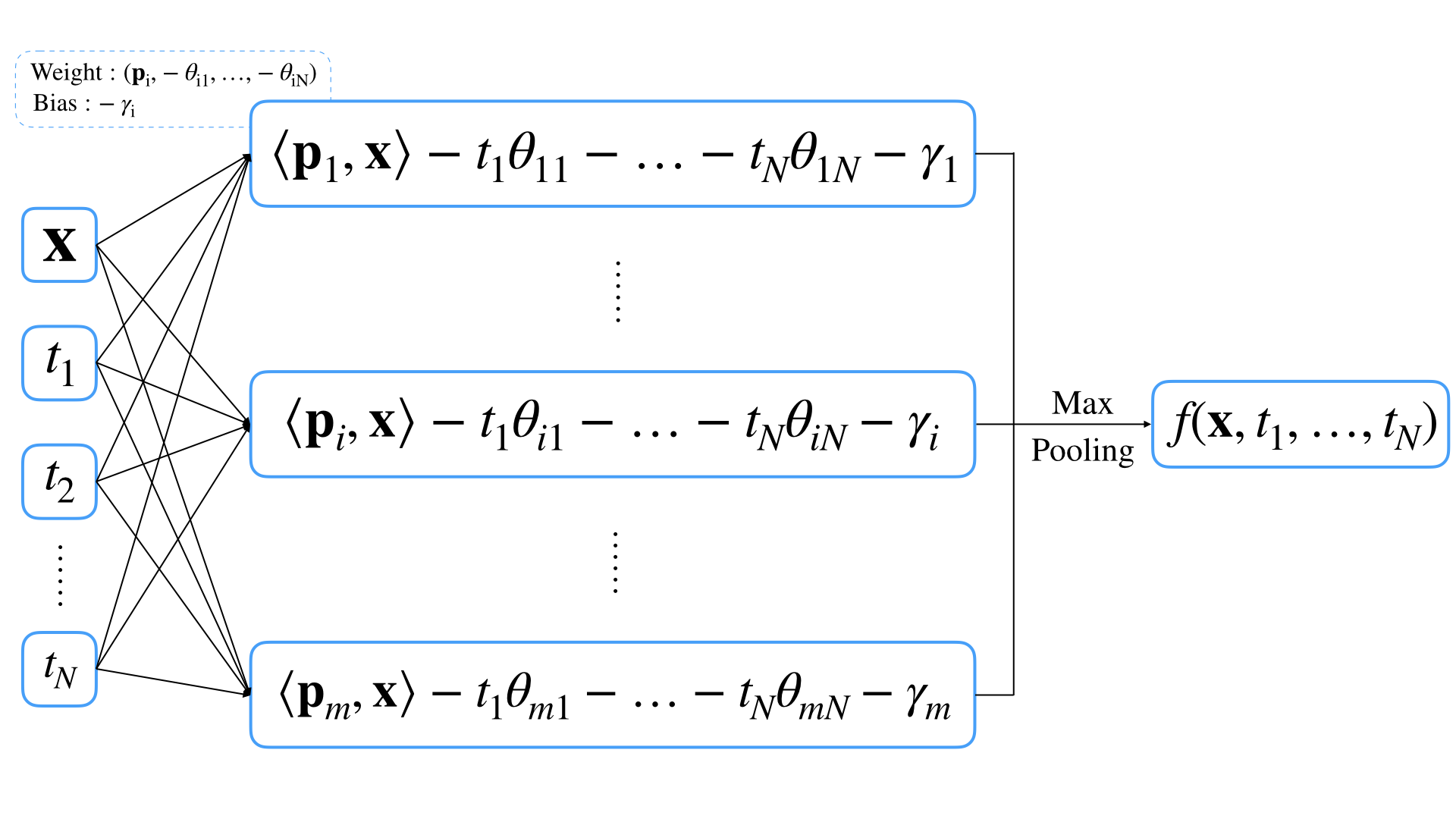}
\caption{
Illustration of the structure of the neural network \eqref{eqt:conclusion-nn-multitime} that can represent  solutions to some first-order multi-time HJ equations.}
\label{fig:nn_multitime}
\end{figure}

As mentioned in \cite{lions1986hopf}, the multi-time HJ equation \eqref{eqt:conclusion-H-J} may not have viscosity solutions. However, under suitable assumptions \cite{barles2001commutation, cardin2008commuting, darbon2019decomposition, motta2006nonsmooth}, the generalized Hopf formula \eqref{eqt:conclusion-multitime-Hopf} is a viscosity solution of the multi-time HJ equation. We intend to clarify the connections between the generalized Hopf formula, multi-time HJ PDEs, viscosity solutions and general solutions in a future work.

\bigbreak

\begin{figure}[htbp]
\centering
\includegraphics[width = \textwidth]{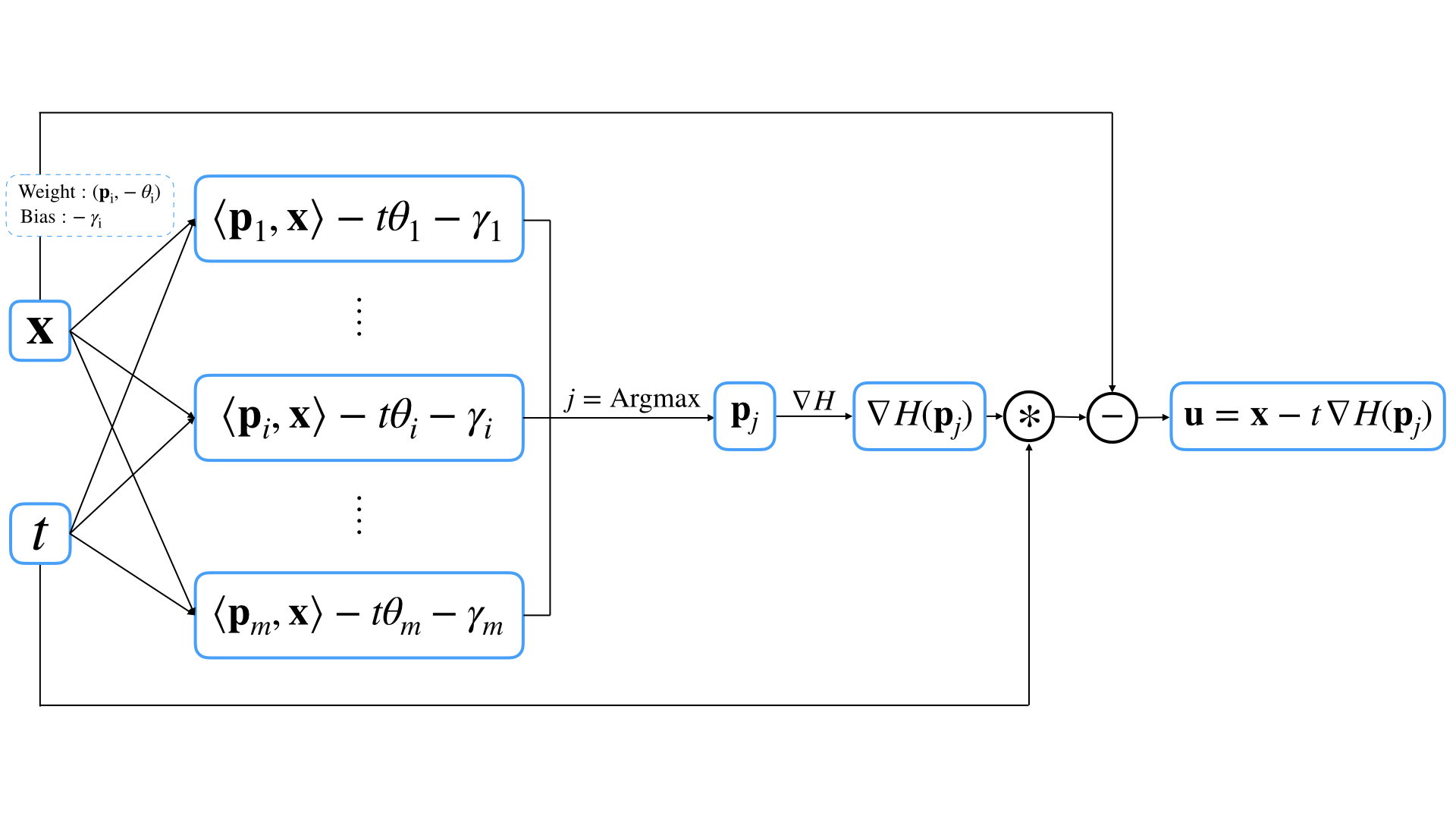}
\caption{Illustration of the structure of the ResNet-type neural network \eqref{eqt:conclusion-nn-resnet} that can represent the minimizer $\bu$ in the Lax-Oleinik formula. Note that the activation function is defined using the gradient of the Hamiltonian $H$, i.e., $\nabla H$.}
\label{fig:nn_resnet}
\end{figure}

\begin{figure}[htbp]
\centering
\includegraphics[width = \textwidth]{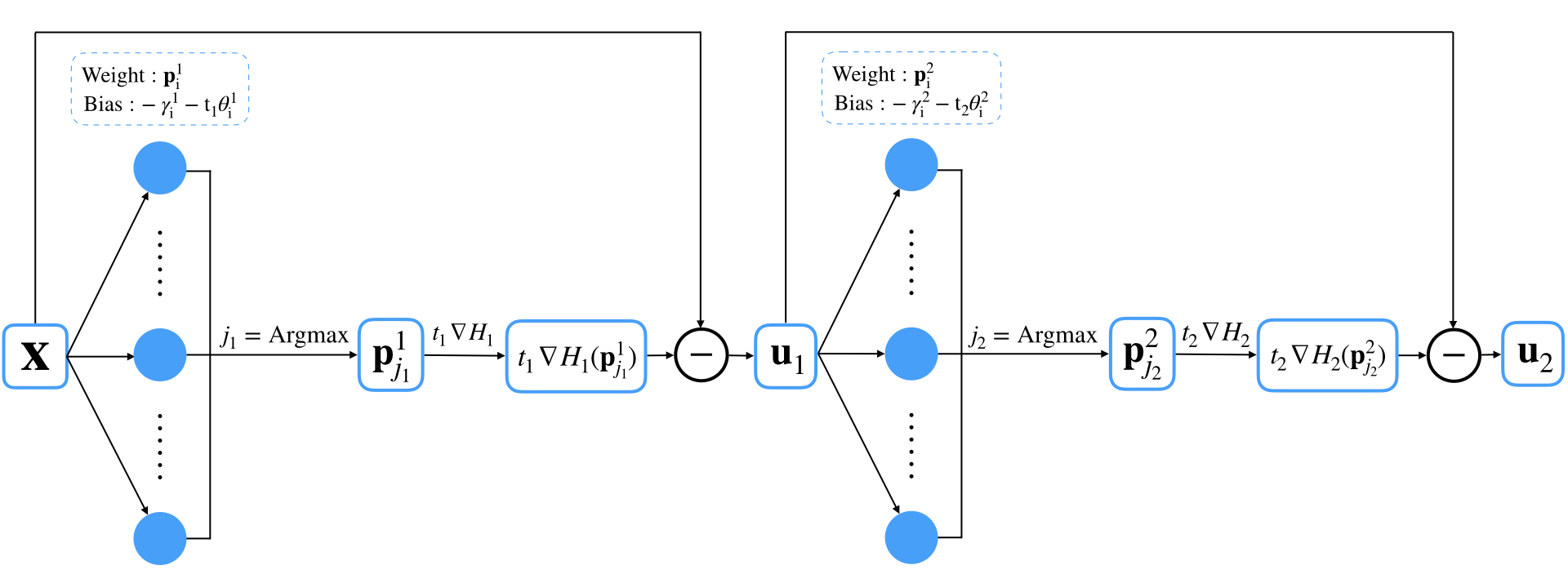}
\caption{Illustration of the structure of the ResNet-type deep neural network \eqref{eqt:conclusion-deepresnet} that can represent the minimizers in the generalized Lax-Oleinik formula for the multi-time HJ PDEs. Note that the activation function in the $k^\text{th}$ layer is defined using the gradient of one Hamiltonian $H_k$, i.e., $\nabla H_k$. This figure only depicts two layers.}
\label{fig:nn_resnet_2layers}
\end{figure}

In \cite{darbon2015convex,darbon2019decomposition}, it is shown that when the Hamiltonian $H$ and the initial data $J$ are both convex, and under appropriate assumptions, the solution $S$ to the following HJ PDE
\begin{equation*}
\begin{dcases} 
\frac{\partial S}{\partial t}(\bx,t)+H(\nabla_{\bx}S(\bx,t))= 0 & \mbox{{\rm in} }\mathbb{R}^{n}\times(0,+\infty),\\
S(\bx,0)=J(\bx) & \mbox{{\rm in} }\mathbb{R}^{n},
\end{dcases}
\end{equation*}
is represented by the Hopf \cite{Hopf1965} and Lax-Oleinik formulas \cite[Sect. 10.3.4]{evans1998partial}. These  formulas read
\begin{align*}
S(\bx,t) &= \max_{\bp\in \R^n}\left\{\langle \bp, \bx\rangle - J^*(\bp) - tH(\bp)\right\} & \text{(Hopf formula)}\\
    &= \min_{\bu\in\R^n} \left\{J(\bu) + t H^*\left(\frac{\bx - \bu}{t}\right)\right\}. & \text{(Lax-Oleinik formula)}
\end{align*}
Let $\bp(\bx,t)$ be the maximizer in the Hopf formula and $\bu(\bx,t)$ be the minimizer in the Lax-Oleinik formula. Then, they satisfy the following relation \cite{darbon2015convex,darbon2019decomposition}
\begin{equation*}
\bu(\bx,t) = \bx - t \nabla H (\bp(\bx,t)).
\end{equation*}
Fig. \ref{fig:nn_resnet} depicts an architecture of a neural network that implements the formula above for the minimizer $\bu(\bx,t)$. In other words, we consider the ResNet-type neural network defined by
\begin{equation}\label{eqt:conclusion-nn-resnet}
    \bu(\bx,t) = \bx - t\nabla H(\bp_j), \text{ where }
    j \in \argmax_{i\in \{1,\dots, m\}} \left\{ \langle \bp_i, \bx\rangle - t\theta_i - \gamma_i \right\}.
\end{equation}
Note that this proposed neural network suggests an interpretation of some ResNet architecture (for details on the ResNet architecture, see \cite{He2016Resnet}) in terms of HJ PDEs. The activation functions of the proposed ResNet architecture is a composition of an argmax based function and $t\nabla H$, where $H$ is the Hamiltonian in the corresponding HJ equation. Moreover, when the time variable is fixed, the input $\bx$ and the output $\bu$ are in the same space $\R^n$, hence one can chain the ResNet structure in Fig. \ref{fig:nn_resnet} to obtain a deep neural network architecture by specifying a sequence of time variables $t_1, t_2, \dots, t_N$. The deep neural network is given by
\begin{equation}\label{eqt:conclusion-deepresnet}
\bu_k = \bu_{k-1} - t_k\nabla H(\bp_{j_k}^k), \quad \text{ for each } k\in \{1,\dots, N\},
\end{equation}
where $\bu_0 = \bx$ and $\bp_{j_k}^k$ is the output of the argmax based activation function in the $k^{\text{th}}$ layer.
For the case when $N=2$, an illustration of this deep ResNet architecture with two layers is shown in Fig. \ref{fig:nn_resnet_2layers}. In fact, this deep ResNet architecture can be formulated as follows
\begin{equation*}
    \bu_N = \bx - \sum_{k=1}^N t_k \nabla H(\bp_{j_k}^k).
\end{equation*}
This formulation suggests that this architecture should also provide the minimizers of the generalized Lax-Oleinik formula for the multi-time HJ PDEs \cite{darbon2019decomposition}. These ideas and perspectives will be presented in detail in a forthcoming paper. 
\appendix

\section{Proofs of lemmas \revision{in Section \ref{subsec:set-up}}}
\subsection{Proof of Lemma \ref{lem:formulaJstar}} \label{sec:prooflem31}
Proof of (i): The convex and lower semicontinuous function $J^*$ satisfies Eq. \eqref{eqt:defJstar} by  \cite[Prop. X.3.4.1]{hiriart2013convexII}. It is also finite and continuous over its polytopal domain $\dom J^* = \conv{\left(\{\bp_i\}_{i=1}^{m}\right)}$ \cite[Thms. 10.2 and 20.5]{rockafellar1970convex}, and moreover the subdifferential $\partial J^*(\bp)$ is non-empty by \cite[Thm. 23.10]{rockafellar1970convex}.

Proof of (ii): First, suppose the vector $(\alpha_1, \dots, \alpha_m)\in\R^m$ satisfies the constraints (a)--(c). Since $\bx\in\partial J^*(\bp)$, there holds $J^*(\bp) = \langle \bp,\bx\rangle - J(\bx)$ \cite[Cor. X.1.4.4]{hiriart2013convexII}, and using the definition of the set $I_{\bx}$ \eqref{eqt:defIx} and constraints (a)--(c) we deduce that
\begin{equation*}
\begin{split}
    J^*(\bp) &= \langle \bp,\bx\rangle - J(\bx)
    = \langle \bp, \bx\rangle - \sum_{i\in I_{\bx}} \alpha_i J(\bx)\\
    &= \langle \bp, \bx\rangle - \sum_{i\in I_{\bx}} \alpha_i (\langle \bp_i, \bx\rangle - \gamma_i)\\
    &= \left\langle \bp - \sum_{i\in I_{\bx}} \alpha_i\bp_i, \bx\right\rangle + \sum_{i\in I_{\bx}} \alpha_i \gamma_i
    = \sum_{i=1}^m \alpha_i\gamma_i.
\end{split}
\end{equation*}
Therefore, $(\alpha_1,\dots, \alpha_m)$ is a minimizer in Eq. \eqref{eqt:defJstar}. Second, let $(\alpha_1,\dots, \alpha_m)$ be a minimizer in Eq. \eqref{eqt:defJstar}. Then (a)--(b) follows directly from the constraints in Eq. \eqref{eqt:defJstar}. A similar argument as above yields
\begin{equation*}
\begin{split}
    J(\bx) &= \langle \bp, \bx\rangle - J^*(\bp)
    = \left\langle \sum_{i=1}^m \alpha_i\bp_i, \bx\right\rangle - \sum_{i=1}^m \alpha_i\gamma_i
    = \sum_{i=1}^m \alpha_i\left(\langle \bp_i, \bx\rangle - \gamma_i\right).
\end{split}
\end{equation*}
But $J(\bx)= \max_{i\in\{1,\dots,m\}}\{\left\langle \boldsymbol{p}_{i},\boldsymbol{x}\right\rangle -\gamma_{i}\}$ by definition, and so there holds $\alpha_i = 0$ whenever $J(\bx)\neq \langle \bp_i, \bx\rangle - \gamma_i$. In other words, $\alpha_i=0$ whenever $i\not\in I_{\bx}$.

Proof of (iii): Let $(\beta_1,\dots, \beta_m) \in \unitsim_m$ satisfy $\sum_{i=1}^m\beta_i \bp_i = \bp_k$. By assumption (A2), we have $\gamma_k=g(\bp_k)$ with $g$ convex, and hence Jensen's inequality yields
\begin{equation*}
    \sum_{i=1}^m\delta_{ik}\gamma_i = \gamma_k = g(\bp_k) = g\left(\sum_{i=1}^m \beta_i \bp_i\right) \leqslant \sum_{i=1}^m \beta_i g(\bp_i) = \sum_{i=1}^m \beta_i \gamma_i.
\end{equation*}
Therefore, the vector $(\delta_{1k},\dots,\delta_{mk})$ is a minimizer in Eq. \eqref{eqt:defJstar} at the point $\bp_k$, and $J^*(\bp_k) = \gamma_k$ follows.
\subsection{Proof of Lemma \ref{lem:Hprop}} \label{sec:prooflem32}

Proof of (i): Let $\bp \in \dom J^*$. The set $\mathcal{A}(\bp) \subseteq \unitsim_m$ is non-empty and bounded by Lem. \ref{lem:formulaJstar}(i), and it is closed since $\mathcal{A}(\bp)$ is the solution set to the linear programming problem \eqref{eqt:defJstar}. Hence, $\mathcal{A}(\bp)$ is compact. As a result, we immediately have that $H(\bp)<+\infty$. Moreover, for each $(\alpha_1,\dots,\alpha_m)\in \mathcal{A}(\bp)$ there holds
\begin{equation*}
-\infty < \min_{i=\{1,\dots, m\}} \theta_i \leqslant \sum_{i=1}^m \alpha_i \theta_i \leqslant \max_{i=\{1,\dots, m\}} \theta_i < +\infty,
\end{equation*} 
from which we conclude that $H$ is a bounded function on $\dom J^*$. Since the target function in the minimization problem \eqref{eqt:defH} is continuous, existence of a minimizer follows by compactness of $\mathcal{A}(\bp)$.

Proof of (ii): We have already shown in the proof of (i) that the restriction of $H$ to $\dom J^*$ is bounded, and so it remains to prove its continuity. For any $\bp\in \dom J^*$, we have that $(\alpha_1,\dots,\alpha_m)\in\mathcal{A}(\bp)$ if and only if $(\alpha_1,\dots, \alpha_m)\in \unitsim_m$, $\sum_{i=1}^m \alpha_i\bp_i = \bp$, and $\sum_{i=1}^m \alpha_i\gamma_i = J^*(\bp)$. As a result, we have
\begin{equation}\label{eqt:proofLem2H}
    H(\bp) = \min \left\{\sum_{i=1}^m\alpha_i \theta_i:\ (\alpha_1,\dots,\alpha_m)\in \unitsim_m, 
    \ \sum_{i=1}^m \alpha_i\bp_i = \bp, \ \sum_{i=1}^m \alpha_i\gamma_i = J^*(\bp) \right\}.
\end{equation}
Define the function $h\colon\ \R^{n+1}\to \R\cup\{+\infty\}$ by 
\begin{equation} \label{eqt:proofLem2h}
\begin{split}
    h(\bp,r) \coloneqq \min \left\{\sum_{i=1}^m\alpha_i \theta_i: (\alpha_1,\dots,\alpha_m)\in \unitsim_m, 
    \ \sum_{i=1}^m \alpha_i\bp_i = \bp, \ \sum_{i=1}^m \alpha_i\gamma_i = r \right\},
\end{split}
\end{equation}
for any $\bp\in \R^n$ and $r\in\R$. Using the same argument as in the proof of Lem. \ref{lem:formulaJstar}(i), we conclude that $h$ is a convex lower semicontinuous function, and in fact continuous over its domain $\dom h = \conv{\{(\bp_i, \gamma_i)\}_{i=1}^m}$. Comparing Eq. \eqref{eqt:proofLem2H} and the definition of $h$ in \eqref{eqt:proofLem2h}, we deduce that $H(\bp) = h(\bp, J^*(\bp))$ for any $\bp\in\dom J^*$. Continuity of $H$ in $\dom J^*$ then follows from the continuity of $h$ and $J^*$ in their own domains.

Proof of (iii): Let $k\in\{1,\dots,m\}$. On the one hand, Lem. \ref{lem:formulaJstar}(iii) implies $(\delta_{1k},\dots, \delta_{mk})\in \mathcal{A}(\bp_k)$, so that
\begin{equation}\label{eqt:proof2Hpk}
H(\bp_k)\leqslant \sum_{i=1}^m \delta_{ik}\theta_i = \theta_k.
\end{equation}
On the other hand, let $(\alpha_1,\dots, \alpha_m)\in\mathcal{A}(\bp_k)$ be a vector different from $(\delta_{k1},\dots, \delta_{km})$. Then $(\alpha_1,\dots, \alpha_m) \in \unitsim_m$ satisfies $\sum_{i=1}^m \alpha_i\bp_i = \bp$, $\sum_{i=1}^m \alpha_i\gamma_i = J^*(\bp)$, and $\alpha_k<1$. Define $(\beta_1,\dots, \beta_m) \in \unitsim_m$ by
\begin{equation*}
    \beta_j \coloneqq \begin{dcases}
    \frac{\alpha_j}{1-\alpha_k}, & \text{if }j\neq k,\\
    0, &\text{if }j=k.
    \end{dcases}
\end{equation*}
A straightforward computation using the properties of $(\alpha_1,\dots, \alpha_m)$, Lem. \ref{lem:formulaJstar}(iii), and the definition of $(\beta_1,\dots, \beta_m)$ yields
\begin{equation*}
    \begin{cases}
    \begin{aligned}
    &(\beta_1,\dots, \beta_m)\in\unitsim_m \text{ with }\beta_k = 0,\\
    &\sum_{i\neq k}\beta_i \bp_i = \sum_{i\neq k}\frac{\alpha_i \bp_i}{1-\alpha_k} = \frac{\bp_k - \alpha_k\bp_k}{1-\alpha_k} = \bp_k,\\
    &\sum_{i\neq k}\beta_i \gamma_i = \sum_{i\neq k}\frac{\alpha_i \gamma_i}{1-\alpha_k} = \frac{J^*(\bp_k) - \alpha_k\gamma_k}{1-\alpha_k}= \frac{\gamma_k - \alpha_k\gamma_k}{1-\alpha_k} =\gamma_k.
    \end{aligned}
    \end{cases}
\end{equation*}
In other words, Eq. \eqref{eqt:assumption3} holds at index $k$, which, by assumption (A3), implies that $\sum_{i\neq k}\beta_i \theta_i > \theta_k$. As a result, we have
\begin{equation*}
    \sum_{i=1}^m\alpha_i \theta_i = \alpha_k\theta_k + (1-\alpha_k)\sum_{i\neq k}\beta_i \theta_i > \alpha_k\theta_k + (1-\alpha_k)\theta_k = \theta_k = \sum_{i=1}^m \delta_{ik}\theta_i.
\end{equation*}
Taken together with Eq. \eqref{eqt:proof2Hpk}, we conclude that $(\delta_{1k},\dots, \delta_{mk})$ is the unique minimizer in \eqref{eqt:defH}, and hence we obtain $H(\bp_k) = \theta_k$.

\section{\revision{Proof of Theorem \ref{thm:constructHJ}}} \label{sec:pf_thmconstructHJ}
\revision{To prove this theorem, we will use three lemmas whose statements and proofs are given in Sect. \ref{sec:prooflem33}, \ref{sec:prooflem34}, and \ref{sec:prooflem35}, respectively. The proof of Thm.~\ref{thm:constructHJ} is given in Sect. \ref{subsec:pf_constructHJ}.}

\subsection{\revision{Statement and proof of Lemma \ref{lem:equivSHandf}}} \label{sec:prooflem33}
\begin{lem}\label{lem:equivSHandf}
Suppose the parameters $\{(\bp_i,\theta_i,\gamma_i)\}_{i=1}^{m}\subset \R^n \times \R \times \R$ satisfy assumptions (A1)-(A3). Let $J$ and $H$ be the functions defined in Eqs. (\ref{eqt:defJ}) and (\ref{eqt:defH}), respectively. Let $\tilde{H}\colon\R^n\to \R$ be a continuous function satisfying $\tilde{H}(\bp_i) = H(\bp_i)$ for each $i\in\{1,\dots,m\}$ and $\tilde{H}(\bp)\geqslant H(\bp)$ for all $\bp\in \dom J^*$. Then the neural network $f$ defined in Eq. (\ref{eqt:deff}) satisfies
\begin{equation}\label{eqt:defSH}
    f(\bx,t) \coloneqq \max_{i\in\{1,\dots,m\}}\{\left\langle \boldsymbol{p}_{i},\boldsymbol{x}\right\rangle -t\theta_{i}-\gamma_{i}\} = \sup_{\bp\in \dom J^*}\left\{\langle \bp,\bx\rangle - t\tilde{H}(\bp) - J^*(\bp)\right\}.
\end{equation}
\end{lem}
\begin{proof}
Let $\bx\in\R^n$ and $t\geqslant 0$. Since $\tilde{H}(\bp)\geqslant H(\bp)$ for every $\bp\in \dom J^*$, we get
\begin{equation}\label{eqt:pflem33_eqt1}
    \langle \bp,\bx\rangle - t\tilde{H}(\bp) - J^*(\bp)
    \leqslant \langle \bp,\bx\rangle - tH(\bp) - J^*(\bp).
\end{equation}
Let $(\alpha_1,\dots, \alpha_m)$ be a minimizer in \eqref{eqt:defH}. By Eqs. \eqref{eqt:defJstar}, \eqref{eqt:defsetA}, and \eqref{eqt:defH}, we have
\begin{equation}\label{eqt:pflem33_eqt2}
    \bp = \sum_{i=1}^{m}\alpha_i\bp_i,\quad H(\bp) = \sum_{i=1}^{m}\alpha_i\theta_i, \quad \text{ and } \quad J^*(\bp)=\sum_{i=1}^{m}\alpha_i\gamma_i.
\end{equation}
Combining Eqs. \eqref{eqt:pflem33_eqt1}, \eqref{eqt:pflem33_eqt2}, and \eqref{eqt:deff}, we get
\begin{equation*}
    \begin{split}
        \langle \bp,\bx\rangle - t\tilde{H}(\bp) - J^*(\bp)
    &\leqslant \sum_{i=1}^m \alpha_i (\langle \bp_i, \bx\rangle -t \theta_i - \gamma_i)\\
    &\leqslant \max_{i\in\{1,\dots,m\}} \{\langle \bp_i, \bx\rangle -t \theta_i - \gamma_i\} = f(\bx,t),
    \end{split}
\end{equation*}
where the second inequality follows from the constraint $(\alpha_1,\dots,\alpha_m)\in\unitsim_m$. Since $\bp \in \dom J^*$ is arbitrary, we obtain
\begin{equation}\label{eqt:SHlessthanf}
    \sup_{\bp\in \dom J^*}\left\{\langle \bp,\bx\rangle - t\tilde{H}(\bp) - J^*(\bp)\right\} \leqslant f(\bx,t). 
\end{equation}

Now, by Lem. \ref{lem:formulaJstar}(iii), Lem. \ref{lem:Hprop}(iii), and the assumptions on $\tilde{H}$, we have
\begin{equation*}
    \tilde{H}(\bp_k)= H(\bp_k) = \theta_k, \quad \text{ and } \quad J^*(\bp_k) = \gamma_k,
\end{equation*}
for each $k\in\{1,\dots, m\}$. A straightforward computation yields
\begin{equation}\label{eqt:SHgeqf}
\begin{split}
    f(\bx,t) &= \max_{i\in\{1,\dots,m\}} \{\langle \bp_i, \bx\rangle -t \theta_i - \gamma_i\} \\
    &= \max_{i\in\{1,\dots,m\}} \left\{\langle \bp_i, \bx\rangle -t \tilde{H}(\bp_i) - J^*(\bp_i)\right\}\\
    &\leqslant \sup_{\bp\in\dom J^*} \left\{\langle \bp, \bx\rangle -t \tilde{H}(\bp) - J^*(\bp)\right\},
\end{split}
\end{equation}
where the inequality holds since $\bp_i\in \dom J^*$ for every $i\in\{1,\dots,m\}$. The conclusion then follows from Eqs. \eqref{eqt:SHlessthanf} and \eqref{eqt:SHgeqf}.
\end{proof}

\subsection{\revision{Statement and proof of Lemma \ref{lem:existdiffpt}}} \label{sec:prooflem34}
\begin{lem}\label{lem:existdiffpt}
Suppose the parameters $\{(\bp_i,\theta_i,\gamma_i)\}_{i=1}^{m}\subset \R^n \times \R \times \R$ satisfy assumptions (A1)-(A3). For every $k\in \{1,\dots, m\}$, there exist $\bx\in\R^n$ and $t>0$ such that $f(\cdot, t)$ is differentiable at $\bx$ and $\nabla_{\bx} f(\bx,t) = \bp_k$.
\end{lem}
\begin{proof}
Since $f$ is the supremum of a finite number of affine functions by definition ($\ref{eqt:deff}$), it is finite-valued and convex for $t\geqslant 0$. As a result, $\nabla_{\bx} f(\bx,t) = \bp_k$ is equivalent to $\partial (f(\cdot, t))(\bx) = \{\bp_k\}$, and so it suffices to prove that $\partial (f(\cdot, t))(\bx) = \{\bp_k\}$ for some $\bx\in\R^n$ and $t>0$. To simplify the notation, we use $\partial_{\bx} f(\bx,t)$ to denote the subdifferential of $f(\cdot, t)$ at $\bx$.

By \cite[Thm. VI.4.4.2]{hiriart2013convexI}, the subdifferential of $f(\cdot, t)$ at $\bx$ is the convex hull of the $\bp_i$'s whose indices $i$'s are maximizers in (\ref{eqt:deff}), that is, 
\begin{equation*}
\partial_{\bx} f(\bx,t) = \co \{\bp_i: i \text{ is a maximizer in (\ref{eqt:deff})}\}.
\end{equation*}
It suffices then to prove the existence of $\bx\in\R^n$ and $t>0$ such that
\begin{equation}\label{eqt:ineq_relation}
    \langle \bp_k,\bx\rangle -t\theta_k - \gamma_k > \langle \bp_i,\bx\rangle -t\theta_i - \gamma_i, \quad \text{ for every $i\neq k$.}
\end{equation}

First, consider the case when there exists $\bx\in\R^n$ such that $\langle \bp_k, \bx\rangle -\gamma_k > \langle \bp_i,\bx\rangle -\gamma_i$ for every $i\neq k$. In that case, by continuity, there exists small $t>0$ such that $\langle \bp_k,\bx\rangle -t\theta_k - \gamma_k > \langle \bp_i,\bx\rangle -t\theta_i - \gamma_i$ for every $i\neq k$ and so (\ref{eqt:ineq_relation}) holds.

Now, consider the case when there does not exist $\bx\in\R^n$ such that $\langle \bp_k, \bx\rangle -\gamma_k > \max_{i\neq k}\{\langle \bp_i,\bx\rangle -\gamma_i\}$. In other words, we assume
\begin{equation}\label{eqt:prooflem4_J}
    J(\bx) = \max_{i\neq k} \{\langle \bp_i,\bx\rangle - \gamma_i\} \text{ for every }\bx\in\R^n.
\end{equation}
We apply Lem. \ref{lem:formulaJstar}(i) to the formula above and obtain 
\begin{equation}\label{eqt:proofprop1_Jstar}
    J^*(\bp_k) = 
\min\left\{\sum_{i=1}^{m}\alpha_{i}\gamma_{i}: (\alpha_{1},\dots,\alpha_{m})\in\unitsim_m,\ 
\sum_{i=1}^{m}\alpha_{i}\bp_{i}=\bp_k,\ 
\alpha_k = 0\right\}.
\end{equation}
Let $\bx_0\in \partial J^*(\bp_k)$.
Denote by $I_{\bx_0}$ the set of maximizers in Eq. (\ref{eqt:prooflem4_J}) at the point $\bx_0$, i.e.,
\begin{equation}\label{eqt:pflem34_defIx0}
    I_{\bx_0}\colon= \argmax_{i\neq k} \{\langle \bp_i,\bx\rangle - \gamma_i\}.
\end{equation}
Note that we have $k\not\in I_{\bx_0}$ by definition of $I_{\bx_0}$.
Define a function $h\colon \R^n\to \R\cup\{+\infty\}$ by
\begin{equation}\label{eqt:proofLem4defh}
    h(\bp) \coloneqq \begin{cases}
    \theta_i, &\text{if }\bp = \bp_i \text{ and }i\in I_{\bx_0},\\
    +\infty, &\text{otherwise}.
    \end{cases}
\end{equation}
Denote the convex lower semicontinuous envelope of $h$ by $\cobar h$. Since $\bx_0\in \partial J^*(\bp_k)$, we can use \cite[Thm. VI.4.4.2]{hiriart2013convexI} and the definition of $I_{\bx_0}$ and $h$ in Eqs. (\ref{eqt:pflem34_defIx0}) and (\ref{eqt:proofLem4defh}) to deduce
\begin{equation}\label{eqt:pflem34_bpkinsubdiff}
\bp_k \in \partial J(\bx_0) = \co \{\bp_i: i\in I_{\bx_0}\} = \dom \cobar h. 
\end{equation}
Hence the point $\bp_k$ is in the domain of the polytopal convex function $\cobar h$. 
Then \cite[Thm. 23.10]{rockafellar1970convex} implies $\partial (\cobar h) (\bp_k)\neq \emptyset$. Let $\bv_0 \in \partial (\cobar h)(\bp_k)$ and $\bx=\bx_0+t\bv_0$. It remains to choose suitable positive $t$ such that (\ref{eqt:ineq_relation}) holds. Letting $\bx=\bx_0+t\bv_0$ in (\ref{eqt:ineq_relation}) yields
\begin{equation}\label{eqt:proofprop1_Ldiff}
    \begin{split}
        &\langle \bp_k,\bx\rangle -t\theta_k - \gamma_k - \left(\langle \bp_i,\bx\rangle -t\theta_i - \gamma_i\right)\\
    =\ &\langle \bp_k,\bx_0+t\bv_0\rangle - t\theta_k -\gamma_k - (\langle \bp_i,\bx_0+t\bv_0\rangle -t\theta_i- \gamma_i)\\
    =\ &\langle \bp_k,\bx_0\rangle -\gamma_k - (\langle \bp_i,\bx_0\rangle - \gamma_i)
    + t(\theta_i-\theta_k - \langle \bp_i-\bp_k, \bv_0\rangle ).
    \end{split}
\end{equation}
Now, we consider two situations, the first when $i\not\in I_{\bx_0} \cup\{k\}$ and the second when $i\in I_{\bx_0}$. It suffices to prove (\ref{eqt:ineq_relation}) hold in each case for small enough positive $t$. 

If $i\not\in I_{\bx_0}\cup\{k\}$, then $i$ is not a maximizer in Eq. (\ref{eqt:prooflem4_J}) at the point $\bx_0$.  By (\ref{eqt:pflem34_bpkinsubdiff}), $\bp_k$ is a convex combination of the set $\{\bp_i: i\in I_{\bx_0}\}$. In other words, there exists $(c_1,\dots,c_m)\in\unitsim_m$ such that $\sum_{j=1}^mc_j\bp_j=\bp_k$ and $c_j=0$ whenever $j\not\in I_{\bx_0}$. Taken together with assumption (A2) and Eqs. (\ref{eqt:defJ}), (\ref{eqt:prooflem4_J}), (\ref{eqt:pflem34_defIx0}), we have
\begin{equation*}
\begin{split}
    J(\bx_0)&\geqslant \langle \bp_k,\bx_0\rangle - \gamma_k 
    = \langle \bp_k,\bx_0\rangle - g(\bp_k)
    = \left\langle \sum_{j\in I_{\bx_0}} c_j\bp_j,\bx_0\right\rangle - g\left(\sum_{j\in I_{\bx_0}} c_j\bp_j\right)\\
    &\geqslant \sum_{j\in I_{\bx_0}} c_j(\langle\bp_j,\bx_0\rangle - g(\bp_j))
    = \sum_{j\in I_{\bx_0}} c_j J(\bx_0) = J(\bx_0).
\end{split}
\end{equation*}
Thus the inequalities become equalities in the equation above. As a result, we have
\begin{equation*}
    \langle \bp_k,\bx_0\rangle - \gamma_k = J(\bx_0) > \langle \bp_i,\bx_0\rangle - \gamma_i,
\end{equation*}
where the inequality holds because $i\not\in I_{\bx_0}\cup\{k\}$ by assumption. This inequality implies that the constant $\langle \bp_k,\bx_0\rangle -\gamma_k - (\langle \bp_i,\bx_0\rangle - \gamma_i)$ is positive, and taken together with (\ref{eqt:proofprop1_Ldiff}), we conclude that the inequality in (\ref{eqt:ineq_relation}) holds for $i\not\in I_{\bx_0}\cup\{k\}$ when $t$ is small enough.

If $i\in I_{\bx_0}$, then both $i$ and $k$ are maximizers in Eq. (\ref{eqt:defJ}) at $\bx_0$, and hence we have
\begin{equation}\label{eqt:proofprop1_5}
\langle \bp_k,\bx_0\rangle - \gamma_k = J(\bx_0) = \langle \bp_i,\bx_0\rangle - \gamma_i.
\end{equation} 
Together with Eq. (\ref{eqt:proofprop1_Ldiff}) and the definition of $h$ in Eq. (\ref{eqt:proofLem4defh}), we obtain
\begin{equation}\label{eqt:pflem53_eqt27}
\begin{split}
    \langle \bp_k,\bx\rangle -t\theta_k - \gamma_k - \left(\langle \bp_i,\bx\rangle -t\theta_i - \gamma_i\right) &= 
    0 + t(h(\bp_i)-\theta_k - \langle \bp_i-\bp_k, \bv_0\rangle)\\
    \ &\geqslant t(\cobar h(\bp_i)-\theta_k - \langle \bp_i-\bp_k, \bv_0\rangle ).
    \end{split}
\end{equation}
In addition, since $\bv_0\in \partial (\cobar h)(\bp_k)$, we have
\begin{equation}\label{eqt:proofprop1_6}
    \cobar h(\bp_i)\geqslant \cobar h(\bp_k) +\langle \bp_i-\bp_k,\bv_0\rangle.
\end{equation}
Combining Eqs. (\ref{eqt:pflem53_eqt27}) and (\ref{eqt:proofprop1_6}), we obtain
\begin{equation}\label{eqt:lem34_eqtLkmLisimple}
    \langle \bp_k,\bx\rangle -t\theta_k - \gamma_k - \left(\langle \bp_i,\bx\rangle -t\theta_i - \gamma_i\right)
    \geqslant t(\cobar h(\bp_k)-\theta_k).
\end{equation}
To prove the result, it suffices to show $\cobar h(\bp_k)>\theta_k$. As $\bp_k \in \cobar h$ (as shown before in Eq. (\ref{eqt:pflem34_bpkinsubdiff})), then according to \cite[Prop. X.1.5.4]{hiriart2013convexII} we have
\begin{equation}\label{eqt:proofprop1_coh}
    \cobar h(\bp_k) = \sum_{j\in I_{\bx_0}} \alpha_j h(\bp_j) = \sum_{j\in I_{\bx_0}} \alpha_j \theta_j,
\end{equation}
for some $(\alpha_1,\dots,\alpha_m)\in \unitsim_m$ satisfying $\bp_k = \sum_{j=1}^m \alpha_j \bp_j$ and $\alpha_j=0$ whenever $j\not\in I_{\bx_0}$. Then, by Lem. \ref{lem:formulaJstar}(ii) $(\alpha_1,\dots,\alpha_m)$ is a minimizer in Eq. (\ref{eqt:proofprop1_Jstar}), that is,
\begin{equation*}
\gamma_k = J^*(\bp_k) =  \sum_{j=1}^m \alpha_j \gamma_j = \sum_{j\in I_{\bx_0}} \alpha_i \gamma_i 
= \sum_{i\neq k}\alpha_i \gamma_i.
\end{equation*}
Hence Eq. (\ref{eqt:assumption3}) holds for the index $k$. By assumption (A3), we have $\theta_k < \sum_{j\neq k}\alpha_j \theta_j$. Taken together with the fact that $\alpha_j=0$ whenever $j\not\in I_{\bx_0}$ and Eq. (\ref{eqt:proofprop1_coh}), we find
\begin{equation}\label{eqt:proofprop1_4}
    \theta_k < \sum_{j\neq k}\alpha_j \theta_j = \sum_{j\in I_{\bx_0}} \alpha_j \theta_j= \cobar h(\bp_k).
\end{equation}
Hence, the right-hand-side of Eq. (\ref{eqt:lem34_eqtLkmLisimple}) is strictly positive, and we conclude that $\langle \bp_k,\bx\rangle -t\theta_k - \gamma_k > \langle \bp_i,\bx\rangle -t\theta_i - \gamma_i$ for $t>0$ if $i\in I_{\bx_0}$.

Therefore, in this case, when $t>0$ is small enough and $\bx$ is chosen as above, we have $\langle \bp_k,\bx\rangle -t\theta_k - \gamma_k > \langle \bp_i,\bx\rangle -t\theta_i - \gamma_i$ for every $i\neq k$, and the proof is complete. 
\end{proof}

\subsection{\revision{Statement and proof of Lemma \ref{lem:formulacoF}}} \label{sec:prooflem35}
\begin{lem}\label{lem:formulacoF}
Suppose the parameters $\{(\bp_i,\theta_i,\gamma_i)\}_{i=1}^{m}\subset \R^n \times \R \times \R$ satisfy assumptions (A1)-(A3). Define a function $F:\ \R^{n+1}\to \R\cup\{+\infty\}$ by
\begin{equation}\label{eqt:lem5defF}
F(\boldsymbol{p},E^{-}) \coloneqq \begin{cases}
J^{*}(\bp), & \text{\rm if \ensuremath{E^{-}+H(\bp)\leqslant0},}\\
+\infty, & {\rm otherwise,}
\end{cases}
\end{equation}
for all $\bp\in\R^n$ and $E^-\in\R$. Then the convex envelope of $F$ is given by
\begin{equation}\label{eqt:lem5coF}
\co F(\bp,E^-) = \inf_{(c_1,\dots, c_m)\in C(\bp,E^-)
}\sum_{i=1}^{m}c_{i}\gamma_{i},
\end{equation}
where the constraint set $C(\bp,E^-)$ is defined by
\begin{equation*}
    C(\bp,E^-)\coloneqq \left\{(c_1, \dots, c_m)\in\unitsim_m\colon
\sum_{i=1}^{m}c_{i}\bp_{i}=\bp, \ 
\sum_{i=1}^{m}c_{i}\theta_{i}\leqslant -E^-
\right\}.
\end{equation*}
\end{lem}
\begin{proof}
First, we compute the convex hull of $\epi F$, which we denote by $\co (\epi F)$.
Let $(\bp, E^-,r)\in \co(\epi F)$, where $\bp\in\R^n$ and $E^-,r\in\R$.
Then there exist $k\in\N$, $(\beta_1, \dots, \beta_k)\in \unitsim_k$ and $(\bq_i, E_i^-,r_i)\in \epi F$ for each $i\in\{1,\dots, k\}$ such that $(\bp, E^-,r) = \sum_{i=1}^k \beta_i (\bq_i, E_i^-,r_i)$. By definition of $F$ in Eq. (\ref{eqt:lem5defF}), $(\bq_i, E_i^-,r_i)\in \epi F$ holds if and only if $\bq_i\in \dom J^*$, $E_i^-+H(\bq_i)\leqslant 0$ and $r_i\geqslant J^*(\bq_i)$. In conclusion, we have
\begin{equation}\label{eqt:lem5eqts1}
    \begin{cases}
    (\beta_1, \dots, \beta_k)\in \unitsim_k,\\
    (\bp, E^-,r) = \sum_{i=1}^k \beta_i (\bq_i, E_i^-,r_i),\\
    \bq_1,\dots,\bq_k\in \dom J^*,\\
    E_i^-+H(\bq_i)\leqslant 0 \text{ for each }i\in\{1,\dots, k\},\\
    r_i\geqslant J^*(\bq_i)  \text{ for each }i\in\{1,\dots, k\}.
    \end{cases}
\end{equation}
For each $i$, since we have $\bq_i\in \dom J^*$, by Lem. \ref{lem:Hprop}(i) the minimization problem in (\ref{eqt:defH}) evaluated at $\bq_i$ has at least one minimizer. Let $(\alpha_{i1}, \dots, \alpha_{im})$ be such a minimizer. Using Eqs. $(\ref{eqt:defJstar})$,  $(\ref{eqt:defH})$, and $(\alpha_{i1}, \dots, \alpha_{im})\in \unitsim_m$, we have
\begin{equation}\label{eqt:lem5eqts2}
    \sum_{j=1}^m \alpha_{ij}(1, \bp_j, \theta_j, \gamma_j) = (1,\bq_i, H(\bq_i), J^*(\bq_i)).
\end{equation}
Define the real number $c_j\coloneqq \sum_{i=1}^k \beta_i \alpha_{ij}$ for any $j\in\{1,\dots,m\}$. Combining Eqs. (\ref{eqt:lem5eqts1}) and (\ref{eqt:lem5eqts2}), we get that $c_j\geqslant 0$ for any $j$ and
\begin{equation*}
    \begin{split}
        &\sum_{j=1}^m c_j(1,\bp_j, \theta_j, \gamma_j) =\sum_{j=1}^m \sum_{i=1}^k \beta_i \alpha_{ij}(1,\bp_j, \theta_j, \gamma_j)\\
         =\ & \sum_{i=1}^k \beta_i \left(\sum_{j=1}^m\alpha_{ij}(1,\bp_j, \theta_j, \gamma_j) \right)
         = \sum_{i=1}^k \beta_i(1,\bq_i, H(\bq_i), J^*(\bq_i)).
    \end{split}
\end{equation*}
We continue the computation using Eq. (\ref{eqt:lem5eqts1}) and get
\begin{equation*}
\begin{split}
    &\sum_{j=1}^m c_j(1,\bp_j) = \sum_{i=1}^k \beta_i(1,\bq_i) = (1,\bp);\\
    &\sum_{j=1}^m c_j\theta_j=\sum_{i=1}^k \beta_iH(\bq_i)\leqslant -\sum_{i=1}^k\beta_iE_i^- = -E^-; \\
    &\sum_{j=1}^m c_j\gamma_j=\sum_{i=1}^k \beta_iJ^*(\bq_i)\leqslant \sum_{i=1}^k\beta_ir_i=r.\\
\end{split}
\end{equation*}
Therefore, we conclude that $(c_1,\dots, c_m)\in \unitsim_m$ and
\begin{equation*}
    \begin{cases}
    \bp = \sum_{j=1}^m c_j\bp_j,\\
    E^- \leqslant -\sum_{j=1}^m c_j\theta_j,\\
    r \geqslant \sum_{j=1}^m c_j\gamma_j.
    \end{cases}
\end{equation*}
As a consequence, $\co (\epi F)\subseteq \co \left(\cup_{j=1}^m \left(\{\bp_j\}\times (-\infty, -\theta_j]\times [\gamma_j, +\infty)\right)\right)$. Now, Lem. \ref{lem:formulaJstar}(iii) and \ref{lem:Hprop}(iii) imply $\{\bp_j\}\times (-\infty, -\theta_j]\times [\gamma_j, +\infty) \subseteq \epi F$ for each $j\in\{1,\dots,m\}$. Therefore, we have
\begin{equation}\label{eqt:lem5coepiF}
\begin{split}
    \co(\epi F) = \Bigg\{(\bp, E^-,r)\in \R^n\times \R\times \R: \text{there exists }(c_1,\dots, c_m)\in\unitsim_m\text{ s.t. }\quad \quad \\
    \bp = \sum_{j=1}^m c_j\bp_j,\ 
    E^- \leqslant -\sum_{j=1}^m c_j\theta_j,
    \ r \geqslant \sum_{j=1}^m c_j\gamma_j.\Bigg\}.
\end{split}
\end{equation}
By \cite[Def. IV.2.5.3 and Prop. IV.2.5.1]{hiriart2013convexII}, we have 
\begin{equation}\label{eqt:lem5coF2}
\begin{split}
\co F(\bp,E^-) = \inf\{r\in \R: (\bp, E^-,r)\in \co(\epi F)\}.
\end{split}
\end{equation}
The conclusion then follows from Eqs. (\ref{eqt:lem5coepiF}) and (\ref{eqt:lem5coF2}).
\end{proof}

\subsection{\revision{Proof of Theorem~\ref{thm:constructHJ}}} \label{subsec:pf_constructHJ}
Proof of (i): First, the neural network $f$ is the pointwise maximum of $m$ affine functions in $(\bx,t)$ and therefore is jointly convex in these variables. Second, as the function $H$ is continuous and bounded in $\dom J^*$ by Lem. \ref{lem:Hprop}(ii), there exists a continuous and bounded function defined in $\R^n$ whose restriction to $\dom J^*$ coincides with $H$ \cite[Thm. 4.16]{folland2013real}. Then statement (i) follows by substituting this function for $\tilde{H}$ in statement (ii), and so it suffices to prove the latter.

Proof of (ii) (sufficiency): Suppose $\tilde{H}(\bp_i) = H(\bp_i)$ for every $i\in\{1,\dots,m\}$ and $\tilde{H}(\bp)\geqslant H(\bp)$ for every $\bp\in \dom J^*$. Since $\tilde{H}$ is continuous on $\R^n$ and $J$ is convex and Lipschitz continuous with Lipschitz constant $L= \max_{i\in\{1,\dots,m\}} \|\bp_i\|$,  \cite[Thm. 3.1]{bardi1984hopf} implies that $(\bx,t) \mapsto \sup_{\bp\in \dom J^*}\left\{\langle \bp,\bx\rangle - t\tilde{H}(\bp) - J^*(\bp)\right\}$ is the unique uniformly continuous viscosity solution to the HJ equation (\ref{eqt:HJ}). But this function is equivalent to the neural network $f$ by Lem. \ref{lem:equivSHandf}, and therefore both sufficiency and statement (i) follow.

Proof of (ii) (necessity): Suppose the neural network $f$ is the unique uniformly continuous viscosity solution to (\ref{eqt:HJ}). First, we prove that $\tilde{H}(\bp_k) = H(\bp_k)$ for every $k\in\{1,\dots, m\}$. Fix $k\in\{1,\dots, m\}$. By Lem. \ref{lem:existdiffpt}, there exist $\bx\in\R^n$ and $t>0$ satisfying $\partial_{\bx} f(\bx,t) = \{\bp_k\}$. Use Lems. \ref{lem:formulaJstar}(iii) and \ref{lem:Hprop}(iii) to write the maximization problem in Eq. (\ref{eqt:deff}) as
\begin{equation}\label{eqt:proofprop1_7}
    f(\bx,t) = \max_{\bp \in \{\bp_1,\dots,\bp_m\}} \{\langle \bp, \bx\rangle - tH(\bp)-J^*(\bp)\},
\end{equation}
where $(\bp,t) \mapsto \langle \bp, \bx\rangle - tH(\bp)-J^*(\bp)$ is continuous in $(\bp,t)$ and differentiable in $t$. As the feasible set $\{\bp_1,\dots,\bp_m\}$ is compact, $f$ is also differentiable with respect to $t$ \cite[Prop. 4.12]{bonnans2013perturbation}, and its derivative equals
\begin{equation*}
    \frac{\partial f}{\partial t}(\bx,t) = \min\left\{-H(\bp):\ \bp \text{ is a maximizer in Eq. (\ref{eqt:proofprop1_7})}\right\}.
\end{equation*}
Since $\bx$ and $t$ satisfy $\partial_{\bx} f(\bx,t)=\{\bp_k\}$, \cite[Thm. VI.4.4.2]{hiriart2013convexI} implies that the only maximizer in Eq. (\ref{eqt:proofprop1_7}) is $\bp_k$. As a result, there holds
\begin{equation} \label{eqt:grad_t_f}
    \frac{\partial f}{\partial t}(\bx,t) = -H(\bp_k).
\end{equation}
Since $f$ is convex on $\R^n$, its subdifferential $\partial f(\bx,t)$ is non-empty and satisfies
\begin{equation*}
 \partial f(\bx,t)\subseteq \partial_{\bx} f(\bx,t)\times \partial_t f(\bx,t) = \{(\bp_k, -H(\bp_k))\}.   
\end{equation*}
In other words, the subdifferential $\partial f(\bx,t)$ contains only one element, and therefore $f$ is differentiable at $(\bx,t)$ and its gradient equals $(\bp_k, -H(\bp_k))$ \cite[Thm. 21.5]{rockafellar1970convex}. Using (\ref{eqt:HJ}) and (\ref{eqt:grad_t_f}), we obtain
\begin{equation*}
    0 = \frac{\partial f}{\partial t}(\bx,t) + \tilde{H}(\nabla_{\bx} f(\bx,t)) = -H(\bp_k) + \tilde{H}(\bp_k).
\end{equation*}
As $k\in\{1,\dots, m\}$ is arbitrary, we find that $H(\bp_k)=\tilde{H}(\bp_k)$ for every $k\in\{1,\dots, m\}$.

Next, we prove by contradiction that $\tilde{H}(\bp)\geqslant H(\bp)$ for every $\bp\in \dom J^*$. It is enough to prove the property only for every $\bp\in \ri \dom J^*$ by continuity of both $\tilde{H}$ and $H$ (where continuity of $H$ is proved in Lem. \ref{lem:Hprop}(ii)). Assume $\tilde{H}(\bp) < H(\bp)$ for some $\bp\in \ri \dom J^*$. Define two functions $F$ and $\tilde{F}$ from $\R^n\times \R$ to $\R\cup\{+\infty\}$ by
\begin{equation}\label{eqt:defF}
F(\boldsymbol{q},E^{-})\coloneqq\begin{cases}
J^{*}(\boldsymbol{q}), & \text{if \ensuremath{E^{-}+H(\boldsymbol{q})\leqslant0},}\\
+\infty, & {\rm otherwise.}
\end{cases}
\quad \text{ and }\quad 
\tilde{F}(\boldsymbol{q},E^{-})\coloneqq\begin{cases}
J^{*}(\boldsymbol{q}), & \text{if \ensuremath{E^{-}+\tilde{H}(\boldsymbol{q})\leqslant0},}\\
+\infty, & {\rm otherwise.}
\end{cases}
\end{equation}
for any $\bq\in\R^n$ and $E^-\in\R$. 
Denoting the convex envelope of $F$ by $\co F$, Lem. \ref{lem:formulacoF} implies
\begin{equation}\label{eqt:defGinprop1}
\begin{split}
    &\co F(\bq,E^-) = \inf_{(c_1,\dots, c_m)\in C(\bq,E^-)
}\sum_{i=1}^{m}c_{i}\gamma_{i},\text{ where }C \text{ is defined by}\\
    &C(\bq,E^-)\coloneqq \left\{(c_1, \dots, c_m)\in\unitsim_m:\ 
\sum_{i=1}^{m}c_{i}\bp_{i}=\bq, \ 
\sum_{i=1}^{m}c_{i}\theta_{i}\leqslant -E^-
\right\}.
\end{split}
\end{equation}
Let $E_1^- \in \left(-H(\bp), -\tilde{H}(\bp)\right)$.
Now, we want to prove that $\co F(\bp, E_1^-)\leqslant J^*(\bp)$; this inequality will lead to a contradiction with the definition of $H$.

Using statement (i) of this theorem and the supposition that $f$ is the unique viscosity solution to the HJ equation (\ref{eqt:HJ}), we have that
\begin{equation*}
    f(\bx, t) = \sup_{\bq\in\R^n} \{\langle \bq,\bx\rangle - tH(\bq) - J^*(\bq)\} = \sup_{\bq\in\R^n} \{\langle \bq,\bx\rangle - t\tilde{H}(\bq) - J^*(\bq)\}. 
\end{equation*}
Furthermore, a similar calculation as in the proof of \cite[Prop. 3.1]{darbon2019decomposition} yields
\begin{equation*}
     f = F^* = \tilde{F}^*, \text{ which implies } f^* =  \cobar F = \cobar \tilde{F}.
\end{equation*}
where $\cobar F$ and $\cobar \tilde{F}$ denotes the convex lower semicontinuous envelopes of $F$ and $\tilde{F}$, respectively. On the one hand, since $f^* = \cobar \tilde{F}$, the definition of $\tilde{F}$ in Eq. (\ref{eqt:defF}) implies
\begin{equation}\label{eqt:pfthm31_Ftineqt1}
    f^*\left(\bp, -\tilde{H}(\bp)\right)\leqslant \tilde{F}\left(\bp, -\tilde{H}(\bp)\right) = J^*(\bp)
    \quad
    \text{ and }
    \quad 
    \{\bp\}\times \left(-\infty, -\tilde{H}(\bp)\right]\subseteq \dom \tilde{F} \subseteq \dom f^*.    
\end{equation}
Recall that $\bp\in\ri\dom J^*$ and $E_1^-<-\tilde{H}(\bp)$, so that $(\bp, E_1^-) \in \ri \dom f^*$. As a result, we get
\begin{equation}\label{eqt:pfthm31_rifstar}
    \left(\bp, \alpha E_1^- +(1-\alpha)(-\tilde{H}(\bp))\right)\in \ri \dom f^* \text{ for all }\alpha\in(0,1).
\end{equation} On the other hand, since $f^* = \co F$, we have $\ri\dom f^* = \ri\dom (\co F)$ and $f^* = \co F$ in $\ri\dom f^*$. Taken together with Eq. (\ref{eqt:pfthm31_rifstar}) and the continuity of $f^*$, there holds
\begin{equation}\label{eqt:pfthm31_ineqtfstar2}
\begin{split}
    f^*\left(\bp, -\tilde{H}(\bp)\right) &= \lim_{\substack{\alpha \to 0\\ 0<\alpha <1}} f^*\left(\bp, \alpha E_1^- +(1-\alpha)(-\tilde{H}(\bp))\right)\\
    &= \lim_{\substack{\alpha \to 0\\ 0<\alpha <1}} \co F\left(\bp, \alpha E_1^- +(1-\alpha)(-\tilde{H}(\bp))\right).
\end{split}
\end{equation}
Note that $\co F(\bp,\cdot)$ is monotone non-decreasing. Indeed, if $E_2^-$ is a real number such that $E_2^- > E_1^-$, by the definition of the set $C$ in Eq. (\ref{eqt:defGinprop1}) there holds $C(\bp, E_2^-)\subseteq C(\bp, E_1^-)$, which implies 
$\co F(\bp,E_2^-)\geqslant \co F(\bp,E_1^-)$.
Recalling that $E_1^- < -\tilde{H}(\bp)$, monotonicity of $\co F(\bp,\cdot)$ and Eq. (\ref{eqt:pfthm31_ineqtfstar2}) imply
\begin{equation}\label{eqt:pfthm31_ineqtfstar3}
\begin{split}
    f^*\left(\bp, -\tilde{H}(\bp)\right) 
    &\geqslant \lim_{\substack{\alpha \to 0\\ 0<\alpha <1}} \co F\left(\bp, \alpha E_1^- +(1-\alpha)E_1^-\right) 
    = \co F(\bp, E_1^-).
\end{split}
\end{equation}
Combining Eqs. (\ref{eqt:pfthm31_Ftineqt1}) and (\ref{eqt:pfthm31_ineqtfstar3}), we get 
\begin{equation}\label{eqt:proofprop1_1}
\co F(\bp, E_1^-)\leqslant J^*(\bp) < +\infty.   
\end{equation}

As a result, the set $C(\bp, E_1^-)$ is non-empty. Since it is also compact, there exists a minimizer in Eq. (\ref{eqt:defGinprop1}) evaluated at the point $(\bp, E_1^-)$. Let $(c_1,\dots, c_m)$ be such a minimizer. By Eqs. (\ref{eqt:defGinprop1}) and (\ref{eqt:proofprop1_1}) and the assumption that $E_1^- \in \left(-H(\bp),-\tilde{H}(\bp)\right)$, there holds
\begin{equation}\label{eqt:proofprop1eqts}
    \begin{cases}
    (c_1, \dots, c_m)\in\unitsim_m,\\
    \sum_{i=1}^{m}c_{i}\bp_{i}=\bp,\\
    \sum_{i=1}^m c_i\gamma_i = \co F(\bp,E_1^-) \leqslant J^*(\bp),\\
    \sum_{i=1}^{m}c_{i}\theta_{i}\leqslant -E_1^- < H(\bp).
    \end{cases}
\end{equation}
Comparing the first three statements in Eq. (\ref{eqt:proofprop1eqts}) and the formula of $J^*$ in Eq. (\ref{eqt:defJstar}), we deduce that $(c_1,\dots, c_m)$ is a minimizer in Eq. (\ref{eqt:defJstar}), i.e., $(c_1,\dots, c_m)\in\mathcal{A}(\bp)$. By definition of $H$ in Eq. (\ref{eqt:defH}), we have
\begin{equation*}
    H(\bp) =\inf_{\bm{\alpha}\in \mathcal{A}(\bp)} \sum_{i=1}^m \alpha_i \theta_i \leqslant \sum_{i=1}^m c_i\theta_i,
\end{equation*}
which contradicts the last inequality in Eq. (\ref{eqt:proofprop1eqts}). Therefore, we conclude that $\tilde{H}(\bp)\geqslant H(\bp)$ for any $\bp\in\ri \dom J^*$ and the proof is finished.

\section{\revision{Connections between the neural network \eqref{eq:log-exponential_network} and the viscous HJ PDE \eqref{eq:hj_visc_pde_sum}}}
\label{sec:pf_thmvisc}

\revision{Let $f_\epsilon$ be the neural network defined by Eq.~\eqref{eq:log-exponential_network} with parameters $\{(\bp_{i}, \theta_i, \gamma_{i})\}_{i=1}^m$ and $\epsilon > 0$, which is illustrated in Fig.~\ref{fig: log-exponential_network_figure}. We will show in this appendix that when the parameter $\theta_{i} = -\frac{1}{2}\left\Vert \bp_{i}\right\Vert _{2}^{2}$ for $i\in\{1,\dots,m\}$, then the neural network $f_\epsilon$ corresponds to the unique, jointly convex smooth solution to the viscous HJ PDE~\eqref{eq:hj_visc_pde_sum}. This result will follow immediately from the following lemma.
\begin{lem}
\label{lem:visc}
Let $\{(\bp_i,\gamma_i)\}_{i=1}^{m}\subset \R^n \times \R$ and $\epsilon > 0$. Then the function $w_{\epsilon} : \R^n \mapsto \R$ defined by
\begin{equation} \label{eq:sum_heat_eq}
w_{\epsilon}(\bx,t)\coloneqq\sum_{i=1}^{m}e^{\left(\left\langle \bp_{i},\bx\right\rangle +\frac{t}{2}\left\Vert \bp_{i}\right\Vert _{2}^{2}-\gamma_{i}\right)/\epsilon}
\end{equation}
is the unique, jointly log-convex and smooth solution to the Cauchy problem
\begin{equation} \label{eq:heat_problem}
\begin{dcases}
\frac{\partial w_{\epsilon}}{\partial t}(\boldsymbol{x},t)=\frac{\epsilon}{2}\Delta_{\boldsymbol{x}}w_{\epsilon}(\boldsymbol{x},t) & \text{\rm in } \mathbb{R}^{n}\times(0,+\infty),\\
w_{\epsilon}(\boldsymbol{x},0)=\sum_{i=1}^{m}e^{\left(\left\langle \bp_{i},\bx\right\rangle -\gamma_{i}\right)/\epsilon} & \text{\rm in }\mathbb{R}^{n}. 
\end{dcases}
\end{equation}
\end{lem}
\begin{proof}
A short calculation shows that the function $w_\epsilon$ defined in Eq. (\ref{eq:sum_heat_eq}) solves the Cauchy problem (\ref{eq:heat_problem}), and uniqueness holds by strict positiveness of the initial data (see \cite[Chap. VIII, Thm. 2.2]{widder1976heat}, and note that the uniqueness result can easily be generalized to $n > 1$). 

Now, let $\lambda\in[0,1]$ and $(\bx_{1},t_{1})$ and $(\bx_{2},t_{2})$
be such that $\bx=\lambda\bx_{1}+(1-\lambda)\bx_{2}$ and $t=\lambda t_{1}+(1-\lambda)t_{2}$.
Then the H{\"o}lder's inequality (see, e.g., \cite[Thm. 6.2]{folland2013real}) implies
\begin{alignat*}{1}
\sum_{i=1}^{m}e^{\left(\left\langle \bp_{i},\bx\right\rangle +\frac{t}{2}\left\Vert \bp_{i}\right\Vert _{2}^{2}-\gamma_{i}\right)/\epsilon}  = & \sum_{i=1}^{m}\left(e^{\lambda\left(\left\langle \bp_{i},\bx_{1}\right\rangle +\frac{t_{1}}{2}\left\Vert \bp_{i}\right\Vert _{2}^{2}-\gamma_{i}\right)/\epsilon}e^{(1-\lambda)\left(\left\langle \bp_{i},\bx_{2}\right\rangle +\frac{t_{2}}{2}\left\Vert \bp_{i}\right\Vert _{2}^{2}-\gamma_{i}\right)/\epsilon}\right)\\
\leqslant & \left(\sum_{i=1}^{m}e^{\left(\left\langle \bp_{i},\bx_{1}\right\rangle +\frac{t_{1}}{2}\left\Vert \bp_{i}\right\Vert _{2}^{2}-\gamma_{i}\right)/\epsilon}\right)^{\lambda}\left(\sum_{i=1}^{m}e^{\left(\left\langle \bp_{i},\bx_{2}\right\rangle +\frac{t_{2}}{2}\left\Vert \bp_{i}\right\Vert _{2}^{2}-\gamma_{i}\right)/\epsilon}\right)^{1-\lambda},
\end{alignat*}
and we find $w_{\epsilon}(\bx,t)\leqslant\left(w_{\epsilon}(\bx_{1},t_{1})\right)^{\lambda}\left(w_{\epsilon}(\bx_{2},t_{2})\right)^{1-\lambda}$, which implies that $w_\epsilon$ is jointly log-convex in $(\bx,t)$.
\end{proof}
Thanks to Lemma \ref{lem:visc} and the Cole--Hopf transformation $f_\epsilon(\bx,t) = \epsilon\log\left(w_\epsilon(\bx,t)\right)$ (see, e.g., \cite{evans1998partial}, Sect. 4.4.1), a short calculation immediately implies that the neural network $f_\epsilon$ solves the viscous HJ PDE \eqref{eq:hj_visc_pde_sum}, and it is also its unique solution because $w_{\epsilon}$ is the unique solution to the Cauchy problem \eqref{eq:heat_problem}. Joint convexity in $(\bx,t)$ follows from log-convexity of $(\bx,t)\mapsto w_{\epsilon}(\bx,t)$ for every $\epsilon>0$.}

\section{Proof of Proposition \ref{thm:conservation}} \label{sec:pf_thmconservation}
To prove this proposition, we will use three lemmas whose statements and proofs are given in Sect. \ref{sec:prooflem51}, \ref{sec:prooflem52}, and \ref{sec:prooflem53}, respectively. The proof of Prop. \ref{thm:conservation} is given in Sect. \ref{subsec:pf_thmconservation}.

\subsection{Statement and proof of Lemma \ref{lem:53}}\label{sec:prooflem51}
\begin{lem}\label{lem:53}
Consider the one-dimensional case, i.e., $n=1$. Let $p_1,\dots,p_m\in\R$ satisfy $p_1<\dots<p_m$, and define the function $J$ using Eq. (\ref{eqt:defJ}). Suppose assumptions (A1)-(A2) hold. Let $x\in\R$, $p\in \partial J(x)$, and suppose $p\neq p_i$ for any $i\in\{1,\dots,m\}$. Then there exists $k\in\{1,\dots,m\}$ such that $p_k < p < p_{k+1}$ and
\begin{equation}\label{eqt:lem53_conclu}
    k,k+1 \in \argmax_{i\in\{1,\dots,m\}}\{xp_{i} -\gamma_{i}\}.
\end{equation}
\end{lem}
\begin{proof}
Let $I_x$ denotes the set of maximizers in Eq. (\ref{eqt:defIx}) at $x$. Since $p\in \partial J(x)$, $p \neq p_i$ for $i \in \{1,\dots,m\}$, and $\partial J(x) = \co \{p_i: i\in I_i\}$ by \cite[Thm. VI.4.4.2]{hiriart2013convexI}, there exist $j,l\in I_x$ such that $p_j < p < p_l$. Moreover, there exists $k$ with $j\leqslant k<k+1 \leqslant l$ such that $p_j \leqslant p_k < p <p_{k+1}\leqslant p_l$. We will show that $k,k+1 \in I_x$. We only prove $k\in I_x$; the case for $k+1$ is similar.

If $p_j=p_k$, then $k=j\in I_x$ and the conclusion follows directly. Hence suppose $p_j < p_k < p_l$. Then there exists $\alpha\in (0,1)$ such that $p_k = \alpha p_j + (1-\alpha) p_l$. Using that $j,l\in I_x$, assumption (A2), and Jensen inequality, we get
\begin{equation*}
    \begin{split}
        xp_k - \gamma_k &= xp_k - g(p_k) 
        = (\alpha p_j + (1-\alpha)p_l)x - g(\alpha p_j + (1-\alpha)p_l)\\
        &\geqslant \alpha xp_j + (1-\alpha)xp_l - \alpha g(p_j) - (1-\alpha)g(p_l)\\
        &= \alpha(xp_j - \gamma_j) + (1-\alpha)(xp_l - \gamma_l)\\
        &= \max_{i\in\{1,\dots,m\}}\{xp_{i} -\gamma_{i}\},
    \end{split}
\end{equation*}
which implies that $k \in I_x$. A similar argument shows that $k+1\in I_x$, which completes the proof.
\end{proof}

\subsection{Statement and proof of Lemma \ref{lem:52}} \label{sec:prooflem52}
\begin{lem} \label{lem:52}
Consider the one-dimensional case, i.e., $n=1$. Let $p_1,\dots,p_m\in\R$ satisfy $p_1<\dots<p_m$, and define the function $H$ using Eq. (\ref{eqt:defH}). Suppose assumptions (A1)-(A3) hold. Let $\unon \in \R$ and $p_k < \unon < p_{k+1}$ for some index $k$. Then there holds
\begin{equation}\label{eqt:pfcons_Hform}
    H(\unon) = \beta_k\theta_k + \beta_{k+1} \theta_{k+1},
\end{equation}
where 
\begin{equation}\label{eqt:lem52_defbetak}
    \beta_k\coloneqq \frac{p_{k+1} - \unon}{p_{k+1}-p_k} \quad \text{ and } \quad
    \beta_{k+1} \coloneqq \frac{\unon - p_k}{p_{k+1}-p_k}.
\end{equation}
\end{lem}
\begin{proof}
Let $\bm{\beta}\coloneqq(\beta_1,\dots,\beta_m) \in \unitsim_{m}$ satisfy
\begin{equation*}
    \beta_k\coloneqq \frac{p_{k+1} - \unon}{p_{k+1}-p_k} \quad \text{ and } \quad
    \beta_{k+1} \coloneqq \frac{\unon - p_k}{p_{k+1}-p_k},
\end{equation*}
and $\beta_i = 0$ for every $i\in \{1,\dots,m\}\setminus\{k,k+1\}$. We will prove that $\bm{\beta}$ is a minimizer in Eq. (\ref{eqt:defH}) evaluated at $\unon$, that is,
\begin{equation*}
    \bm{\beta} \in \argmin_{\bm{\alpha}\in \mathcal{A}(\unon)} \left\{\sum_{i=1}^m \alpha_i \theta_i\right\},
\end{equation*}
where
\begin{equation*}
\mathcal{A}(\unon) \coloneqq
\argmin_{\substack{(\alpha_{1},\dots\alpha_{m})\in \unitsim_m\\
\sum_{i=1}^{m}\alpha_{i}p_{i}=\unon
}
}\left\{\sum_{i=1}^{m}\alpha_{i}\gamma_{i}\right\}.
\end{equation*}

First, we show that $\bm{\beta} \in \mathcal{A}(\unon)$. By definition of $\bm{\beta}$ and Lem. \ref{lem:formulaJstar}(ii) with $p = \unon$, the statement holds provided $k,k+1\in I_x$, where the set $I_x$ contains the maximizers in Eq. (\ref{eqt:defJ}) evaluated at $x \in \partial J^*(\unon)$. But if $x\in \partial J^*(\unon)$, we have $\unon\in \partial J(x)$, and Lem. \ref{lem:53} implies $k,k+1\in I_x$. Hence $\bm{\beta} \in \mathcal{A}(\unon)$.

Now, suppose that $\bm{\beta}$ is not a minimizer in Eq. (\ref{eqt:defH}) evaluated at $\unon$. By Lem. \ref{lem:Hprop}(i), there exists a minimizer in Eq. (\ref{eqt:defH}) evaluated at the point $\unon$, which we denote by $(\alpha_1,\dots, \alpha_m)$. Then there holds
\begin{equation}\label{eqt:pfcons_4eqts}
    \begin{cases}
    \sum_{i=1}^m \alpha_i = \sum_{i=1}^m \beta_i = 1,\\
    \sum_{i=1}^m \alpha_i p_i = \sum_{i=1}^m \beta_i p_i = \unon,\\
    \sum_{i=1}^m \alpha_i \gamma_i = \sum_{i=1}^m \beta_i \gamma_i = J^*(\unon),\\
    \sum_{i=1}^m \alpha_i \theta_i < \sum_{i=1}^m \beta_i \theta_i.
    \end{cases}
\end{equation}
Since $\alpha_i \geqslant 0$ for every $i$ and $\beta_i = 0$ for every $i\in \{1,\dots,m\}\setminus\{k,k+1\}$, we have $\alpha_k + \alpha_{k+1} \leqslant 1 = \beta_k + \beta_{k+1}$.
As $\bm{\alpha}\neq \bm{\beta}$, then one or both of the inequalities $\alpha_k < \beta_k$ and $\alpha_{k+1} < \beta_{k+1}$ hold. This leaves three possible cases, and we now show that each case leads to a contradiction.

    Case 1: Let $\alpha_k < \beta_k$ and $\alpha_{k+1} \geqslant \beta_{k+1}$. Define the coefficient $c_i$ by
    \begin{equation*}
        c_i \coloneqq \begin{dcases}
        \frac{\alpha_i - \beta_i}{\beta_k - \alpha_k}, & i\neq k,\\
        0, & i=k.
        \end{dcases}
    \end{equation*}
    The following equations then hold
    \begin{equation*}
        \begin{cases}
        (c_1,\dots, c_m)\in\Delta_m \text{ with } c_k = 0,\\
        \sum_{i\neq k}c_i p_i = p_k,\\
        \sum_{i\neq k}c_i \gamma_i = \gamma_k,\\
        \sum_{i\neq k}c_i \theta_i < \theta_k.
        \end{cases}
    \end{equation*}
    These equations, however, violate assumption (A3), and so we get a contradiction.
    
    Case 2: Let $\alpha_k \geqslant \beta_k$ and $\alpha_{k+1} < \beta_{k+1}$. A similar argument as in case 1 can be applied here by exchanging the indices $k$ and $k+1$ to derive a contradiction.
    
    Case 3: Let $\alpha_k < \beta_k$ and $\alpha_{k+1} < \beta_{k+1}$. From Eq. (\ref{eqt:pfcons_4eqts}), we obtain
    \begin{equation}\label{eqt:pfcons_4eqt_case2}
    \begin{cases}
    \beta_k - \alpha_k
         + \beta_{k+1} - \alpha_{k+1}
        = \sum_{i\neq k,k+1} \alpha_i,\\
    (\beta_k - \alpha_k)
           p_k 
         + (\beta_{k+1} - \alpha_{k+1})
           p_{k+1} 
        = \sum_{i\neq k,k+1} \alpha_i 
           p_i,\\
    (\beta_k - \alpha_k)
           \gamma_k
         + (\beta_{k+1} - \alpha_{k+1})
           \gamma_{k+1}
        = \sum_{i\neq k,k+1} \alpha_i 
           \gamma_i,\\
    (\beta_k - \alpha_k)\theta_k + (\beta_{k+1} - \alpha_{k+1})\theta_{k+1} > \sum_{i\neq k,k+1} \alpha_i \theta_i.
    \end{cases}
    \end{equation}
\def \qk {q_k}
\def \qkp {q_{k+1}}
    Define two numbers $\qk$ and $\qkp$ by
    \begin{equation}\label{eqt:pfcons_defqk}
        \begin{split}
            \qk \coloneqq \frac{\sum_{i<k} \alpha_ip_i}{\sum_{i<k}\alpha_i} \quad \text{ and } \quad
            \qkp \coloneqq \frac{\sum_{i>k+1} \alpha_ip_i}{\sum_{i>k+1}\alpha_i}.
        \end{split}
    \end{equation}
    Note that from the first two equations in (\ref{eqt:pfcons_4eqts}) and the assumption that $\alpha_k < \beta_k$ and $\alpha_{k+1} < \beta_{k+1}$, there exist $i_1<k$ and $i_2>k+1$ such that $\alpha_{i_1}\neq 0$ and $\alpha_{i_2}\neq 0$, and hence the numbers $\qk$ and $\qkp$ are well-defined. By definition, we have $\qk<p_k<p_{k+1}<\qkp$. Therefore, there exist $\ak,\akp\in(0,1)$ such that 
    \begin{equation}\label{eqt:pfcons_pkpkp}
    p_k = \ak \qk + (1-\ak)\qkp \quad \text{ and } \quad p_{k+1} = \akp \qk + (1-\akp)\qkp.     
    \end{equation}
    A straightforward computation yields
    \begin{equation}\label{eqt:pflem52_formak}
        \ak = \frac{\qkp - p_k}{\qkp - \qk} \quad\text{ and }\quad  \akp = \frac{\qkp - p_{k+1}}{\qkp - \qk}.
    \end{equation}
    
    Define the coefficients $\cik$ and $\cikp$ as follows
    \begin{equation}\label{eqt:pflem52_defcik}
    \begin{split}
        \cik \coloneqq \begin{dcases}
        \frac{\ak\alpha_i}{\sum_{\ind<k}\alpha_\ind}, & i<k,\\
        \frac{(1-\ak)\alpha_i}{\sum_{\ind>k+1}\alpha_\ind}, & i>k+1,\\
        0, &\text{otherwise},
        \end{dcases}\quad \text{ and }\quad 
        \cikp \coloneqq \begin{dcases}
        \frac{\akp\alpha_i}{\sum_{\ind<k}\alpha_\ind}, & i<k,\\
        \frac{(1-\akp)\alpha_i}{\sum_{\ind>k+1}\alpha_\ind}, & i>k+1,\\
        0, &\text{otherwise}.
        \end{dcases}
    \end{split}
    \end{equation}
    These coefficients satisfy $\cik, \cikp\in [0,1]$ for any $i$ and $\sum_{i=1}^m \cik= \sum_{i=1}^m \cikp=1$. In other words, we have
    \begin{equation}\label{eqt:pfcons_asmp4_1}
        (\conek,\dots, \cmk)\in\Delta_m \text{ with }\ckk = 0\quad \text{ and }\quad 
        (\conekp,\dots, \cmkp)\in\Delta_m \text{ with }\ckpkp = 0.
    \end{equation}
    Hence, the first equality in Eq. (\ref{eqt:assumption3}) holds for the coefficients $(\conek,\dots, \cmk)$ with the index $k$ and also for the coefficients $(\conekp, \dots, \cmkp)$ with the index $k+1$. We show next that these coefficients satisfy the second and third equalities in (\ref{eqt:assumption3}) and draw a contradiction with assumption (A3).
    
    Using Eqs. (\ref{eqt:pfcons_defqk}), (\ref{eqt:pfcons_pkpkp}), and (\ref{eqt:pflem52_defcik}) to write the formulas for $p_k$ and $p_{k+1}$ via the coefficients $\cik$ and $\cikp$, we find
    \begin{equation}\label{eqt:pfcons_pkpkp2}
    \begin{split}
        &p_k = \ak \frac{\sum_{i<k} \alpha_ip_i}{\sum_{i<k}\alpha_i} + (1-\ak)\frac{\sum_{i>k+1} \alpha_ip_i}{\sum_{i>k+1}\alpha_i}
        = \sum_{i\neq k,k+1}\cik p_i= \sum_{i\neq k}\cik p_i,\\
        &p_{k+1} = \akp \frac{\sum_{i<k} \alpha_ip_i}{\sum_{i<k}\alpha_i} + (1-\akp)\frac{\sum_{i>k+1} \alpha_ip_i}{\sum_{i>k+1}\alpha_i}=\sum_{i\neq k,k+1} \cikp p_i = \sum_{i\neq k+1} \cikp p_i,
        \end{split}
    \end{equation}
    where the last equalities in the two formulas above hold because $c_{k+1}^k=0$ and $c_k^{k+1} = 0$ by definition.
    Hence the second equality in Eq. (\ref{eqt:assumption3}) also holds for both the index $k$ and $k+1$.
    
    From the third equality in Eq. (\ref{eqt:pfcons_4eqt_case2}), assumption (A2), Eq. (\ref{eqt:pfcons_pkpkp2}), and Jensen's inequality, we have
    \begin{equation}\label{eqt:pflem52_cond3}
        \begin{split}
            &\sum_{i\neq k,k+1} \alpha_i\gamma_i
            = (\beta_k - \alpha_k)\gamma_k + (\beta_{k+1} - \alpha_{k+1})\gamma_{k+1}\\
            =\ & (\beta_k - \alpha_k)g(p_k) + (\beta_{k+1} - \alpha_{k+1})g(p_{k+1})\\
            =\ & (\beta_k - \alpha_k)g\left(\sum_{i\neq k,k+1}\cik p_i\right) + (\beta_{k+1} - \alpha_{k+1})g\left(\sum_{i\neq k,k+1}\cikp p_i\right)\\
            \leqslant\ & (\beta_k - \alpha_k)\left(\sum_{i\neq k,k+1}\cik g(p_i)\right) + (\beta_{k+1} - \alpha_{k+1})\left(\sum_{i\neq k,k+1}\cikp g(p_i)\right)\\
            =\ & \sum_{i\neq k,k+1} ((\beta_k - \alpha_k)\cik + (\beta_{k+1} - \alpha_{k+1})\cikp)g(p_i)\\
            =\ & \sum_{i\neq k,k+1} ((\beta_k - \alpha_k)\cik + (\beta_{k+1} - \alpha_{k+1})\cikp)\gamma_i.
        \end{split}
    \end{equation}
    We now compute and simplify the coefficients $(\beta_k - \alpha_k)\cik + (\beta_{k+1} - \alpha_{k+1})\cikp$ in the formula above. First, consider the case when $i<k$. 
    Eqs. (\ref{eqt:pflem52_formak}) and (\ref{eqt:pflem52_defcik}) imply
    \begin{equation*}
        \begin{split}
            &(\beta_k - \alpha_k)\cik + (\beta_{k+1} - \alpha_{k+1})\cikp\\
            =\ & (\beta_k - \alpha_k) \frac{\ak \alpha_i}{\sum_{\ind<k}\alpha_\ind} + (\beta_{k+1} - \alpha_{k+1})\frac{\akp\alpha_i}{\sum_{\ind<k}\alpha_\ind}\\
            =\ &\frac{\alpha_i}{\sum_{\ind<k}\alpha_\ind} ((\beta_k-\alpha_k)\ak + (\beta_{k+1} - \alpha_{k+1})\akp)\\
            =\ & \frac{\alpha_i}{\sum_{\ind<k}\alpha_\ind}\left((\beta_k-\alpha_k)\frac{\qkp - p_k}{\qkp - \qk} + (\beta_{k+1} - \alpha_{k+1})\frac{\qkp - p_{k+1}}{\qkp - \qk}\right)\\
            =\ & \frac{\alpha_i}{\sum_{\ind<k}\alpha_\ind}\cdot \frac{1}{\qkp - \qk} ((\beta_k - \alpha_k + \beta_{k+1} - \alpha_{k+1})\qkp - (\beta_k - \alpha_k) p_k - (\beta_{k+1} - \alpha_{k+1}) p_{k+1}).
        \end{split}
    \end{equation*}
    Applying the first two equalities in Eq. (\ref{eqt:pfcons_4eqt_case2}) and Eq. (\ref{eqt:pfcons_defqk})
    to the last formula above, we obtain
    \begin{equation*}
        \begin{split}
            &(\beta_k - \alpha_k)\cik + (\beta_{k+1} - \alpha_{k+1})\cikp\\
            =\ & \frac{\alpha_i}{\sum_{\ind<k}\alpha_\ind}\cdot \frac{1}{\qkp - \qk} \left(\left(\sum_{i\neq k,k+1}\alpha_i\right)\qkp - \sum_{i\neq k,k+1}\alpha_ip_i\right)\\
            =\ & \frac{\alpha_i}{\sum_{\ind<k}\alpha_\ind}\cdot \frac{1}{\qkp - \qk} \left(\sum_{i\neq k,k+1}\alpha_i\qkp - \sum_{i< k}\alpha_ip_i - \sum_{i> k+1}\alpha_ip_i\right)\\
            =\ & \frac{\alpha_i}{\sum_{\ind<k}\alpha_\ind}\cdot \frac{1}{\qkp - \qk} \left(\sum_{i\neq k,k+1}\alpha_i\qkp - \left(\sum_{i< k}\alpha_i\right)\qk - \left(\sum_{i> k+1}\alpha_i\right)\qkp\right)\\
            =\ & \frac{\alpha_i}{\sum_{\ind<k}\alpha_\ind}\cdot \frac{1}{\qkp - \qk} \left(\sum_{i<k}\alpha_i(\qkp - \qk)\right)\\
            =\ & \alpha_i.
        \end{split}
    \end{equation*}
    The same result for the case when $i>k+1$ also holds and the proof is similar. Therefore, we have
    \begin{equation}\label{eqt:pfcons_coeffs}
        (\beta_k - \alpha_k)\cik + (\beta_{k+1} - \alpha_{k+1})\cikp = \alpha_i \text{ for each } i\neq k, k+1.
    \end{equation}
    Combining Eqs. (\ref{eqt:pflem52_cond3}) and (\ref{eqt:pfcons_coeffs}), we have
    \begin{equation*}
        \begin{split}
            &\sum_{i\neq k,k+1} \alpha_i\gamma_i
            \leqslant \sum_{i\neq k,k+1} ((\beta_k - \alpha_k)\cik + (\beta_{k+1} - \alpha_{k+1})\cikp)\gamma_i
            = \sum_{i\neq k,k+1} \alpha_i\gamma_i.
        \end{split}
    \end{equation*}
    Since the left side and right side are the same, the inequality above becomes equality, which implies that the inequality in Eq. (\ref{eqt:pflem52_cond3}) also becomes equality. 
    In other words, we have
    \begin{equation}\label{eqt:lemB2_eq10_3}
    \begin{split}
        &\gamma_k = g\left(p_k\right)
            =\sum_{i\neq k,k+1}\cik g(p_i) = \sum_{i\neq k,k+1} \cik \gamma_i = \sum_{i\neq k} \cik \gamma_i,\\
        &\gamma_{k+1} = g\left(p_{k+1}\right) = \sum_{i\neq k,k+1}\cikp g(p_i) = \sum_{i\neq k,k+1} \cikp \gamma_i = \sum_{i\neq k+1} \cikp \gamma_i,
    \end{split}
    \end{equation}
    where the last equalities in the two formulas above hold because $c_{k+1}^k=0$ and $c_k^{k+1} = 0$ by definition. Hence the third equality in (\ref{eqt:assumption3}) also holds for both indices $k$ and $k+1$.
    
    In summary, Eqs. (\ref{eqt:pfcons_asmp4_1}), (\ref{eqt:pfcons_pkpkp2}), and (\ref{eqt:lemB2_eq10_3}) imply that Eq. (\ref{eqt:assumption3}) holds for the index $k$ with coefficients $(\conek,\dots, \cmk)$ and also for the index $k+1$ with coefficients $(\conekp, \dots, \cmkp)$. Hence, by assumption (A3), we find
    \begin{equation*}
    \sum_{i\neq k} \cik \theta_i > \theta_k
    \quad
    \text{ and }
    \quad
    \sum_{i\neq k+1} \cikp \theta_i > \theta_{k+1}.
    \end{equation*}
    Using the inequalities above with Eq. (\ref{eqt:pfcons_coeffs}) and the fact that $c_{k+1}^k=0$ and $c_k^{k+1} = 0$, we find
    \begin{equation*}
    \begin{split}
        &(\beta_k - \alpha_k) \theta_k + (\beta_{k+1} - \alpha_{k+1})\theta_{k+1}
        < (\beta_k - \alpha_k) \sum_{i\neq k} \cik \theta_i + (\beta_{k+1} - \alpha_{k+1} )\sum_{i\neq k+1} \cikp \theta_i\\
        =& \sum_{i\neq k,k+1} ((\beta_k - \alpha_k) \cik + (\beta_{k+1} - \alpha_{k+1} )\cikp)\theta_i 
        = \sum_{i\neq k, k+1} \alpha_i\theta_i,
    \end{split}
    \end{equation*}
    which contradicts the last inequality in Eq. (\ref{eqt:pfcons_4eqt_case2}). 

In conclusion, we obtain contradictions in all the three cases. As a consequence, we conclude that $\bm{\beta}$ is a minimizer in Eq. (\ref{eqt:defH}) evaluated at $\unon$ and  Eq. (\ref{eqt:pfcons_Hform}) follows from the definition of $H$ in (\ref{eqt:defH}).
\end{proof}

\subsection{Statement and proof of Lemma \ref{lem:51}} \label{sec:prooflem53}
\begin{lem}\label{lem:51}
Consider the one-dimensional case, i.e., $n=1$. Let $p_1,\dots,p_m\in\R$ satisfy $p_1<\dots<p_m$. Suppose assumptions (A1)-(A2) hold. Let $x\in\R$ and $t>0$. Assume $j,k,l$ are three indices such that $1\leqslant j\leqslant k<l\leqslant m$ and
\begin{equation} \label{eqt:lem51_assump}
    j, l \in \argmax_{i\in\{1,\dots,m\}}\{xp_{i} -t\theta_{i}-\gamma_{i}\}.
\end{equation}
Then there holds
\begin{equation}\label{eqt:lem51_conclusion}
    \frac{\theta_l - \theta_k}{p_l - p_k} \leqslant \frac{\theta_l - \theta_j}{p_l - p_j}.
\end{equation}
\end{lem}
\begin{proof}
Note that Eq. (\ref{eqt:lem51_conclusion}) holds trivially when $j=k$, so we only need to consider the case when $j<k<l$. On the one hand, Eq. (\ref{eqt:lem51_assump}) implies
\begin{equation*}
    xp_j - t\theta_j - \gamma_j = xp_l - t\theta_l - \gamma_l \geqslant xp_k - t\theta_k - \gamma_k,
\end{equation*}
which yields
\begin{equation} \label{eqt:pflem51_gamma}
    \begin{split}
        &\gamma_l - \gamma_k \leqslant x(p_l - p_k) - t(\theta_l - \theta_k),\\
        &\gamma_l - \gamma_j = x(p_l - p_j) - t(\theta_l - \theta_j).\\
    \end{split}
\end{equation}
On the other hand, for each $i\in \{j,j+1,\dots, l-1\}$ let $q_i\in (p_i, p_{i+1})$ and $x_i \in \partial J^*(q_i)$. Such $x_i$ exists because $q_i \in \inter{\dom J^*}$, so that the subdifferential $\partial J^*(q_i)$ is non-empty.
Then $q_i\in \partial J(x_i)$ and Lem. \ref{lem:53} imply
\begin{equation*}
    x_ip_i - \gamma_i = x_ip_{i+1} - \gamma_{i+1} = \max_{\ind\in \{1,\dots,m\}} \{x_i p_\ind - \gamma_\ind\}.
\end{equation*}
A straightforward computation yields
\begin{equation*}
    \begin{split}
        &\gamma_l - \gamma_k = \sum_{i=k}^{l-1} (\gamma_{i+1} - \gamma_i) = \sum_{i=k}^{l-1} x_i(p_{i+1} - p_i),\\
        &\gamma_l - \gamma_j = \sum_{i=j}^{l-1} (\gamma_{i+1} - \gamma_i) = \sum_{i=j}^{l-1} x_i(p_{i+1} - p_i).
    \end{split}
\end{equation*}
Combining the two equalities above with Eq. (\ref{eqt:pflem51_gamma}), we conclude that
\begin{equation*}
    \begin{split}
        &x(p_l - p_k) - t(\theta_l - \theta_k)\geqslant \sum_{i=k}^{l-1} x_i(p_{i+1} - p_i),\\
        &x(p_l - p_j) - t(\theta_l - \theta_j)= \sum_{i=j}^{l-1} x_i(p_{i+1} - p_i).
    \end{split}
\end{equation*}
Now, divide the inequality above by $t(p_l - p_k) > 0$ (because by assumption $t>0$ and $l>k$, which implies that $p_l>p_k$), divide the equality above by $t(p_l - p_j) > 0$ (because $l>j$, which implies that $t(p_l - p_j) \neq 0$), and rearrange the terms to obtain
\begin{equation}\label{eqt:pflem51_ratio}
\begin{split}
    &\frac{\theta_l - \theta_k}{p_l - p_k} \leqslant \frac{x}{t} - \frac{1}{t} \frac{\sum_{i=k}^{l-1} x_i(p_{i+1} - p_i)}{p_l - p_k},\\
    &\frac{\theta_l - \theta_j}{p_l - p_j} = \frac{x}{t} - \frac{1}{t} \frac{\sum_{i=j}^{l-1} x_i(p_{i+1} - p_i)}{p_l - p_j}.
\end{split}
\end{equation}
Recall that $q_j < q_{j+1} < \dots < q_{l-1}$ and $x_i \in \partial J^*(q_i)$ for any $j\leqslant i<l$. Since the function $J^*$ is convex, the subdifferential operator $\partial J^*$ is a monotone non-decreasing operator \cite[Def. IV.4.1.3, and Prop. VI.6.1.1]{hiriart2013convexI}, which yields $x_j\leqslant x_{j+1}\leqslant \dots\leqslant x_{l-1}$. Using that $p_1 < p_2 < \dots < p_m$ and $j<k<l$, we obtain
\begin{equation} \label{eqt:lemB2_ineq_sum}
    \frac{\sum_{i=k}^{l-1} x_i(p_{i+1} - p_i)}{p_l - p_k}
    \geqslant \frac{\sum_{i=k}^{l-1} x_k(p_{i+1} - p_i)}{p_l - p_k}
    = x_k
    = \frac{\sum_{i=j}^{k-1} x_k(p_{i+1} - p_i)}{p_k - p_j}
    \geqslant \frac{\sum_{i=j}^{k-1} x_i(p_{i+1} - p_i)}{p_k - p_j}.
\end{equation}
To proceed, we now use that fact that if four real numbers $a,c\in \R$ and $b,d>0$ satisfy $\frac{a}{b}\geqslant \frac{c}{d}$, then $\frac{a}{b}\geqslant \frac{a+c}{b+d}$. Combining this fact with inequality (\ref{eqt:lemB2_ineq_sum}), we find
\begin{equation*}
    \frac{\sum_{i=k}^{l-1} x_i(p_{i+1} - p_i)}{p_l - p_k}
    \geqslant \frac{\sum_{i=k}^{l-1} x_i(p_{i+1} - p_i) + \sum_{i=j}^{k-1} x_i(p_{i+1} - p_i)}{p_l - p_k + p_k - p_j}
    = \frac{\sum_{i=j}^{l-1} x_i(p_{i+1} - p_i)}{p_l - p_j}.
\end{equation*}
We combine the inequality above with  (\ref{eqt:pflem51_ratio})
to obtain 
\begin{equation*}
    \frac{\theta_l - \theta_k}{p_l - p_k} \leqslant \frac{\theta_l - \theta_j}{p_l - p_j}.
\end{equation*}
which concludes the proof.
\end{proof}

\subsection{Proof of Proposition \ref{thm:conservation}}
\label{subsec:pf_thmconservation}
Proof of (i): First, note that $u$ is piecewise constant. Second, recall that $J$ is defined as the pointwise maximum of a finite number of affine functions. Therefore, the initial data $u(\cdot,0) = \nabla J(\cdot)$ (recall that here, the gradient $\nabla$ is taken in the sense of distribution) is bounded and of locally bounded variation (see \cite[Chap. 5, page 167]{evans2015measure} for the definition of locally bounded variation). Finally, the flux function $H$, defined in Eq. (\ref{eqt:defH}), is Lipschitz continuous in $\dom J^*$ by Lem. \ref{lem:52}. It can therefore be extended to $\R$ while preserving its Lipschitz property \cite[Thm. 4.16]{folland2013real}. Therefore, we can invoke \cite[Prop. 2.1]{DAFERMOS197233} to conclude that $u$ is the entropy solution to the conservation law (\ref{eqt:conservation_1D}) provided it satisfies the two following conditions. Let $\bar{x}(t)$ be any smooth line of discontinuity of $u$. Fix $t>0$ and define $\um$ and $\up$ as
\begin{equation}\label{eqt:pfcons_defumup}
    \um \coloneqq \lim_{x\to\bar{x}(t)^-} u(x,t),\quad \text{ and } \quad \up \coloneqq \lim_{x\to\bar{x}(t)^+} u(x,t).
\end{equation}
Then the two conditions are:
\begin{itemize}
    \item[1.] The curve $\bar{x}(t)$ is a straight line with the slope 
    \begin{equation} \label{eqt:pfcons_cond1}
        \frac{d\bar{x}}{dt} = \frac{H(\up)-H(\um)}{\up - \um}.
    \end{equation}
    \item[2.] For any $\unon$ between $\up$ and $\um$, we have 
    \begin{equation}\label{eqt:pfcons_cond2}
    \frac{H(\up) - H(\unon)}{\up - \unon} \leqslant \frac{H(\up) - H(\um)}{\up - \um}.
    \end{equation}
\end{itemize}

\def \xm {x^-}
\def \xp {x^+}
First, we prove the first condition and Eq. (\ref{eqt:pfcons_cond1}). According to the definition of $u$ in Eq. (\ref{eqt:nn_conservation}), the range of $u$ is the compact set $\{p_1,\dots, p_m\}$. As a result, $\um$ and $\up$ are in the range of $u$, i.e., there exist indices $j$ and $l$ such that 
\begin{equation}\label{eqt:pfcons_defjl}
    \um = p_j,\quad \text{ and }\quad \up = p_l.
\end{equation}
Let $(\bar{x}(s),s)$ be a point on the curve $\bar{x}$ which is not one of the endpoints. Since $u$ is piecewise constant, there exists a neighborhood $\Ncal$ of $(\bar{x}(s),s)$ such that for any $(\xm,t), (\xp,t)\in \Ncal$ satisfying $\xm < \bar{x}(t) < \xp$, we have $u(\xm,t) = \um = p_j$ and $u(\xp,t) = \up = p_l$.
In other words, if $\xm, \xp, t$ are chosen as above, according to the definition of $u$ in Eq. (\ref{eqt:nn_conservation}), we have
\begin{equation}\label{eqt:pf_thmconserv_xminxplus}
    j\in \argmax_{i\in\{1,\dots,m\}}\{\xm p_{i} -t\theta_{i}-\gamma_{i}\} \quad \text{ and }\quad 
    l\in \argmax_{i\in\{1,\dots,m\}}\{\xp p_{i} -t\theta_{i}-\gamma_{i}\}.
\end{equation}
Define a sequence $\{\xm_k\}_{k=1}^{+\infty} \subset (-\infty,\bar{x}(s))$ such that $(\xm_k, s)\in\Ncal$ for any $k\in\N$ and $\lim_{k\to +\infty}\xm_k = \bar{x}(s)$. By Eq. (\ref{eqt:pf_thmconserv_xminxplus}), we have
\begin{equation*}
    \xm_k p_{j} -s\theta_{j}-\gamma_{j} \geq \xm_k p_{i} -s\theta_{i}-\gamma_{i} \text{ for any }i\in\{1,\dots,m\}.
\end{equation*}
When $k$ approaches infinity, the above inequality implies
\begin{equation*}
    \bar{x}(s) p_{j} -s\theta_{j}-\gamma_{j} \geq \bar{x}(s) p_{i} -s\theta_{i}-\gamma_{i} \text{ for any }i\in\{1,\dots,m\}.
\end{equation*}
In other words, we have 
\begin{equation}\label{eqt:pf_thmconserv_eq01}
    j\in \argmax_{i\in\{1,\dots,m\}}\{\bar{x}(s)p_{i} -s\theta_{i}-\gamma_{i}\}.
\end{equation}
Similarly, define a sequence $\{\xp_k\}_{k=1}^{+\infty} \subset (\bar{x}(s), +\infty)$ such that $(\xp_k, s)\in\Ncal$ for any $k\in\N$ and $\lim_{k\to +\infty}\xp_k = \bar{x}(s)$. Using a similar argument as above, we can conclude that 
\begin{equation}\label{eqt:pf_thmconserv_eq02}
    l\in \argmax_{i\in\{1,\dots,m\}}\{\bar{x}(s)p_{i} -s\theta_{i}-\gamma_{i}\}.
\end{equation}
By a continuity argument, Eqs. (\ref{eqt:pf_thmconserv_eq01}) and (\ref{eqt:pf_thmconserv_eq02}) also hold for the end points of $\bar{x}$.
In conclusion, for any $(\bar{x}(t), t)$ on the curve $\bar{x}$, we have
\begin{equation}\label{eqt:pfcons_jlmaximizer}
    j, l\in \argmax_{i\in\{1,\dots,m\}}\{\bar{x}(t)p_{i} -t\theta_{i}-\gamma_{i}\},
\end{equation}
which implies that
\begin{equation*}
    \bar{x}(t)p_{l} -t\theta_{l}-\gamma_{l} = 
    \bar{x}(t)p_{j} -t\theta_{j}-\gamma_{j}.
\end{equation*}
Therefore, the curve $\bar{x}(t)$ lies on the straight line 
\begin{equation*}
    x(p_l - p_j) - t(\theta_l - \theta_j) -(\gamma_l - \gamma_j) = 0
\end{equation*}
and Eq. (\ref{eqt:pfcons_defjl}) and Lem. \ref{lem:Hprop}(iii) imply that its slope equals
\begin{equation*}
    \frac{d\bar{x}}{dt} = \frac{\theta_l - \theta_j}{p_l - p_j} = \frac{H(\up) - H(\um)}{\up - \um}.
\end{equation*}
This proves Eq. (\ref{eqt:pfcons_cond1}) and the first condition holds.

It remains to show the second condition. Since $u$ equals $\nabla_x f$ and $f$ is convex by Thm. \ref{thm:constructHJ}, its corresponding subdifferential operator $u$ is monotone non-decreasing with respect to $x$ \cite[Def. IV.4.1.3 and Prop. VI.6.1.1]{hiriart2013convexI}. As a result, $\um < \up$ and $\unon\in (\um, \up)$, where we still adopt the notation $\um = p_j$ and $\up = p_l$. Recall that Lem. \ref{lem:Hprop}(iii) implies $H(p_i) = \theta_i$ for any $i$. Then, Eq. (\ref{eqt:pfcons_cond2}) in the second condition becomes
\begin{equation}\label{eqt:pfcons_cond2_2}
     \frac{\theta_l - H(\unon)}{p_l - \unon} \leqslant \frac{\theta_l - \theta_j}{p_l - p_j}.
\end{equation}
Without loss of generality, we may assume that $p_1< p_2< \dots < p_m$. Then the fact $p_j=\um < \up = p_l$ implies $j<l$. We consider the following two cases.

First, if there exists some $k$ such that $\unon = p_k$, then $H(\unon) = \theta_k$ by Lem. \ref{lem:Hprop}(iii).
Since $\um < \unon < \up$, we have $j < k < l$. Recall that Eq. (\ref{eqt:pfcons_jlmaximizer}) holds. Therefore the assumptions of Lem. \ref{lem:51} are satisfied, which implies Eq. (\ref{eqt:pfcons_cond2_2}) holds.

Second, suppose $\unon \neq p_i$ for every $i\in\{1,\dots,m\}$. 
Then there exists some $k\in \{j, j+1,\dots, l-1\}$ such that $p_k< \unon < p_{k+1}$. 
Lem. \ref{lem:52} then implies that Eqs. (\ref{eqt:pfcons_Hform}) and (\ref{eqt:lem52_defbetak}) hold, that is,
\begin{equation*}
    H(\unon) = \beta_k \theta_k + \beta_{k+1}\theta_{k+1}, \quad
    \unon = \beta_k p_k + \beta_{k+1} p_{k+1}, \quad \text{ and }\quad
    \beta_k + \beta_{k+1} = 1.
\end{equation*}
Using these three equations, we can write the left hand side of Eq. (\ref{eqt:pfcons_cond2_2}) as 
\begin{equation}\label{eqt:pfcons_cond2left}
    \frac{\theta_l - H(\unon)}{p_l - \unon}
    = \frac{\theta_l - \beta_k \theta_k - \beta_{k+1} \theta_{k+1}}{p_l - \beta_kp_k - \beta_{k+1}p_{k+1}}
    = \frac{\beta_k (\theta_l - \theta_k) + \beta_{k+1} (\theta_l-\theta_{k+1})}{\beta_k(p_l - p_k) + \beta_{k+1}(p_l - p_{k+1})}.
\end{equation}
If $k+1 = l$, then this equation become
\begin{equation*}
    \frac{\theta_l - H(\unon)}{p_l - \unon}
    = \frac{\theta_l - \theta_k}{p_l - p_k}.
\end{equation*}
Since $j\leqslant k<l$ and Eq. (\ref{eqt:pfcons_jlmaximizer}) holds, then the assumptions of Lem. \ref{lem:51} are satisfied. This allows us to conclude that Eq. (\ref{eqt:pfcons_cond2_2}) holds.
\\
If $k+1\neq l$, 
then using Eq. (\ref{eqt:pfcons_jlmaximizer}), the inequalities  $j\leqslant k<k+1< l$, and Lem. \ref{lem:51}, we obtain
\begin{equation*}
    \frac{\beta_k(\theta_l - \theta_k)}{\beta_k(p_l - p_k)} = \frac{\theta_l - \theta_k}{p_l - p_k} \leqslant \frac{\theta_l - \theta_j}{p_l - p_j} \quad \text{ and }\quad
    \frac{\beta_{k+1}(\theta_l - \theta_{k+1})}{\beta_{k+1}(p_l - p_{k+1})} = \frac{\theta_l - \theta_{k+1}}{p_l - p_{k+1}} \leqslant \frac{\theta_l - \theta_j}{p_l - p_j}.
\end{equation*}
Note that if $a_i \in \R$ and $b_i \in (0,+\infty)$ for $i\in\{1,2,3\}$ satisfy $\frac{a_1}{b_1} \leqslant \frac{a_3}{b_3}$ and $\frac{a_2}{b_2} \leqslant \frac{a_3}{b_3}$, then $\frac{a_1+a_2}{b_1+b_2}\leqslant \frac{a_3}{b_3}$. Then, since $\beta_k(p_l - p_k)$, $\beta_{k+1}(p_l - p_{k+1})$ and $p_l-p_j$ are positive, we have
\begin{equation*}
    \frac{\beta_k (\theta_l - \theta_k) + \beta_{k+1} (\theta_l-\theta_{k+1})}{\beta_k(p_l - p_k) + \beta_{k+1}(p_l - p_{k+1})}\leqslant \frac{\theta_l - \theta_j}{p_l - p_j}.
\end{equation*}
Hence Eq. (\ref{eqt:pfcons_cond2_2}) follows directly from the inequality above and Eq. (\ref{eqt:pfcons_cond2left}).

Therefore, the two conditions, including Eqs. (\ref{eqt:pfcons_cond1}) and (\ref{eqt:pfcons_cond2}), are satisfied and we apply \cite[Prop 2.1]{DAFERMOS197233} to conclude that the function $u$ is the entropy solution to the conservation law (\ref{eqt:conservation_1D}).

Proof of (ii) (sufficiency): Without loss of generality, assume $p_1<p_2<\dots<p_m$. Let $C\in \R$. Suppose $\tilde{H}$ satisfies $\tilde{H}(p_i)=H(p_i)+C$ for each $i\in\{1,\dots,m\}$ and $\tilde{H}(p)\geqslant H(p)+C$ for any $p\in[p_1,p_m]$. We want to prove that $u$ is the entropy solution to the conservation law (\ref{eqt:conservation_Htilde}).

As in the proof of (i), we apply \cite[Prop 2.1]{DAFERMOS197233} and verify that the two conditions hold through Eqs. (\ref{eqt:pfcons_cond1}) and (\ref{eqt:pfcons_cond2}). Let $\bar{x}(t)$ be any smooth line of discontinuity of $u$, define $\um$ and $\up$ by Eq. (\ref{eqt:pfcons_defumup}) (and recall that $\um = p_j$ and $\up = p_l$), and let $\unon\in (\um,\up)$. We proved in the proof of (i) that $\bar{x}(t)$ is a straight line, and so it suffices to prove that
\begin{equation}\label{eqt:pfcons_condii}
    \frac{d\bar{x}}{dt} = \frac{\tilde{H}(\up)-\tilde{H}(\um)}{\up - \um}, \quad \text{ and }\quad 
    \frac{\tilde{H}(\up) - \tilde{H}(\unon)}{\up - \unon} \leqslant \frac{\tilde{H}(\up) - \tilde{H}(\um)}{\up - \um}.
\end{equation}

We start with proving the equality in Eq. (\ref{eqt:pfcons_condii}). By assumption, there holds
\begin{equation}\label{eqt:pfcons_equal_H_Htilde}
\tilde{H}(\um) = \tilde{H}(p_j) = H(p_j)+C=H(\um)+C\quad \text{ and }\quad \tilde{H}(\up) = \tilde{H}(p_l) = H(p_l)+C=H(\up)+C.
\end{equation}
We combine Eq. (\ref{eqt:pfcons_equal_H_Htilde}) with Eq. (\ref{eqt:pfcons_cond1}), (which we proved in the proof of (i)), we obtain
\begin{equation*}
    \frac{d\bar{x}}{dt} = \frac{H(\up)-H(\um)}{\up - \um} = \frac{H(\up)+C-(H(\um)+C)}{\up - \um} = \frac{\tilde{H}(\up)-\tilde{H}(\um)}{\up - \um}.
\end{equation*}
Therefore, the equality in (\ref{eqt:pfcons_condii}) holds.

Next, we prove the inequality in Eq. (\ref{eqt:pfcons_condii}). Since $\unon\in (\um,\up)\subseteq [p_1,p_m]$, by assumption there holds $\tilde{H}(\unon) \geqslant H(\unon) +C$. Taken together with Eqs. (\ref{eqt:pfcons_cond2}) and (\ref{eqt:pfcons_equal_H_Htilde}), we get
\begin{equation*}
    \frac{\tilde{H}(\up) - \tilde{H}(\unon)}{\up - \unon} \leqslant \frac{H(\up)+C - (H(\unon)+C)}{\up - \unon} \leqslant \frac{H(\up) - H(\um)}{\up - \um} = \frac{\tilde{H}(\up) - \tilde{H}(\um)}{\up - \um},
\end{equation*}
which shows that the inequality in Eq. (\ref{eqt:pfcons_condii}) holds. 

Hence, we can invoke \cite[Prop 2.1]{DAFERMOS197233} to conclude that $u$ is the entropy solution to the conservation law (\ref{eqt:conservation_Htilde}).

Proof of (ii) (necessity): Suppose that $u$ is the entropy solution to the conservation law (\ref{eqt:conservation_Htilde}). We prove that there exists $C\in\R$ such that $\tilde{H}(p_i)=H(p_i)+C$ for any $i$ and $\tilde{H}(p)\geqslant H(p)+C$ for any $p\in[p_1,p_m]$. 

By Lem. \ref{lem:existdiffpt}, for each $i\in\{1,\dots,m\}$ there exist $x\in\R$ and $t>0$ such that 
\begin{equation}\label{eqt:pfcons_utakepi}
f(\cdot,t) \text{ is differentiable at }x, \text{ and }\nabla_x f(x,t)=p_i.    
\end{equation}
Moreover, the proof of Lem. \ref{lem:existdiffpt} implies there exists $T>0$ such that for any $0<t<T$, there exists $x\in\R$ such that Eq. (\ref{eqt:pfcons_utakepi}) holds. As a result, there exists $t>0$ such that for each $i\in\{1,\dots,m\}$, there exists $x_i\in\R$ satisfying Eq. (\ref{eqt:pfcons_utakepi}) at the point $(x_i,t)$, which implies $u(x_i,t) = p_i$. Note that $p_i\neq p_j$ implies that $x_i \neq x_j$. (Indeed, if $x_i = x_j$, then $p_i = \nabla_x f(x_i,t) = \nabla_x f(x_j,t) = p_j$ which gives a contradiction since $p_i \neq p_j$ by assumption (A1).) As mentioned before, the function $u(\cdot,t) \equiv \nabla_{x}f$ is a monotone non-decreasing operator and $p_i$ is increasing with respect to $i$, and therefore $x_1<x_2<\dots<x_m$. Since $u$ is piecewise constant, for each $k\in\{1,\dots,m-1\}$ there exists a curve of discontinuity of $u$ with $u = p_k$ on the left hand side of the curve and $u=p_{k+1}$ on the right hand side of the curve. Let $\bar{x}(s)$ be such a curve and let $\um$ and $\up$ be the corresponding numbers defined in Eq. (\ref{eqt:pfcons_defumup}). The argument above proves that we have $\um = p_k$ and $\up = p_{k+1}$. 

Since $u$ is the piecewise constant entropy solution, we invoke \cite[Prop 2.1]{DAFERMOS197233} to conclude that the two aforementioned conditions hold for the curve $\bar{x}(s)$, i.e., (\ref{eqt:pfcons_condii}) holds with $\um = p_k$ and $\up = p_{k+1}$. From the equality in (\ref{eqt:pfcons_condii}) and Eq. (\ref{eqt:pfcons_cond1}) proved in (i), we deduce
\begin{equation*}
    \frac{\tilde{H}(p_{k+1})-\tilde{H}(p_k)}{p_{k+1} - p_k} = \frac{\tilde{H}(\up)-\tilde{H}(\um)}{\up - \um} = \frac{d\bar{x}}{dt} = \frac{H(\up)-H(\um)}{\up - \um} = \frac{H(p_{k+1})-H(p_k)}{p_{k+1} - p_k}.
\end{equation*}
Since $k$ is an arbitrary index, the equality above implies that $\tilde{H}(p_{k+1})-\tilde{H}(p_k) = H(p_{k+1})-H(p_k)$ holds for any $k\in\{1,\dots,m-1\}$.
Therefore, there exists $C\in\R$ such that 
\begin{equation}\label{eqt:pfcons_ii_Hpi}
\tilde{H}(p_k) = H(p_k)+C \text{ for any }k\in \{1,\dots,m\}.    
\end{equation}

It remains to prove $\tilde{H}(\unon)\geqslant H(\unon)+C$ for all $\unon\in[p_k,p_{k+1}]$. If this inequality holds, then the statement follows because $k$ is an arbitrary index. We already proved that
$\tilde{H}(\unon)\geqslant H(\unon)+C$ for $u_0=p_k$ with $k\in\{1,\dots,m\}$. Therefore, we need to prove that $\tilde{H}(\unon)\geqslant H(\unon)+C$ for all $\unon\in(p_k,p_{k+1})$. Let $\unon \in (p_k,p_{k+1})$. By Eq. (\ref{eqt:pfcons_ii_Hpi}) and the inequality in (\ref{eqt:pfcons_condii}), we have
\begin{equation}\label{eqt:pfcons_ii_Hunon_ineqt}
    \frac{H(p_{k+1})+C - \tilde{H}(\unon)}{p_{k+1} - \unon} = \frac{\tilde{H}(\up) - \tilde{H}(\unon)}{\up - \unon} \leqslant \frac{\tilde{H}(\up) - \tilde{H}(\um)}{\up - \um} = \frac{H(p_{k+1}) - H(p_k)}{p_{k+1} - p_k}.
\end{equation}
By Lem. \ref{lem:52} and a straightforward computation, we also have
\begin{equation}\label{eqt:pfcons_ii_Hunon}
    \frac{H(p_{k+1}) - H(\unon)}{p_{k+1} - \unon} = \frac{H(p_{k+1}) - H(p_k)}{p_{k+1} - p_k}.
\end{equation}
Comparing Eqs. (\ref{eqt:pfcons_ii_Hunon_ineqt}) and (\ref{eqt:pfcons_ii_Hunon}), we obtain $\tilde{H}(\unon)\geqslant H(\unon)+C$. Since $k$ is arbitrary, we conclude that $\tilde{H}(\unon)\geqslant H(\unon)+C$ holds for all $\unon\in[p_1,p_m]$ and the proof is complete.

{\bf Competing interests}
 On behalf of all authors, the corresponding author (J. Darbon) states that there is no conflict of interest.

\bibliographystyle{spmpsci}

\end{document}